\documentclass[11pt,a4paper,final]{amsart}
\usepackage[english]{babel}
\usepackage[utf8x]{inputenc}
\usepackage{hyperref}
\usepackage{amsmath}
\usepackage{amssymb}
\usepackage{amsfonts}
\usepackage{amsthm}
\usepackage{amscd}
\usepackage{mathdots}
\usepackage{tikz}
\usepackage{enumerate}
\usepackage{pgfplots}
\usetikzlibrary{positioning}
\pgfplotsset{compat=1.8}
\usepackage{subfig}
\usepackage{bbold}

\usepackage[noblocks]{authblk}
\usepackage{verbatim}
\makeatletter
\def\revddots{\mathinner{\mkern1mu\raise\p@
\vbox{\kern7\p@\hbox{.}}\mkern2mu
\raise4\p@\hbox{.}\mkern1mu\raise7\p@\hbox{.}\mkern1mu}}

\g@addto@macro{\endabstract}{\@setabstract}
\newcommand{\authorfootnotes}{\renewcommand\thefootnote{\@fnsymbol\c@footnote}}%

\makeatother

\theoremstyle{plain}
\newtheorem{thm}{Theorem}[section]
\newtheorem{lem}[thm]{Lemma}
\newtheorem{prop}[thm]{Proposition}
\newtheorem{cor}[thm]{Corollary}

\theoremstyle{definition}

\newtheorem{exmp}[thm]{Example}

\theoremstyle{remark}
\newtheorem*{rem}{Remark}

\newtheorem*{claim}{Claim}

\newcommand{\R}{\mathbb{R}}

\newcommand{\C}{\mathbb{C}}
\newcommand{\N}{\mathbb{N}}
\newcommand{\Z}{\mathbb{Z}}

\newcommand{\Hom}{\mathrm{Hom}}

\newcommand{\deri}{\mathrm{d}}

\title[Generic representations of metaplectic groups]{Generic representations of metaplectic groups and their theta lifts} 
\author{Petar Baki\'c and Marcela Hanzer}
\address{Department of Mathematics, University of Zagreb, Croatia}
\email{pbakic@math.hr, hanmar@math.hr}


\date{}

\begin{document}

\maketitle

\begin{center}
\textsc{Petar Baki\'{c} and Marcela Hanzer}
\end{center}
%
%
%

\begin{abstract}In this paper we give the description of generic representations of metaplectic groups over p-adic fields in terms of their Langlands parameters and calculate their theta lifts on all levels for any tower of odd orthogonal groups.
We also describe precisely all the occurrences of the failure of the standard module conjecture for metaplectic groups.
\end{abstract}
\section{Introduction}
In this paper, we investigate two important problems concerning the representation theory of the metaplectic group, defined over a non-Archimedean local field $F$ of characteristic zero. To be precise, let $W$ be a symplectic vector space of (even) dimension $2n$ over $F$. Let $Sp(2n,F)$ denote the corresponding symplectic group, i.e., the group of isometries of $W$. The metaplectic group $\widetilde{Sp(2n,F)}$ is the unique non-trivial two-fold central extension of $Sp(2n,F)$:
\[
1 \to \{\pm 1\} \to \widetilde{Sp(2n,F)} \to Sp(2n,F) \to 1.
\]
The purpose of this paper is to study the generic representations of $\widetilde{Sp(2n,F)}$. To define them, we fix a non-trivial additive character $\psi$ of $F$. After fixing a Borel subgroup $B(F)$ of $Sp(2n,F)$, we let $U(F)$ denote the unipotent radical of $B(F)$ and $\widetilde{U(F)}$ its preimage in $\widetilde{Sp(2n,F)}$. The character $\psi$ can be extended to a non-degenerate character of $U(F)$ (also denoted $\psi$) and further to $\widetilde{U(F)}$, which is isomorphic to $U(F) \times \{\pm 1\}$, by $\psi(z,\epsilon) = \epsilon\psi(z)$. We say that a smooth representation of $(\pi,V_\pi)$ of $\widetilde{Sp(2n,F)}$ is $\psi$-generic if there is a non-zero linear functional $\lambda$ on $V_\pi$ such that
\[
\lambda(\pi(z)v)=\psi(z)\lambda(v),\forall v\in V_{\pi},z\in \widetilde{U(F)}.
\]
Our goal is to answer some basic questions regarding the generic representations of the metaplectic group. More specifically, we 
\begin{enumerate}[(i)]
\item obtain a complete description of the irreducible generic representations of $\widetilde{Sp(2n,F)}$ in terms of their Langlands parameters 
\item compute the theta lifts of all the irreducible generic representations to any orthogonal tower on all levels.
\end{enumerate}

In the case of quasi-split classical groups, the first question is answered by the standard module conjecture (\cite{Muic_proof_Casselman_Shahidi}, \cite{Hanzer_injectivity}): the standard module of an irreducible generic representation is itself irreducible. Thus, the set of (equivalence classes of) generic irreducible representation corresponds to the set of those generic standard representations which are irreducible.

In case of the metaplectic group, the situation is more complicated because the standard module conjecture no longer holds---see Proposition \ref{cuspidal_reducibility} and Example \ref{primjer}. Our first main result provides explicit necessary and sufficient conditions for the Langlands quotient of a standard representation to be generic (cf. Theorem \ref{generic_reps_metaplectic}).

\begin{thm}
\label{tm1.1}
Let  $\delta_i,\;i=1,2,\ldots,k$ be irreducible unitarizable square integrable representations of $GL(n_i,F),$ $\sigma$ a tempered $\psi$--generic representation  of $\widetilde{Sp(2n_0,F)}$ and $\overrightarrow{s_0}=(s_1,s_2,\ldots,s_k)$ satisfy $s_1\ge  s_2\ge \cdots\ge s_k>0$ so that for $\overrightarrow{s}=\overrightarrow{s_0}$ the representation
\[\gamma_{\psi}^{-1}\delta_1\nu^{s_1}\times \gamma_{\psi}^{-1}\delta_2\nu^{s_2}\times \dotsm \times \gamma_{\psi}^{-1}\delta_k\nu^{s_k}\rtimes \sigma\]
 is standard. Then, the Langlands quotient $L(\gamma_{\psi}^{-1}\delta_1\nu^{s_1},dotsc,\gamma_{\psi}^{-1}\delta_k\nu^{s_k};\sigma)$ is $\psi$--generic if and only if for each $i,j\in\{1,2,\ldots,k\},\;i\neq j$
 \begin{equation}
\label{GL_conditions0}
 \delta_i\nu^{s_i}\times \delta_j\nu^{s_j} \text{ and } \delta_i\nu^{s_i}\times \widetilde{\delta_j}\nu^{-s_j}
 \end{equation}
 is irreducible and the following holds:
 \begin{itemize}
\item[(i)] if, for    $i\in\{1,2,\ldots,k\},$ $\delta_i\nu^{s_i}$ is not equal to $\delta([\nu^{1/2},\nu^{m+1/2}])$ for some $m\in \N_0,$ (actually for $2s_i-1$) then $\gamma_{\psi}^{-1}\delta_i\nu^{s_i}\rtimes \sigma$ is irreducible.
\item [(ii)] if, for  $i\in\{1,2,\ldots,k\},$ $\delta_i\nu^{s_i}=\delta([\nu^{1/2},\nu^{m+1/2}])$ for some $m\in \N_0,$ then, for each even $a$ such that $1\otimes S_a$ occurs in the L-parameter of $\sigma,$ we have $a\ge 4s_i.$
\end{itemize}
Note that because of \eqref{GL_conditions0}, there can be at most one occurrence of a representation of the form  $\delta([\nu^{1/2},\nu^{m+1/2}])$ among $\delta_1\nu^{s_1},\ldots,\delta_k\nu^{s_k}.$
\end{thm}
The main ingredient used in the proof of Theorem \ref{tm1.1} is the analysis of the local coefficients $C_\psi^{\widetilde{Sp(2n,F)}}(P,\overrightarrow{s}, (\otimes_{i=1}^r \tau_i)\otimes \sigma, w)$ as introduced by Shahidi (\cite{Shahidi_certain}, \cite{Shahidi_multiplicativity}, \cite{Sh2}) and adapted to the setting of metaplectic groups by Szpruch in \cite{Szpruch_PhD}. By analyzing the holomorphy of the appropriate local coefficient we are able to determine first the cuspidal reduciblities (Proposition \ref{cuspidal_reducibility}), and later on the conditions which appear in Theorem \ref{tm1.1}.

In the final stage of writing this paper we learned about the results of Atobe in \cite{Atobe_gen} concerning generic representations of metaplectic groups; specifically, we refer to Theorem 3.11 of \cite{Atobe_gen}. There, the author uses the results of Szpruch on the local coefficients for generic representations of metaplectic groups (\cite{Szpruch_PhD}), in the same way we do, to give a criterion of genericity of general Langlands quotient of a metaplectic group. This criterion of Atobe is given solely in terms of the L-function attached to the Langlands parameter of this Langlands quotient. In contrast, in the above theorem, we expressed this criterion in terms of the inducing representation of the standard Levi subgroup of metaplectic group which arises in the description of this Langlands quotient. This way of expressing this result is particularly well-suited for our application in determination of all  the theta lifts of our generic Langlands quotient.

Moreover, with our approach, we were able to provide the complete answer to the question of reducibility of generic standard representations with generic Langlands quotients (Theorem \ref{standard_module_reducibility}). In this way, we have strengthened the result of Atobe which gives only the sufficient condition (cf. Theorem 3.13. of \cite{Atobe_gen}).

The second part of this paper deals with the theta lifts of generic representations of the metaplectic group. Let us briefly recall the basic setting (for more details, see Section \ref{subs:thetaintro}). Let $V$ be a quadratic space of dimension $2r+1$ over $F$ (i.e., a space endowed with a non-degenerate $F$-bilinear form) and let $O(V)$ denote the corresponding orthogonal group. If $W_{2n}$ is a symplectic space of dimension $2n$, then $O(V)$ and $\widetilde{Sp(W_{2n})}$ form a dual pair inside $\widetilde{Sp(W_{2n(2r+1)})}$. Fixing a non-trivial additive character $\psi$ of $F$, we obtain the so-called Weil representation $\omega_{2n(2r+1),\psi}$ of the metaplectic group $\widetilde{Sp(W_{2n(2r+1)})}$. Restricting this representation to $O(V) \times \widetilde{Sp(W_{2n})}$ we obtain the Weil representation $\omega_{2r+1,2n,\psi}$ of this dual pair.

Now for any $\pi \in \text{Irr} \widetilde{Sp(W_{2n})}$ we may look at the maximal $\pi$-isotypic quotient of $\omega_{2r+1,2n,\psi}$. We denote it $\Theta(\pi,2r+1)$ and call it the full theta lift of $\pi$ to $V$. This representation, when non-zero, has a unique irreducible quotient, denoted $\theta(\pi,2r+1)$---the small theta lift of $\pi$. This basic fact, called the Howe duality conjecture, was first formulated by Howe \cite{Howe_theta_series}, proven by Waldspurger \cite{Waldspurger_howe_duality} (for odd residue characteristic) and by Gan and Takeda \cite{Gan_Takeda_proof_of_Howe} in general.

The Howe duality establishes a map $\pi \mapsto \theta(\pi)$ which is called the theta correspondence. It is an exceptionally useful tool in the representation theory of p-adic groups. However, its importance also stems from number-theoretic considerations, since the global variant of theta correspondence can be used for constructing automorphic representations. For this reason, theta correspondence has been an area of active research for the last forty years. The study of theta correspondence was initiated by Roger Howe \cite{Howe_theta_series,Howe_transcending} and continued by Kudla \cite{Kudla2,Kudla1}, Rallis \cite{rallis1984howe}, Kudla-Rallis \cite{KR1}, Moeglin-Vigneras-Waldspurger \cite{MVW_Howe}, Waldspurger \cite{Waldspurger_howe_duality} and many others. In recent years this topic has seen a major revival of interest, with many new developments and many old problems being resolved.
However, the two main problems concerning theta lifts still remain open: determining when $\Theta(\pi,2r+1)$ is non-zero and identifying $\theta(\pi,2r+1)$ explicitly. The second main contribution of this paper is the complete resolution of these problems in case when $\pi$ is a generic representation of the metaplectic group.

Our approach to these problems relies on a number of different techniques and results: the Jacquet module technique as utilized by Mui\'{c} \cite{muic2004howe,Muic_theta_discrete_Israel,muic2008theta}, the results of Atobe and Gan on the lifts of tempered representations \cite{Atobe_Gan}, the results of Gan and Savin \cite{Gan_Savin_Metaplectic_2012} on the theta lifts of representations of the metaplectic group, and the work of Mui\'{c} \cite{Muic_proof_Casselman_Shahidi} and the second author \cite{Hanzer_injectivity} on the standard module/generalized injectivity conjecture.

The question of occurrence is answered in Propositions \ref{l(pi)} and \ref{prop:l(pi),l(sigma)=0}. We have 
\begin{prop}
Let $\pi=L(\gamma_{\psi}^{-1}\delta_1\nu^{s_1},\gamma_{\psi}^{-1}\delta_2\nu^{s_2},\ldots,\gamma_{\psi}^{-1}\delta_k\nu^{s_k};\sigma)$ be a $\psi$--generic representation of $\widetilde{Sp(2n,F)}.$ Then we have one of the following
\begin{enumerate}
\item Assume that none of $\delta_1\nu^{s_1},\ldots,\delta_k\nu^{s_k}$ is of the form $\delta([\nu^{1/2},\nu^{2s_i-1/2}]).$ Then, $l(\pi)=l(\sigma).$
\item Assume that there exists (a unique!, cf.~Theorem  \ref{generic_reps_metaplectic}) $i\in \{1,2,\ldots,k\}$ such that $\delta_i\nu^{s_i}=\delta([\nu^{1/2},\nu^{2s_i-1/2}]).$ Then,
\begin{itemize}
\item If $l(\sigma)=2,$ then $s_i=\frac{1}{2}$ and $l(\pi)=l(\sigma)=2.$
\item If $l(\sigma)=0,$  then $l(\pi)=l(\sigma)=0$ unless $s_i=\frac{1}{2}.$ In that case, $l(\pi)=2$.
\end{itemize}
\end{enumerate}
\end{prop}
Here we set $l(\pi) = 2n + 1 - m^\text{down}(\pi)$, where $m^\text{down}(\pi)$ denotes the smaller of the two first occurrence indices of $\pi$ (see \S\ref{sec:non-vanishing} for notation).

The answer to the second question, i.e., the description of the lifts, is given as a series of results: Propositions \ref{prop:lifts}, \ref{first_lifts}, \ref{higher_lifts_l(sigma)=0_down}, \ref{higher_lifts_l(sigma)=0_up}, \ref{non_trivial_charcter_first}, \ref{non_trivial_charcter_higher} and Corollary \ref{s_k_1} provide a complete description of $\theta(\pi,2r+1)$ in terms of its standard module.

Let us briefly describe the contents of this paper. In Section \ref{sec_preli} we recall some basic definitions and results concerning the metaplectic group and theta correspondence. In Section \ref{sec:gen} we describe the generic representations of $\widetilde{Sp(2n,F)}$. Our main tool is the theory of local coefficients, which we also introduce in this section. After determining the cuspidal reducibilities (Proposition \ref{cuspidal_reducibility}), we prove our main result on generic representations, Theorem \ref{generic_reps_metaplectic}. We also state Theorem \ref{standard_module_reducibility} which describes all the situations in which the standard module conjecture fails. In Section \ref{sec:non-vanishing} we address the non-vanishing of theta lifts. Sections \ref{sec:lifts} and \ref{sec:lifts2} contain the statements and the proofs of the results providing a complete description of the theta lifts of generic representations. In Section \ref{sec:lifts} we compute the lifts of those representations which satisfy the standard module conjecture, whereas in Section \ref{sec:lifts2} we treat the exceptional cases. Appendix A contains the proof of Theorem \ref{standard_module_reducibility}.

\medskip
\noindent{\textbf{Acknowledgements}}\\
This work is  supported in part by Croatian Science Foundation under the project IP-2018-01-3628.

\section{Preliminaries}
\label{sec_preli}
We start by introducing the notation concerning (complex) smooth representations of general linear and classical groups over non-Archimedean fields and by reviewing some basic results. Let $F$ be a non-Archimedean field of characteristic zero.
We recall the Zelevinsky notation for the parabolic induction for general linear and classical  $p$-adic groups. Let $\pi_1,\ldots,\pi_k$ be  representations of $GL(n_i,F),\;i=1,\ldots, k.$ The group $GL(n_1+n_2+\cdots +n_k,F)$ has a  standard parabolic subgroup, say $P,$ whose Levi subgroup $M$ is isomorphic to $GL(n_1,F)\times GL(n_2,F)\times \cdots\times GL(n_k,F).$ Then we denote $\mathrm{Ind}_P^{GL(n_1+n_2+\cdots +n_k,F)}(\pi_1\otimes \pi_2\otimes \cdots \otimes \pi_k)$ (the normalized induction)
   by $\pi_1\times \pi_2\times \cdots \times \pi_k.$ Analogously, if $G$ is a classical group which has a Levi subgroup $M$ (of a standard parabolic subgroup $P$) isomorphic to $GL(n_1,F)\times GL(n_2,F)\times \cdots\times GL(n_k,F)\times G',$ where $G'$ is a classical group of the same type  and  smaller rank, and if $\pi_1,\ldots,\pi_k$ are  representations of $GL(n_i,F),\;i=1,\ldots, k$ and $\sigma$ a representation of $G',$ we denote  $\mathrm{Ind}_P^{G}(\pi_1\otimes \pi_2\otimes \cdots \otimes \pi_k\otimes \sigma)$ by $\pi_1\times \pi_2\times \cdots \times \pi_k\rtimes \sigma.$ We denote by $\nu$ a character of $GL(n,F)$ obtained by composing the determinant character with the absolute value on $F^*.$

Let $\rho$ be an irreducible unitary cuspidal representation of $GL(m,F)$ and $\alpha,\beta \in \R$ such that $\alpha-\beta+1=k\in \N.$ Then, the induced representation
$\rho \nu^{\alpha}\times \rho \nu^{\alpha-1}\times \cdots\times \rho \nu^{\beta}$ has a unique irreducible subrepresentation which we denote by $\delta([\rho\nu^{\beta},\rho\nu^{\alpha}]).$ This representation of $GL(mk,F)$ is essentially square-integrable and any essentially square integrable representation of a general linear group is obtained in this way. We will also denote by $St_n$
 the Steinberg representation of $GL(n,F),$ which is $\delta([\nu^{-\frac{n-1}{2}},\nu^{\frac{n-1}{2}}]).$ The (standard) representation $\rho \nu^{\alpha}\times \rho \nu^{\alpha-1}\times \cdots\times \rho \nu^{\beta}$ has also the unique irreducible quotient, namely the Langlands quotient, which we, in this situation, denote by $\zeta(\rho\nu^{\beta},\rho\nu^{\alpha}).$ When $\rho=1_{GL(1,F)},$ we shorten the notation even more, and use $\zeta(\beta,\alpha)$ to denote it. Note that this is a twist of the trivial representation of $GL(k,F).$

 \subsection{Symplectic and orthogonal groups}
For $n\in \Z_{\geq 0},$ let $W_{2n}$ be a symplectic vector space  over $F$ of dimension $2n.$ We fix a complete polarization as follows
\[W_{2n}=W_n'\oplus W_n'',\; W_n'=\mathrm{span}_F \{e_1,\ldots
  e_n\},\;W_n''=\mathrm{span}_F \{e_1',\ldots e_n'\},\]
  where $e_i,e_i',\;i=1,\ldots ,n$ are  basis vectors of $W_{2n}$ and the skew--symmetric form on $W_{2n}$ is described by the relations
  \[\langle e_i,e_j\rangle=0,\;i,j=1,2,\ldots,n,\;\langle e_i,e_j'\rangle=\delta_{ij}.\]
  The group $Sp(W_{2n})$ fixes this form. Let $P_j$ denote the maximal parabolic subgroup of $Sp(W_{2n})$ stabilizing the isotropic space $W_n'^j=\mathrm{span}_F \{e_1,\ldots e_j\};$ then there is a Levi decomposition $P_j=M_jN_j$ where $M_j=GL(W_n'^j).$ By adding, in each step, a hyperbolic plane to the previous symplectic vector space, we obtain a tower of symplectic spaces and corresponding symplectic groups. We  also use $Sp(2n,F)$  to denote $Sp(W_{2n}).$

  Now we describe the orthogonal groups we consider. Let $V_0$ be an
  an\-iso\-tro\-pic quadratic space over $F$ of odd dimension; then $\mathrm{dim} V_0\in \{1,3\}.$ For the description  of the invariants of this quadratic space, including the quadratic character $\chi_{V_0}$ describing the quadratic form on
  $V_0,$ we refer to e.g.~Chapter V of \cite{Kudla1}. In each step, as for the symplectic situation, we add a hyperbolic plane and obtain an
  enlarged quadratic space and, consequently, a tower of quadratic spaces and a tower of corresponding orthogonal groups. In the case in which $r$ hyperbolic planes are added to the anisotropic space, a corresponding orthogonal group will be denoted $O(V_m),$ where $V_m=V_r'+V_0+V_r''$ and $V_r'$ and $V_r''$ are defined analogously
  as in the symplectic space and $m=\mathrm{dim}V_m=2r+\dim V_0.$ Again, $P_j$ will be the maximal parabolic subgroup stabilizing $\mathrm{span}_F \{e_1,\ldots e_j\}.$  We will also use $O(m,F)$ to denote $O(V_m),$ and, more specifically, we use $O(m,F)^+$ to emphasize that $O(m,F)$ is in the split tower and $O(m,F)^-$ to emphasize that $O(m,F)$ is in the non-split tower.

  Now we give a simple consequence of the Langlands classification for the discrete series representations of $SO(2n+1,F)$ obtained in \cite{Arthur_endoscopic}, paired with explicit classification of these representations given by Moeglin-Tadić (\cite{MT}). In \cite{MT}, the authors gave a classification of the discrete series representations for quasi-split classical groups in terms of three invariants: the Jordan block, partial cuspidal support and $\epsilon$--function. Using these invariants one can realize a discrete series representation as a subrepresentation of a certain parabolically induced representation. In \cite{Moeglin_1} it is proved that Moeglin-Tadić classification indeed corresponds to Langlands parametrization, and $\epsilon$--function corresponds to the  parametrization of the discrete representations inside one L-packet by characters of the component group of the centralizer of the L-parameter. In the case of generic representations this factor is trivial (cf.~\cite{Hanzer_injectivity}, Proposition 3.1 and \cite{Atobe_uniqueness_generic_L_packet}). 

  \noindent Thus, let $\phi:WD_F\to Sp(2n,\C)$ be the L-parameter of a generic discrete series $\pi$ of $SO(2n+1,F).$
  i.e., $\phi=\oplus (\rho \otimes S_a)$ is a multiplicity one representation, where $\rho \otimes S_a$ is an irreducible symplectic representation of $WD_F.$ Here we can identify an irreducible smooth $m$-dimensional representation $\rho$ of $W_F$ with an irreducible cuspidal representation of $GL(m,F)$ (Henniart, Harris-Taylor) and $S_a$ denotes the unique algebraic representation of $SL(2,\C)$ of dimension $a.$
 For each $\rho$ we denote 
 \[\phi_{\rho}(\pi)=\{a\in \N: \rho \otimes S_a\hookrightarrow \phi\}.\]
 Let $\phi_{\rho}(\pi)=\{a_1^{\rho}<a_2^{\rho}<\cdots<a_{l_{\rho}}^{\rho}\}.$ Note that members of $\phi_{\rho}(\pi)$ are all of the same parity and $l_{\rho}$ may be odd or even.
 Then, by \cite{MT},\cite{Arthur_endoscopic} and \cite{Moeglin_1} (recall that $\pi$ is generic), we have
 \begin{gather}
 \label{embedding_discrete}
 \pi\hookrightarrow \prod_{\substack{\rho,\\
  \phi_{\rho}(\pi) \text{ are even }}}\left(\prod_{r=1}^{t_{\rho}}\delta([\rho\nu^{-\frac{a_{2r-1}^{\rho}-1}{2}},\nu^{\frac{a_{2r}^{\rho}-1}{2}}])\right )\times \delta([\rho\nu^{\frac{1}{2}},\rho\nu^{\frac{a_{2t_{\rho}+1}^{\rho}-1}{2}}]) \times  \\  \notag
\prod_{\substack{\rho,\\
  \phi_{\rho}(\pi) \text{ are odd }}}\left(\prod_{r=1}^{t_{\rho}}\delta([\rho\nu^{-\frac{a_{2r-1}^{\rho}-1}{2}},\nu^{\frac{a_{2r}^{\rho}-1}{2}}])\right )\times \delta([\rho\nu^{1},\rho\nu^{\frac{a_{2t_{\rho}+1}^{\rho}-1}{2}}])
  \rtimes\sigma_{cusp}'.
  \end{gather}
Here, if members of $\phi_{\rho}(\pi)$ are even, and $|\phi_{\rho}(\pi)|$ is even, there is no part $\delta([\rho\nu^{\frac{1}{2}},\rho\nu^{\frac{a_{2t_{\rho}+1}^{\rho}-1}{2}}])$ above; analogously, if  members of $\phi_{\rho}(\pi)$ are odd, and $|\phi_{\rho}(\pi)|$ is even, there is no part $\delta([\rho\nu^1,\rho\nu^{\frac{a_{2t_{\rho}+1}^{\rho}-1}{2}}])$ above.

\subsection{The metaplectic group}
 Let $W_{2n}$ be the symplectic space as above. The metaplectic group $\widetilde{Sp(2n,F)}$ is given as the non-trivial  central extension 
\begin{equation}
  \label{egzaktni}
 1\to \mu_2\to \widetilde{Sp(2n,F)}\to Sp(2n,F)\to 1
\end{equation}
where $\mu_2=\{1,-1\}$ and the cocyle involved is Rao's cocycle (\cite{Rao}). For the more thorough description of the structural theory of the metaplectic group we refer to \cite{Kudla1},\cite{Rao},\cite{Gelbart},\cite{Hanzer_Muic_metaplectic_Jacquet}. Specifically,
for every subgroup $G$ of $Sp(2n,F)$ we denote by $\widetilde{G}$ its preimage in $\widetilde{Sp(2n,F)}.$ In this way, the standard parabolic subgroups of $\widetilde{Sp(2n,F)}$ are defined. Then, we have $\widetilde{P_j}=\widetilde{M_j}N_j',$ where $N_j'$ is the image in $\widetilde{Sp(2n,F)}$ of the unique monomorphism from $N_j$ (the unipotent radical of $P_j$) to $\widetilde{Sp(2n,F)}.$ We emphasize that $\widetilde{M_j}$ is not a product of $GL$ factors and a metaplectic
group of smaller rank, but there is an epimorphism (this is the case of maximal parabolic subgroup, as denoted in the previous subsection)
\[\phi:\widetilde{GL(j,F)}\times \widetilde{Sp(2n-2j,F)}\to \widetilde{M_j}. \]
More information can be found in \cite{Hanzer_Muic_metaplectic_Jacquet} and in \cite{Szpruch_PhD}, the fourth chapter.
Here, we can view $\widetilde{GL(j,F)}$ as  two--fold cover of $GL(j,F)$ in its own right.
In this way, an irreducible representation $\pi$ of $\widetilde{M_j}$ can be considered as a representation $\rho\otimes \sigma$ of $\widetilde{GL(j,F)} \times \widetilde{Sp(2n,F)},$ where $\rho$ and $\sigma$ are irreducible
representations, provided they are both trivial or both non--trivial when restricted to $\mu_2.$
\subsubsection{Genuine representations of the metaplectic group}
We can form the genuine representations (i.e., the ones which are non-trivial on  $\mu_2$, cf. \eqref{egzaktni}) of $\widetilde{GL(j,F)}$ (and then of $\widetilde{M_j}$) by tensoring representations of $GL(j,F)$ by a genuine character of  $\widetilde{GL(j,F)}$ given by 
\begin{equation}
\label{gamma_psi}
(g,\varepsilon)\mapsto \varepsilon \gamma_{\psi'}(\det g)^{-1}
\end{equation} (cf.~p.~36 of \cite{Szpruch_PhD}). Here $\gamma_{\psi'}$ is the Weil index of a character of the second degree given by $x\mapsto \psi'(x^2)$ (cf.~\cite{Kudla2} and \cite{Szpruch_PhD}, the third section), attached to a non-trivial character $\psi'$ of $F.$
It is defined as a quotient of the (principal values) of  the following integrals:
\[\gamma_{\psi'}(a)=|a|^{\frac{1}{2}} \frac{\int_F\psi'(ax^2)\deri x}{\int_F\psi'(x^2)\deri x}.\]
From the definition we can easily derive the relation between $\gamma_{\psi}$ and $\gamma_{\psi_a},$ where $\psi_a(x)=\psi(ax),\forall x\in F.$ Namely,
\begin{equation}
\label{relation_psis}
\gamma_{\psi_a}(x)=(x,a)_F\gamma_{\psi}(x).
\end{equation} Here $(x,a)_F$ denotes the Hilbert symbol for the field $F,$ so that $x\mapsto (x,a)_F$ is a quadratic character of $F^*$ which we denote by $\chi_a.$
\smallskip

\subsubsection{Theta correspondence}
\label{subs:thetaintro}
We need some notation: Assume that $\Pi$ is a smooth representation of a  product of  $l$--groups $G_1\times G_2.$ Let $\xi$ be an irreducible smooth representation of $G_1;$ by $\Theta (\xi, \Pi)$ we denote the isotypic component of $\xi$ in $\Pi.$  More explicitly, with  
\[W:=\bigcap_{\substack{f:\Pi\to \xi\\ G_1  intertwining }}\mathrm{Ker}f\]
 we have $\Theta (\xi, \Pi)=\Pi/W.$ The representation $\Theta (\xi, \Pi)$ has a natural structure of $G_2$--module and 
\begin{equation}
\label{hom_isotypic}
\Hom_{G_1}(\Pi,\xi)_{\infty}\cong \Theta (\xi, \Pi)^{\vee{}};
\end{equation}
here $\Hom_{G_1}(\Pi,\xi)_{\infty}$ denotes the smooth part of $\Hom_{G_1}(\Pi,\xi)$ and $\vee{}$ denotes the contragredient.

\noindent We now return to theta correspondence.
We fix a non-trivial additive character $\psi$ of $F$ and  we let $\omega_{2r+1,2n,\psi}$ be the pullback of the Weil representation
$\omega_{2n(2r+1),\psi}$ of the group $\widetilde{Sp(2n(2r+1),F)}$, restricted to the dual pair $\widetilde{Sp(2n,F)}× O(2r+1,F)$
(\cite{Kudla2}). Here  $O(2r+1,F)$ belongs to one of two towers of odd orthogonal groups with fixed quadratic character (cf.~\cite{Kudla2}, chapter  V) and  $2r+1$ is the dimension of the quadratic space. We refer to it as "a pair of orthogonal towers".
Let $\pi$ be an irreducible smooth representation of  $\widetilde{Sp(2n,F)}.$  We say that the theta lift on the space of dimension $m$ (on one of the orthogonal towers) is non-zero if $\Theta (\pi, \omega_{m,2n,\psi})\neq 0.$ We then call $\Theta(\pi, \omega_{m,2n,\psi})$ {\it{the full theta lift}} of $\pi$ on $V_m$ and, to simplify the notation,  we will denote it by $\Theta(\pi,m).$ This is, as observed above, a representation of $O(m,F).$ Note that, in this notation, the dependence on $\psi$ is suppressed; also it is assumed that we know to which tower this lift refers to (i.e., whether it is a representation of $O(m,F)^{+}$ or $O(m,F)^{-}$).

By the {\it{Howe duality conjecture}} (cf.~\cite{Howe_theta_series}, \cite{Howe_transcending}), proved by Waldspurger when residual characteristic is different form 2 (\cite{Waldspurger_howe_duality}), and in general case by Gan and Takeda (\cite{Gan_Takeda_proof_of_Howe}), the representation $\Theta (\pi, m)$ has a unique irreducible quotient which we call the {\it{small theta lift}} and  denote by $\theta(\pi, m).$ 
Moreover, the correspondence
\[\pi\leftrightarrow \theta(\pi, m)\]
is a bijection between representations of  $\widetilde{Sp(2n,F)}$  and $O(V_m)$ participating in the theta correspondence (i.e. having a non-zero lifts).

 It is known that there is exactly one odd orthogonal tower (in a pair, as above) such that the theta lift of $\pi$ to that tower on the dimension level $2n+1$ is non-zero. This follows from {\it{the conservation conjecture}}, originally conjectured by Kudla and Rallis (\cite{Kudla_Rallis_progress}), and finally proved (in the general case) by Sun and Zhu (\cite{Sun_Zhu_conservation}).  Throughout the second section,  we fix a pair of orthogonal towers which are attached to the trivial character, as in \cite{Gan_Savin_Metaplectic_2012}. There (cf. Introduction of \cite{Gan_Savin_Metaplectic_2012}) the following parametrization of the irreducible representations of $\widetilde{Sp(2n,F)}$ is given
\begin{equation}
\label{Classification_metaplectic}
\mathrm{Irr}\widetilde{Sp(2n,F)}\longleftrightarrow \mathrm{Irr}SO(2n+1,F)^+\cup \mathrm{Irr}SO(2n+1,F)^{-}.
\end{equation}

This bijection is given by theta correspondence: for a given representation $\pi$ of $\widetilde{Sp(2n,F)}$ we obtain representation $\theta(\pi,2n+1)\neq 0$ in one of the towers, say $\varepsilon\in\{+,-\}$ (the lift to the other tower is zero)  and then restrict it to a representation of $SO(2n+1,F)^{\varepsilon}.$ This restriction remains irreducible.
 On the other hand, for given  irreducible representation $\sigma$ of $SO(2n+1,F)^{\varepsilon}$ exactly one of the two possible extensions of this representation to 
$O(2n+1,F)^{\varepsilon}$ participates in the theta correspondence with the metaplectic group  $\widetilde{Sp(2n,F)}.$ We denote this (cf.~\eqref{Classification_metaplectic}), slightly modified theta correspondence, by $\Theta_{\psi}(\cdot).$  That is, if $\pi$ is an irreducible representation of $\widetilde{Sp(2n,F)},$ then 

\begin{equation}
\label{dfn:GS_1}
\Theta_{\psi}(\pi, 2n+1)=\theta(\pi,2n+1)|_{SO(2n+1,F)^{\varepsilon}}
\end{equation} 

and if $\sigma$ is a representation of $SO(2n+1,F)^{\varepsilon},$ then
\begin{equation}
\label{dfn:GS_2}
 \Theta_{\psi}(\sigma,2n)=\theta(\sigma^{\delta},2n),
\end{equation}
  where $\sigma^{\delta}$ is the unique extension of $\sigma$ to $O(2n+1,F)^{\varepsilon}$ whose lift to $\widetilde{Sp(2n,F)}$ is non-zero. Here $\delta\in\{1,-1\}$ denotes the value of the extended representation on $-I\in O(2n+1,F)^{\varepsilon}\setminus SO(2n+1,F)^{\varepsilon}.$

We now give a result which we use later. Although parts of it were known earlier in some form (e.g.~\cite{Muic_theta_discrete_Israel}, Theorem 6.2), we are stating it in a form given in \cite{Gan_Savin_Metaplectic_2012}. We are taking into account the fact that the Howe duality conjecture has meanwhile been proven  (cf.~\cite{Gan_Takeda_proof_of_Howe}).
\begin{prop}(cf.~Theorem 8.1.(i) and (ii) of \cite{Gan_Savin_Metaplectic_2012})
\label{irreducibility_temp_lift}
 For an irreducible tempered representation $\pi$ of $\widetilde{Sp(2n,F)},$ $\Theta(\pi, 2n+1)$  (the non-zero full lift on the appropriate tower) is irreducible and tempered. If $\pi$ is, moreover, a discrete series representation,  $\Theta(\pi, 2n+1)$ is a discrete series representation (and, of course, irreducible). The analogous claim holds for the irreducible tempered representations of  orthogonal groups $O(2n+1,F)^{\varepsilon}.$
\end{prop}
Finally, we take a moment to explain the notation for theta lifts we use in the rest of the paper, motivated by \cite{Atobe_Gan}. Let $\pi$ be an irreducible representation of $\widetilde{Sp(2n,F)}$; assume that we have fixed the target tower $(V_m)$. Then we denote by $\theta_l(\pi)$ the lift of $\pi$ to $V_m$ ($m=\dim V_m$), where $l$ and $m$ are related by $l=2n-m+1$. Similarly, if $\pi$ is an irreducible $O(V_m)$-representation, we let $\theta_l(\pi)$ denote the lift to $\widetilde{Sp(2n,F)}$, where $l=m - 2n - 1$. We use analogous notation for the full lifts $\Theta(\pi)$. Note that $2n$ is even and $m$ is odd, so that $l$ is always even. Using this notation, the theta correspondence between $\widetilde{Sp(2n,F)}$ and $O(2n+1,F)$ referred to in this section corresponds to $l=0$. Similarly, we have $\pi = \theta_l(\theta_{-l}(\pi))$ whenever $\theta_{-l}(\pi) \neq 0$.

\section{Generic representations of the metaplectic group}
\label{sec:gen}
Let $\psi$  be  an additive, non--trivial character of $F.$ By fixing the base $e_i,e_i',\;i=1,\ldots ,n$ as in the section 2.1., we get a matrix realization of $Sp(2n,F)$ where we fix a Borel subgroup $B(F)$ consisting of the upper triangular matrices, so that the unipotent radical  $U(F)$ of $B(F)$ consists of the unipotent upper triangular matrices in $Sp(2n,F). $ We use $\psi$ to also denote a non-degenerate character of this unipotent subgroup, cf.~p.~15 of \cite{Szpruch_PhD}. This unipotent subgroup splits uniquely in $\widetilde{Sp(2n,F)},$ so that the full preimage  $\widetilde{U(F)}$ in $\widetilde{Sp(2n,F)}$ is isomorphic to $U(F)\times \{\pm 1\}.$ Again we use $\psi$ to denote the character of $\widetilde{U(F)}$ given by $(z,\varepsilon)\mapsto \varepsilon\psi(z)$ (cf.~p.~41  of \cite{Szpruch_PhD}).
\smallskip

Let $(\pi, V_{\pi})$ be a smooth representation of $\widetilde{Sp(2n,F)}.$ By a $\psi$--Whittaker functional on $\pi$ we mean a linear functional $\lambda$ on $V_{\pi}$ satisfying
\[\lambda(\pi(z)v)=\psi(z)\lambda(v),\forall v\in V_{\pi},z\in \widetilde{U(F)}.\]
Similarly to the algebraic case, in the situation in which $\pi$ is $\psi$--generic, we can form the space of Whittaker functions $W_{\pi,\psi}$ which affords a representation of $\pi$ by right translations.
Now, in the case of  connected, quasi-split reductive groups, the uniqueness of the Whittaker model enables the definition of the local coefficients. The same thing can be done in the metaplectic case, due to the following theorem of Szpruch: 
\begin{thm}[Szpruch, Theorem 5.1. of  \cite{Szpruch_PhD}]
\label{thm_uniqueness_whittaker}
Let $\pi$ be an irreducible admissible representation of $\widetilde{Sp(2n,F)}.$ Then,
\[\mathrm{dim}W_{\pi,\psi}\le 1.\]
\end{thm}
Analogously to the algebraic case, the heredity property is valid for the parabolically induced generic representations of metaplectic group (cf.~\cite{Rodier_Whittaker}  for the algebraic case and \cite{Banks_Whittaker_heredity} for the covering case).

\subsection{Local coefficients}
We continue with the definition of the local coefficients, following the exposition in the sixth chapter of \cite{Szpruch_PhD}, which, in turn, follows the algebraic case defined in Theorem 3.1. of \cite{Shahidi_certain}. Let $M\cong GL(n_1,F)\times GL(n_2,F)\times \cdots \times GL(n_k,F)\times Sp(2n',F)$ be a standard Levi subgroup attached to a (standard) parabolic subgroup $P=MN$ of $Sp(2n,F).$ Let $w$ be an element of the Weyl group which takes $M$ to another standard Levi subgroup.  Let $\tau_1,\ldots,\tau_k$ be irreducible admissible generic representations of $GL(n_i,F),\;i=1,2,\ldots,k$ and let $\sigma$ be a $\psi$--generic irreducible admissible representation of $\widetilde{Sp(2n',F)}.$ Then, by 
\begin{equation}
\label{induced_rep}
I(\tau_{1_{(s_1)}},\tau_{2_{(s_2)}},\ldots,\tau_{1_{(s_k)}},\sigma)
\end{equation}
we denote the representation of $\widetilde{Sp(2n)}$ induced  from the representation of $\widetilde{P}$ as described in the section  4.1. of \cite{Szpruch_PhD}. Here $s_i\in \C,\;i=1,2,\ldots, k$ so that  $\tau_{i_{(s_i)}}=\tau_i|\mathrm{det}(\cdot)|^{s_i}.$ We often abbreviate $\overrightarrow{s}=(s_1,\ldots,s_k).$ Note that in the notation for the induced representation \eqref{induced_rep} we supressed the dependence of the induced representation on a character $\psi'$ (needed to form a genuine representation, cf. subsection above and  \eqref{gamma_psi}). If we want to emphasize the dependence on $\psi',$ we will use Zelevinsky notation for parabolic induction and write
\[\gamma_{\psi'}^{-1}(\tau_1\nu^{s_1}\times \tau_2\nu^{s_2}\times \cdots \times \tau_k\nu^{s_k})\rtimes \sigma.\]
Since all irreducible representations $\tau_i,\;i=1,2,\ldots,k$ and $\sigma$ are $\psi$--generic, by heredity, there is a (unique, up to a scalar) $\psi$-- Whittaker functional  on $I(\tau_{1_{(s_1)}},\tau_{2_{(s_2)}},\ldots,\tau_{1_{(s_2)}},\sigma),$ denoted by 
\[\lambda (\overrightarrow{s},(\otimes_{i=1}^r\tau_i)\otimes \sigma,\psi).\]
This functional is defined by an integral in \cite{Szpruch_PhD}, (6.4), analogously to the algebraic case defined by Shahidi.
Note that this functional is entire in $\overrightarrow{s}$  and not identically zero for each $\overrightarrow{s}$ (we choose a flat section from the compact picture of the induced representation and understand meromorphic properties form that;  cf. \cite{Banks_corollary} for the  proof in the case of covering groups).
Let $A_w$ be the standard intertwining operator acting on  the representation \eqref{induced_rep}, defined by (4.10) of \cite{Szpruch_PhD}. Then, by Theorem \ref{thm_uniqueness_whittaker}, we can define the \textbf{local coefficient} $C_{\psi}^{\widetilde{Sp(2n,F)}}(\widetilde{P},\overrightarrow{s},(\otimes_{i=1}^r\tau_i)\otimes \sigma,w)$ as follows:
\begin{equation}
\label{defn_local_coeff}
\begin{split}
\lambda (\overrightarrow{s},(&\otimes_{i=1}^r\tau_i)\otimes \sigma,\psi)=\\&C_{\psi}^{\widetilde{Sp(2n,F)}}(\widetilde{P},\overrightarrow{s},(\otimes_{i=1}^r\tau_i)\otimes \sigma,w)\lambda (\overrightarrow{s}^{w},(\otimes_{i=1}^r\tau_i)^w\otimes \sigma,\psi)\circ A_w.
\end{split}
\end{equation}

\begin{rem}Note that, although it is suppressed from the notation, the local coefficient 
\[C_{\psi}^{\widetilde{Sp(2n,F)}}(\widetilde{P},\overrightarrow{s},(\otimes_{i=1}^r\tau_i)\otimes \sigma,w)\]
 actually depends on two additive characters: $\psi,$ which enters the definition of $\psi$--generic representations, and $\psi'$, which enters the definition of the induced representation \eqref{induced_rep}. Majority of results in \cite{Szpruch_PhD} are obtained for the case $\psi=\psi',$ except in the seventh and eight section where the local coefficients for the principal series  are calculated. The possible consequences of the case in which $\psi\neq \psi'$ are discussed in the section 8.4. there. By \eqref{relation_psis}, and relation (4.1) of \cite{Szpruch_PhD} (which explicitly describes the formation of the genuine representation of a Levi subgroup), we  easily get that
 \begin{equation}
\label{comparison}
 \gamma_{\psi}^{-1}(\tau_1\nu^{s_1}\times \dotsm \times \tau_k\nu^{s_k})\rtimes \sigma=\gamma_{\psi_a}^{-1}(\tau_1\nu^{s_1}\chi_a\times \dotsm \times \tau_k\nu^{s_k}\chi_a)\rtimes \sigma.
 \end{equation}

\end{rem}

Now we assume that we are examining local factor with respect to $\psi$--generic representation, and that the induced representations are formed using $\gamma_{\psi}$ (so, in the light of previous Remark, $\psi=\psi'$).
Assume that the representation \eqref{induced_rep} is induced from an irreducible representation of a maximal parabolic subgroup, i.e., $k=1.$ By relation (6.8) and Theorem 9.3.~of \cite{Szpruch_PhD}, for  irreducible admissible $\psi$--generic representations $\tau$ of $GL(k,F)$ and $\sigma$ of $\widetilde{Sp(2n-2k,F)},$ there exists an exponential function $c_F$ of a complex variable, such that
\begin{equation}
\label{C_factor_gamma_factors}
C_{\psi}^{\widetilde{Sp(2n,F)}}(\widetilde{P},s,\tau\otimes \sigma,w_0)=c_F(s) \frac{\gamma(\tau,sym^2,2s,\psi)}{\gamma(\tau,s+ \frac{1}{2},\psi)}\gamma(\sigma\times \tau,s,\psi).
\end{equation}
Here $w_0$ is the longest element in the Weyl group of $Sp(2n,F)$ modulo the longest one in $M$ (here $P=MN$ is a maximal parabolic subgroup). The $\gamma$-factors $\gamma(\tau,sym^2,2s,\psi)$ and $\gamma(\tau,s+ \frac{1}{2},\psi)$ are the ones defined by Shahidi  in \cite{Shahidi_Langlands_91} (and they appear in the calculation of the local coefficients for split odd orthogonal groups). As for the gamma factor $\gamma(\sigma\times \tau,s,\psi),$ in their paper \cite{Gan_Savin_Metaplectic_2012}, Gan and Savin relate this $\gamma$-factor to a $\gamma$-factor $ \gamma(\Theta_{\psi}(\sigma)\times \tau,s,\psi)$ for $SO(2n+1,F).$  
We need the following

\begin{lem} (Corollary 9.3. of \cite{Gan_Savin_Metaplectic_2012})
\label{generic_transfer}
\begin{enumerate}
\item Assume $\pi\in Irr(SO(2n+1,F)^+)$ is generic. Then, $\Theta_{\psi}(\pi)$ is $\psi$--generic.
\item Assume  $\pi\in Irr(\widetilde{Sp(2n,F)})$ is $\psi$--generic and tempered. Then,  $\Theta_{\psi}(\pi)$ is  a generic representation of $SO(2n+1,F)^+.$
\end{enumerate}
\end{lem}

We also have
\begin{cor}
\label{transfer_gamma_factors}
Assume $\sigma$ is an irreducible, tempered and $\psi$--generic representation of $\widetilde{Sp(2n)}.$ Then, the following holds:
\begin{equation}
\label{gamma_metaplectic_gamma_classical}
\gamma(\sigma\times \tau,s,\psi)=\gamma(\Theta_{\psi}(\sigma)\times \tau,s,\psi).
\end{equation}
 \end{cor}.
 \begin{proof} This follows from Lemma \ref{generic_transfer}  and Proposition 11.1  and the discussion after Corollary 11.2 of \cite{Gan_Savin_Metaplectic_2012}.
 \end{proof}

Immediately from \eqref{defn_local_coeff} and properties of $\psi$-Whittaker functional $\lambda,$ we can derive the following
\begin{cor}  
\label{generic_Langlands_quotient}
We consider an induced representation \eqref{induced_rep}. Assume that $\tau_i,\;i=1,2,\ldots,k$ are tempered representations, $\sigma$ a tempered $\psi$--generic representation and $\overrightarrow{s_0}=(s_1,s_2,\ldots,s_k)$ satisfy $s_1> s_2> \cdots> s_k>0$ so that for $\overrightarrow{s}=\overrightarrow{s_0}$ the representation \eqref{induced_rep} is standard. Then, the Langlands quotient $L(\tau_{1_{(s_1)}},\tau_{2_{(s_2)}},\ldots,\tau_{1_{(s_k)}};\sigma)$ is $\psi$--generic if and only if the local coefficient $C_{\psi}^{\widetilde{Sp(2n,F)}}(\widetilde{P},\overrightarrow{s},(\otimes_{i=1}^r\tau_i)\otimes \sigma,w_0)$ is holomorphic at $\overrightarrow{s}=\overrightarrow{s_0}$ (it is necessarily non-zero at such $s_0$). Here $w_0$ is the longest element in the Weyl group of $\widetilde{Sp(2n,F)}$ modulo the longest one in the corresponding Levi subgroup.
\end{cor}
\begin{proof} Note that, on the right-hand side of  \eqref{defn_local_coeff}, $\lambda (\overrightarrow{s}^{w_0},(\otimes_{i=1}^r\tau_i)^{w_0}\otimes \sigma,\psi)\circ A_{w_0}$ is holomorphic at $\overrightarrow{s}=\overrightarrow{s_0}.$ The image of $ A_{w_0}$ is exactly the Langlands quotient, so that $\lambda (\overrightarrow{s}^{w_0},(\otimes_{i=1}^r\tau_i)^{w_0}\otimes \sigma,\psi)\circ A_{w_0}$ is non-zero exactly in case this Langlands quotient is $\psi$-generic. Assume that Langlands quotient is $\psi$--generic; then  $C_{\psi}^{\widetilde{Sp(2n,F)}}(\widetilde{P},\overrightarrow{s},(\otimes_{i=1}^r\tau_i)\otimes \sigma,w_0)$ needs to be holomorphic in order for the left-hand side of \eqref{defn_local_coeff} to be holomorphic (and it always is). On the other hand, if we assume $C_{\psi}^{\widetilde{Sp(2n,F)}}(\widetilde{P},\overrightarrow{s},(\otimes_{i=1}^r\tau_i)\otimes \sigma,w_0)$ is holomorphic, and since  $\lambda (\overrightarrow{s}^{w_0},(\otimes_{i=1}^r\tau_i)^{w_0}\otimes \sigma,\psi)\circ A_{w_0}$ is holomorphic, for the right-hand side to be non-zero (since left-hand side is), both  the local coefficient and $\psi$-Whittaker functional have to be non-zero (on the image of $A_{w_0}$), and this is precisely the Langlands quotient, which is, thus, $\psi$-generic.
\end{proof}

\subsection{Cuspidal reducibilities}
We are now able to completely describe the cuspidal reducibilities for the metaplectic case, and, in case of reducibility, to describe which subquotient is generic. This result is not new  (part of it  is already present in \cite{MVW_Howe}, \cite{Hanzer_Muic_rank_one}; we also use the relation \eqref{gamma_metaplectic_gamma_classical}  above from \cite{Gan_Savin_Metaplectic_2012} and Atobe and Gan results (\cite{Atobe_Gan})), but we find it convenient to express these facts explicitly. The equivalence of the conditions expressed by  $L$--functions and by Langlands parameter in Proposition \ref{cuspidal_reducibility} is a consequence of Arthur's work (\cite{Arthur_endoscopic}) for classical groups. Note that the work \cite{Gan_Savin_Metaplectic_2012} enabled the Langlands classification for metaplectic group via the classification for odd-orthogonal groups, which fulfills a lot of expected properties.
\begin{prop}
\label{cuspidal_reducibility}
Let $\tau$ be an irreducible, self-dual, cuspidal representation of $GL(k,F)$ and $\sigma$ an irreducible, cuspidal and $\psi$--generic representation of $\widetilde{Sp(2n-2k,F)}.$ We consider an induced representation
\begin{equation}
\label{cusp_reduc}
\gamma_{\psi}^{-1}\nu^s\tau\rtimes \sigma,\;s\ge 0.
\end{equation}
\begin{enumerate}
\item Assume that the Langlands parameter $\phi$ of $\sigma$ does not contain $1\otimes S_2.$ 
Then $\Theta_{\psi}(\sigma)$ is a cuspidal representation.
\begin{enumerate}
\item Assume that the Langlands parameter of $\tau$ factorizes through $Sp(m,\C)$ and $\tau$ (we can consider it as a representation of $W_F$) does not occur as $\tau \otimes S_1$ in $\phi$ (this condition is equivalent to the requirement that neither  $L(2s,\tau,Sym^2)$ nor $L(s,\tau \times \Theta_{\psi}(\sigma))$  has a pole for $s=0$).
Then, the representation \eqref{cusp_reduc} reduces  for $s=0.$
\item Assume that the Langlands parameter of $\tau$ factorizes through $Sp(m,\C)$ and $\tau$  does  occur as $\tau \otimes S_1$ in $\phi$ (this condition is equivalent to the requirement that   $L(2s,\tau,Sym^2)$ does not have a pole for $s=0$ and  $L(s,\tau \times \Theta_{\psi}(\sigma))$ does have a pole for $s=0$).
Then, the representation \eqref{cusp_reduc} reduces  for $s=1.$ Then, the unique  $\psi$--generic subquotient of this representation is a subrepresentation.
\item Assume that the Langlands parameter of $\tau$ factorizes through $SO(m,\C)$ (this condition is equivalent to the condition that $L(2s,\allowbreak \tau,Sym^2)$ has a pole for $s=0$). 
Then, the representation \eqref{cusp_reduc} reduces  for $s=\frac{1}{2}.$ If $\tau=1_{GL(1,F)},$ the unique $\psi$--generic subquotient is the unique irreducible quotient of that representation; and if $\tau \neq 1_{GL(1,F)}$ the unique $\psi$--generic subquotient is a subrepresentation.
\end{enumerate}

\item Assume that the Langlands parameter $\phi$ of $\sigma$  contains $1\otimes S_2.$ 
Then $\Theta_{\psi}(\sigma)$ is a square-integrable representation, and the following holds
\[\Theta_{\psi}(\sigma)\hookrightarrow \nu^{\frac{1}{2}}\rtimes \Theta(\sigma,2n-1)|_{SO(2n+1,F)}.\]
Here $ \Theta(\sigma,2n-1)|_{SO(2n+1,F)}$ is an irreducible cuspidal generic representation.
\begin{enumerate}
\item Assume that the Langlands parameter of $\tau$ factorizes through $Sp(m,\C)$ and $\tau$ does not occur as $\tau \otimes S_1$ in $\phi-1\otimes S_2$  (this condition is equivalent to the requirement that neither  $L(2s,\tau,Sym^2)$ nor $L(s,\tau \times \Theta(\sigma,2n-1)|_{SO(2n+1,F)})$  has a pole for $s=0$).
Then, the representation \eqref{cusp_reduc} reduces  for $s=0.$
\item Assume that the Langlands parameter of $\tau$ factorizes through $Sp(m,\C)$ and $\tau$ does  occur as $\tau \otimes S_1$ in $\phi-1\otimes S_2$ (this condition is equivalent to the requirement that   $L(2s,\tau,Sym^2)$ does not have a pole for $s=0$ and  $L(s,\tau \times \Theta(\sigma,2n-1)|_{SO(2n+1,F)})$ does have a pole for $s=0$).
Then, the representation \eqref{cusp_reduc} reduces  for $s=1$ and the unique  $\psi$--generic subquotient of this representation is a subrepresentation.

\item Assume that the Langlands parameter of $\tau$ factorizes through $SO(m,\C)$ (this condition is equivalent to the condition that $L(2s,\allowbreak \tau,Sym^2)$ has a pole for $s=0$) and that $\tau\neq 1_{GL(1,F)}.$
Then, the representation \eqref{cusp_reduc} reduces  for $s=\frac{1}{2}$ and the unique $\psi$--generic subquotient is a subrepresentation.

\item Assume that $\tau=1_{GL(1,F)}.$
Then, the representation \eqref{cusp_reduc} reduces  for $s=\frac{3}{2}$ and the unique $\psi$--generic subquotient is a subrepresentation.
\end{enumerate}
\end{enumerate}
\end{prop}

\begin{rem} 1. Note that in Proposition \ref{cuspidal_reducibility}, in situation in which the Langlands parameter $\phi$ of $\sigma$ does not contain $1\otimes S_2$, possible reducibilities are generic (in the sense used for the classical groups), i.e.~$0,\frac{1}{2},1$  but it may happen that the generic subquotient is not a subrepresentation of a standard representation, but its Langlands quotient. Thus, the \textbf{standard module conjecture} and \textbf{the generalized injectivity conjecture} fail for the metaplectic groups already in this very simple case.

\noindent 2. If the Langlands parameter $\phi$ of $\sigma$ does  contain $1\otimes S_2,$ although the representation $\sigma$ is generic, a case of a "non-generic reducibility"  at $\frac{3}{2}$ may occur.
\end{rem}

\begin{proof}
(of Proposition \ref{cuspidal_reducibility}) In the same way as for the classical groups, as a composition of two intertwining operators, one can define the Plancherel measure for metaplectic groups (cf.~\cite{Gan_Savin_Metaplectic_2012}, section 10). The behavior of Plancherel measure when inducing from supercusupidals totally governs the reducibility of the induced representation. We have reducibility for $s_0=0$ if and only if $\mu(s,\tau\otimes \sigma,\psi)\neq 0$ for $s=0$ (recall that $\tau\cong \widetilde{\tau}$), and reducibility for $s_0>0$ if and only if the Plancherel measure has a pole for $s=s_0.$ Since we are in the generic case, we can express the Plancherel measure as follows (cf.~\cite{Szpruch_PhD}, sections 4 and 10):
\[\mu(s,\tau\otimes \sigma,\psi)=C_{\psi}^{\widetilde{Sp(2n,F)}}(\widetilde{P},s,\tau\otimes \sigma,w_0)C_{\psi}^{\widetilde{Sp(2n,F)}}(\widetilde{P},-s,\widetilde{\tau}\otimes \sigma,w_0).\]
Now we express the local factors above in terms of $\gamma$--factors as in \eqref{C_factor_gamma_factors}, and $\gamma$--factors in terms of $L$--factors in a usual way, as in \cite{Shahidi_certain}. This readily gives the reducibility points in case 1).
In case 2), we have to use the multiplicativity of $\gamma$--factors (with respect to $\Theta_{\psi}(\sigma)\hookrightarrow \nu^{\frac{1}{2}}\rtimes \Theta(\sigma,2n-1)|_{SO(2n+1)}$) to obtain the claims (for  the multiplicativity, cf.~\cite{Shahidi_multiplicativity}).
Now we can deduce the information about the position of the generic subquotient by analyzing the holomorphicity of $C_{\psi}^{\widetilde{Sp(2n,F)}}(\widetilde{P},s,\tau\otimes \sigma,w_0)$ at the point of reducibility $s=s_0$, as shown in Corollary \ref{generic_Langlands_quotient}.
\end{proof}

Now we give an example where on the same induced representation we can observe Whittaker models with respect to different Whittaker functionals.

\begin{exmp}
\label{primjer}
Let $m\in \N_0.$ Let $\mu_0$ be the non-trivial character of the group $\mu_2$ (cf.~\eqref{egzaktni}), which we view as a representation of $\widetilde{Sp(0,F)}.$ Let $a\in F^*\setminus F^{*,2}.$ Then, the representation
\[\gamma_{\psi}^{-1}\delta([\nu^{\frac{1}{2}},\nu^{m+ \frac{1}{2}}])\rtimes \mu_0 \]
is  both $\psi$-- and $\psi_a$--generic. It has  the length two and the irreducible $\psi_a$--generic subquotient is a (square-integrable) subrepresentation which we denote by $St_{\psi,m+1}$ and the Langlands quotient is the $\psi$--generic subquotient.

\smallskip

\noindent Indeed, the representation 
\[\gamma_{\psi}^{-1}(\nu^{m+ \frac{1}{2}}\times \nu^{m- \frac{1}{2}}\times \cdots \times \nu^{\frac{1}{2}})\rtimes \mu_0\]
is regular, and is of the length $2^{m+1}$ (similarly as in the algebraic case, by Casselman, also cf.~\cite{Tadic_regular} and \cite{Hanzer_Muic_metaplectic_Jacquet} for the metaplectic case), so it easily follows that the representation in question is of the length two, and the unique subrepresentation is the Steinberg representation. Also, since $\mu_0$ is trivially generic (both with respect to $\psi$ and $\psi_a$), we get that the representation in question also is.  In the Siegel case the expression \eqref{C_factor_gamma_factors} becomes
\begin{equation}
\label{lokalni_koeficijent_psi}
C_{\psi}^{\widetilde{Sp(2n,F)}}(\widetilde{P},s,\delta([\nu^{\frac{1}{2}},\nu^{m+ \frac{1}{2}}])\otimes \mu_0,w_0)=c_F(s) \frac{\gamma(\delta([\nu^{\frac{1}{2}},\nu^{m+ \frac{1}{2}}]),Sym^2,2s,\psi)}{\gamma(\delta([\nu^{\frac{1}{2}},\nu^{m+ \frac{1}{2}}]),s+ \frac{1}{2},\psi)}
\end{equation}
and if we want to calculate the local coefficient for the $\psi_a$-Whittaker functional, by \eqref{comparison}, we have
\begin{equation}
\label{lokalni_koeficijent_psi_a}
\begin{split}
C_{\psi_a}^{\widetilde{Sp(2n,F)}}(\widetilde{P},s,&\delta([\nu^{\frac{1}{2}},\nu^{m+ \frac{1}{2}}]) \chi_a \otimes \mu_0,w_0)= \\
& c_{F,a}(s) \frac{\gamma(\delta([\nu^{\frac{1}{2}},\nu^{m+ \frac{1}{2}}]) \chi_a,Sym^2,2s,\psi_a)}{\gamma(\delta([\nu^{\frac{1}{2}},\nu^{m+ \frac{1}{2}}])\chi_a,s+ \frac{1}{2},\psi_a)}.
\end{split}
\end{equation}
Now we use the results of Shahidi \cite{Shahidi_Twisted_endoscopy} to expand these $\gamma$--factors. In the case of \eqref{lokalni_koeficijent_psi}, both the numerator and denominator have a pole of the first order at $s=0,$ so that, in the end, the local coefficient is holomorphic at $s=0.$ In the case of \eqref{lokalni_koeficijent_psi_a} the numerator is the same as in the previous case (thus has a pole at $s=0$), but the denominator is holomorphic, so the local coefficient has a pole for $s=0.$ Now the conclusion follows from Corollary \ref{generic_Langlands_quotient}.
\end{exmp}
\begin{rem}
1. Note that in the case if classical groups, if a standard representation has a square-integrable subquotient, then all the generic subquotients are necessarily square-integrable (Proposition 1.1 of \cite{Muic_proof_Casselman_Shahidi}), which is clearly not the case in the example above.

\noindent 2. We shall see that each instance of the failure of the standard module conjecture and generalized injectivity conjecture for metaplectic groups  originates essentially from the situation of this example.
\end{rem}

\subsection{$\psi$-generic representations of metaplectic groups}

We are  now  ready to describe the form of $\psi$-generic representations of $\widetilde{Sp(2n,F)}.$

In his paper \cite{Atobe_gen}, Theorem 3.11, Atobe gives a condition characterizing the L-parameter of a $\psi$-generic Langlands quotient in terms of the L-function attached to this L-parameter. He obtained his results using Szpruch's theory of local coefficients for the metaplectic group, similarly to the use of those coefficients in this section. We reprove this result, but we express it  in terms of the conditions on the inducing representation for  this Langlands quotient and related generalized rank one reducibilities. This suites our needs when calculating theta lifts in the rest of the paper. This form of the condition for $\psi$-genericity (and our proof) is also easy to apply to answer a question about preservation of genericity under the mapping $\Theta_{\psi},$ cf. Corollary \ref{cor_transfer_gen}.

We also completely characterize the reducibility of a standard $\psi$-generic representation with a $\psi$--generic Langlands quotient, thus completing Atobe's result in  Theorem 3.13 of  \cite{Atobe_gen}, which only gives a sufficient condition. This is done in Theorem \ref{standard_module_reducibility}.
\begin{thm} 
\label{generic_reps_metaplectic}
Let  $\delta_i,\;i=1,2,\ldots,k$ be irreducible unitarizable square integrable representations of $GL(n_i,F),$ $\sigma$ a tempered $\psi$--generic representation  of $\widetilde{Sp(2n_0,F)}$ and $\overrightarrow{s_0}=(s_1,s_2,\ldots,s_k)$ satisfy $s_1\ge  s_2\ge \cdots\ge s_k>0$ so that for $\overrightarrow{s}=\overrightarrow{s_0}$ the representation
\[\gamma_{\psi}^{-1}\delta_1\nu^{s_1}\times \gamma_{\psi}^{-1}\delta_2\nu^{s_2}\times \cdots\times \gamma_{\psi}^{-1}\delta_k\nu^{s_k}\rtimes \sigma\]
 is standard. Then, the Langlands quotient $L(\gamma_{\psi}^{-1}\delta_1\nu^{s_1},\dotsc,\gamma_{\psi}^{-1}\delta_k\nu^{s_k};\sigma)$ is $\psi$--generic if and only if for each $i,j\in\{1,2,\ldots,k\},\;i\neq j$
 \begin{equation}
\label{GL_conditions}
 \delta_i\nu^{s_i}\times \delta_j\nu^{s_j} \text{ and } \delta_i\nu^{s_i}\times \widetilde{\delta_j}\nu^{-s_j}
 \end{equation}
 is irreducible and the following holds:
 \begin{itemize}
\item[(i)] if, for    $i\in\{1,2,\ldots,k\},$ $\delta_i\nu^{s_i}$ is not equal to $\delta([\nu^{1/2},\nu^{m+1/2}])$ for some $m\in \N_0,$ (actually for $2s_i-1$) then $\gamma_{\psi}^{-1}\delta_i\nu^{s_i}\rtimes \sigma$ is irreducible.
\item [(ii)] if, for  $i\in\{1,2,\ldots,k\},$ $\delta_i\nu^{s_i}=\delta([\nu^{1/2},\nu^{m+1/2}])$ for some $m\in \N_0,$ then, for each even $a$ such that $1\otimes S_a$ occurs in the L-parameter of $\sigma,$ we have $a\ge 4s_i.$
\end{itemize}
Note that because of \eqref{GL_conditions}, there can be at most one occurrence of a representation of the form  $\delta([\nu^{1/2},\nu^{m+1/2}])$ among $\delta_1\nu^{s_1},\ldots,\delta_k\nu^{s_k}.$
\end{thm}
Before proving the theorem, we state a useful lemma based on Kudla's filtration (we use the notation introduced at the end of Section \ref{subs:thetaintro}).
\begin{lem}
\label{cor5.3}
Let $\sigma$ be an irreducible representation of $\widetilde{Sp(2n,F)}$ or $O(V)$. Let $\delta$ be an irreducible essentially discrete series representation of $GL(k,F)$ such that $\delta \ncong St_k\nu^\frac{l-k}{2}$. If $\pi$ is an irreducible representation such that $\delta \rtimes \sigma \twoheadrightarrow \pi$, then $\delta \rtimes \Theta_l(\sigma) \twoheadrightarrow \Theta_l(\pi)$.
\end{lem}
\begin{proof}
This is Corollary 5.3 of \cite{Atobe_Gan}. Note that the original statement requires $l>0$, but it can easily be extended to any $l$ by examining the isotypic components in Kudla's filtration of the Jacquet modules of the Weil representation, as in the proof of Lemma 3.3 of \cite{Hanzer_R_metaplectic}.
\end{proof}

\begin{proof} (of Theorem) By grouping equal $s_i$'s together, we can write down the standard representation $\gamma_{\psi}^{-1}\delta_1\nu^{s_1}\times \gamma_{\psi}^{-1}\delta_2\nu^{s_2}\times \cdots\times \gamma_{\psi}^{-1}\delta_k\nu^{s_k}\rtimes \sigma$ as 
\[\Pi:=\tau_1\nu^{s_1'}\times \tau_2\nu^{s_2'}\times \cdots \tau_l\nu^{s_l'}\rtimes \sigma,\]
where $s_1'>s_2'>\cdots >s_l',\;\{s_1',s_2,'\ldots,s_l'\}\subset \{s_1,s_2,\ldots,s_k\},\;l\le k$ and $\tau_i$ are irreducible tempered representations. Now,  the long intertwining operator  $A_{w_0}$ (cf.~ Corollary \ref{generic_Langlands_quotient}) attached to representation $\Pi$ is holomorphic and non-zero. It is,
in a well-known way (e.g. \cite{Shahidi_certain}), composed from the (again holomorphic and non-zero) intertwining operators induced from the ones in the generalized rank-one case
\begin{align}
\label{rank_one_operators}
&\gamma_{\psi}^{-1}(\tau_i\nu^{s_i'}\times \tau_j\nu^{s_j'})\to \gamma_{\psi}^{-1}(\tau_j\nu^{s_j'}\times \tau_i\nu^{s_i'}),\\ \notag
&\gamma_{\psi}^{-1}(\tau_j\nu^{s_j'}\times \widetilde{\tau_i}\nu^{-s_i'})\to  \gamma_{\psi}^{-1}(\widetilde{\tau_i}\nu^{-s_i'}\times \tau_j\nu^{s_j'})\\ \
&\gamma_{\psi}^{-1}\tau_i\nu^{s_i'}\rtimes \sigma \to \widetilde{\tau_i}\nu^{-s_i'}\rtimes \sigma, \notag
\end{align}
where $i<j.$ The image of the operator $A_{w_0}$ is exactly the Langlands quotient of $\Pi,$ which we assumed to be $\psi$-generic.
Thus, each subquotient of the  kernel of $A_{w_0}$ is degenerate. Then, it follows that all the  generalized rank one operators in \eqref{rank_one_operators} have to have degenerate kernels. They are standard intertwining operators whose images are corresponding Langlands quotients, which are, thus, $\psi$--generic. In the first two cases in \eqref{rank_one_operators} we are, essentially, in the general linear group case, in which the standard module conjecture holds. This means that $\gamma_{\psi}^{-1}(\tau_i\nu^{s_i'}\times \tau_j\nu^{s_j'})$ and $\gamma_{\psi}^{-1}(\tau_j\nu^{s_j'}\times \widetilde{\tau_i}\nu^{-s_i'})$ have to be irreducible. From the Corollary \ref{generic_Langlands_quotient} applied to the third case of \eqref{rank_one_operators}, it follows that 
\[C_{\psi}^{\widetilde{Sp(2m,F)}}(\widetilde{P_i},s,\tau_i\nu^{s_i'}\otimes \sigma,w_0)\]
is holomorphic at $s=0$ (and we already know that it is non-zero). From the general linear case we conclude (in the language of our original data):
if $s_i>s_j,$ then $\delta_i\nu^{s_i}\times \delta_j\nu^{s_j}$ is irreducible and $\delta_i\nu^{s_i}\times \widetilde{\delta_j}\nu^{-s_j}$ is irreducible. If $s_i=s_j$  for some $i<j$ we have: from the irreducibility of the unitary induction in the $GL$--case, we know that $\delta_i\nu^{s_i}\times \delta_j\nu^{s_i}$ is irreducible. We just have to prove that $\delta_j\nu^{s_i}\times \widetilde{\delta_i}\nu^{-s_i}$ is irreducible if $i<j.$

Now we use \eqref{C_factor_gamma_factors} to study the holomorphicity of  $C_{\psi}^{\widetilde{Sp(2m,F)}}(\widetilde{P_i},s,\tau_i\nu^{s_i'}\otimes \sigma,w_0)$ at $s=0.$ Let $\tau_i=\delta_1\times \cdots\times \delta_r$ where $\delta_i$'s are square-integrable. We have 
\begin{align}
\label{C_tempered}
&C_{\psi}^{\widetilde{Sp(2m,F)}}(\widetilde{P_i},s,\tau_i\nu^{s_i'}\otimes \sigma,w_0)= \epsilon(s)\\ \notag
&\prod_{j=1}^r\left (\prod_{k=1}^{j-1}\gamma(\delta_k\nu^{s_i'}\times \widetilde{\delta_j}\nu^{-s_i'},s,\psi)\right )\\
&\frac{\gamma(\delta_j\nu^{s_i'},sym^2,2s,\psi)}{\gamma(\delta_j\nu^{s_i'},s+ \frac{1}{2},\psi)}\gamma(\sigma\times \delta_j\nu^{s_i'},s,\psi).\notag
\end{align}
Note that  the  expression $\frac{\gamma(\delta_j\nu^{s_i'},sym^2,2s,\psi)}{\gamma(\delta_j\nu^{s_i'},s+ \frac{1}{2},\psi)}\gamma(\sigma\times \delta_j\nu^{s_i'},s,\psi)$ is, up to an exponential factor,  the local coefficient attached to a representation $\gamma_{\psi}^{-1}\delta_j\nu^{s_i'}\rtimes \sigma$ (and the longest Weyl group element). Thus, we know that this local coefficient is not zero (as in  Corollary \ref{generic_Langlands_quotient}).
This means that, in order for $C_{\psi}^{\widetilde{Sp(2m,F)}}(\widetilde{P_i},s,\tau_i\nu^{s_i'}\otimes \sigma,w_0)$ to be holomorphic, the expression
\[\prod_{j=1}^r\left (\prod_{k=1}^{j-1}\gamma(\delta_k\nu^{s_i'}\times \widetilde{\delta_j}\nu^{-s_i'},s,\psi)\right )\]
has to be holomorphic. But this means that $\delta_k\nu^{s_i'}\times \widetilde{\delta_j}\nu^{-s_i'}$ has to be irreducible, since these Rankin-Selberg $\gamma$-factors are essentially the local factors (and are non-zero). This finishes the proof of \eqref{GL_conditions}.

The next step is to examine the consequences of the holomorphy at $s=0$ of
\[\frac{\gamma(\delta_j\nu^{s_i'},sym^2,2s,\psi)}{\gamma(\delta_j\nu^{s_i'},s+ \frac{1}{2},\psi)}\gamma(\sigma\times \delta_j\nu^{s_i'},s,\psi)\]
for each $i=1,2,\ldots,r.$
Note that $\gamma(\delta_j\nu^{s_i'},sym^2,2s,\psi)\gamma(\sigma\times \delta_j\nu^{s_i'},s,\psi)$ is, using Corollary \ref{transfer_gamma_factors}, up to an exponential factor, exactly the local coefficient attached to the  (standard) representation $\delta_j\nu^{s_i'}\rtimes \Theta_{\psi}(\sigma).$ Thus, it is non-zero. Note that, when we expand 
\[\frac{1}{\gamma(\delta_j\nu^{s_i'},s+ \frac{1}{2},\psi)}\]
 using L-functions, we get
 \[\frac{L(s+ \frac{1}{2},\delta_j\nu^{s_i'})}{L(\frac{1}{2}-s,\widetilde{\delta_j}\nu^{-s_i'})}\]
which is always holomorphic for $s=0$ by the well-known properties of Rankin-Selberg L-functions for   tempered representations (\cite{Sh1}).
It may have a zero if $L(\frac{1}{2}-s,\widetilde{\delta_j}\nu^{-s_i'})$ has a pole for $s=0.$ By expanding as in \cite{Shahidi_Twisted_endoscopy}, we have that $L(\frac{1}{2},\widetilde{\delta_j}\nu^{-s_i'})=L(\frac{1}{2}-s_i'+\frac{a-1}{2},\widetilde{\rho_j})$ if $\delta_j=\delta([\rho_j\nu^{-\frac{a-1}{2}},\rho_j\nu^{\frac{a-1}{2}}]).$ We get that a pole occurs only if $\rho_j=1_{GL_1}$ and 
$\delta_j\nu^{s_i'}=\delta([\nu^{\frac{1}{2}},\nu^{2s_i'-\frac{1}{2}}])=St_{2s_i'}\nu^{s_i'}.$ Thus, we have two cases.

\smallskip

\noindent Assume $\delta_j\nu^{s_i'}\neq \delta([\nu^{\frac{1}{2}},\nu^{2s_i'-\frac{1}{2}}]).$
Then, $\gamma(\delta_j\nu^{s_i'},sym^2,2s,\psi)\gamma(\Theta_{\psi}(\sigma)\times \delta_j\nu^{s_i'},\allowbreak s,\psi)$ is holomorphic, thus by Corollary \ref{generic_Langlands_quotient} (since the same argument as there holds in the case of odd orthogonal groups), the Langlands quotient $L(\delta_j\nu^{s_i'};\Theta_{\psi}(\sigma))$ is generic. But, since the standard module conjecture holds for odd orthogonal groups, this means that
\[
\label{eq:singlestar}\tag{*}
\delta_j\nu^{s_i'}\rtimes \Theta_{\psi}(\sigma) \text{ is irreducible}.
\]
However, $\gamma_{\psi}^{-1}\delta_j\nu^{s_i'}\rtimes \sigma$  is also irreducible. Indeed,  let $\Theta_{\psi}(\sigma)^{\epsilon}$ be the unique extension of $\Theta_{\psi}(\sigma)$ from $SO(2n_0+1,F)$ to 
$O(2n_0+1,F)$ which participates in $\psi$--theta correspondence with $\widetilde{Sp(2n_0,F)}.$ Then, 
$\Theta(\Theta_{\psi}(\sigma)^{\epsilon},2n_0)=\sigma.$ We now use Lemma \ref{cor5.3} with $l=0$: since $\delta_j\nu^{s_i'}\rtimes \Theta_{\psi}(\sigma)$ is irreducible, we have
\[\delta_j\nu^{-s_i'}\rtimes \Theta_{\psi}(\sigma)^{\epsilon}\twoheadrightarrow L(\delta_j\nu^{s_i'};\Theta_{\psi}(\sigma)^{\epsilon}).\]
Since $\delta_j\nu^{-s_i'}\neq St_{2s_i'}\nu^{-s_i'},$ Lemma \ref{cor5.3} gives
\[
\begin{split}
    \gamma_{\psi}^{-1}\delta_j\nu^{-s_i'}\rtimes \sigma \twoheadrightarrow \Theta(L(\delta_j\nu^{s_i'}; \Theta_{\psi}(\sigma)^{\epsilon}),2n_0+2n_j)\\
    \twoheadrightarrow \theta(L(\delta_j\nu^{s_i'}; \Theta_{\psi}(\sigma)^{\epsilon}),2n_0+2n_j).
    \end{split}
    \]
Then, if $\Theta(L(\delta_j\nu^{s_i'}; \Theta_{\psi}(\sigma)^{\epsilon}),\allowbreak 2n_0+2n_j)\neq 0,$ we have $\theta(L(\delta_j\nu^{s_i'}; \Theta_{\psi}(\sigma)^{\epsilon}),2n_0+2n_j)=\Theta_{\psi}(L(\delta_j\nu^{s_i'}; \Theta_{\psi}(\sigma)).$ This is, by Theorem 1.3.(iii) of \cite{Gan_Savin_Metaplectic_2012}, exactly $L(\gamma_{\psi}^{-1}\delta_j\nu^{s_i'}; \sigma).$
Thus, we have
\[\gamma_{\psi}^{-1}\delta_j\nu^{-s_i'}\rtimes \sigma \twoheadrightarrow L(\gamma_{\psi}^{-1}\delta_j\nu^{s_i'}; \sigma)\]
 and the properties of Langlands classification guarantee that $\gamma_{\psi}^{-1}\delta_j\nu^{-s_i'}\rtimes \sigma $ is irreducible.
 We now comment on non-vanishing of $\Theta(L(\delta_j\nu^{s_i'}; \Theta_{\psi}(\sigma)^{\epsilon}),2n_0+2n_j).$ We want to show that 
 $L(\delta_j\nu^{s_i'}; \Theta_{\psi}(\sigma)^{\epsilon})$ is the right extension of $L(\delta_j\nu^{s_i'};\allowbreak \Theta_{\psi}(\sigma))$ which participates in theta correspondence. But this easily follows from the multiplicativity of epsilon factors, which govern the non-vanishing of the lift (cf.~the seventh section of \cite{Gan_Savin_Metaplectic_2012}).

 \smallskip

 \noindent Now we assume that $\delta_j\nu^{s_i'}=\delta([\nu^{\frac{1}{2}},\nu^{2s_i'-\frac{1}{2}}])$ (so, necessarily $2s_i'\in \N$).  We saw that, in this case, 
 \begin{itemize}
 \item $\frac{1}{\gamma(\delta_j\nu^{s_i'},s+ \frac{1}{2},\psi)}$ has a simple zero at $s=0$
 \item we easily get  (e.g.~by using  Proposition 8.1 of \cite{Shahidi_Twisted_endoscopy}) that $\gamma(\delta_j\nu^{s_i'},\allowbreak sym^2,2s,\psi)$ has a pole of the first order. 
 \end{itemize}
From this, we conclude that
 \begin{itemize}
 \item $\gamma(\sigma\times \delta_j\nu^{s_i'},s,\psi)$ is holomorphic and non-zero
 \item
 \label{reducibilty_SO}$\gamma(\delta_j\nu^{s_i'},sym^2,2s,\psi)\gamma(\sigma\times \delta_j\nu^{s_i'},s,\psi)$ has a pole; using arguments as in the previous case, this means that the Langlands quotient $L(\delta_j\nu^{s_i'};\Theta_{\psi}(\sigma))$ is  not generic; thus
 \[
 \label{eq:doublestar}\tag{**}
 \delta_j\nu^{s_i'}\rtimes \Theta_{\psi}(\sigma) \text{ reduces}.
 \]
 \end{itemize}
 We now study the condition of $\gamma(\sigma\times \delta_j\nu^{s_i'},s,\psi)=\gamma(\Theta_{\psi}(\sigma)\times \delta_j\nu^{s_i'},s,\psi)$
  being holomorphic. Let $\delta_1',\ldots,\delta_l'$  and $\sigma'$ denote the tempered support of $\Theta_{\psi}(\sigma),$ i.e.
  \[\Theta_{\psi}(\sigma)\hookrightarrow \delta_1'\times \cdots\times \delta_l'\rtimes \sigma'.\]
  Then, by the multiplicativity of gamma factors (\cite{Shahidi_multiplicativity}),
  \[
  \begin{split}
      \gamma(\Theta_{\psi}(\sigma)\times \delta_j,s+s_i',\psi)&=\\ \gamma(\sigma'\times \delta_j,s+s_i',\psi)&\prod_{k=1}^{l'}\gamma(\delta_j\times \delta_k',s+s_i',\psi)\gamma(\delta_j\times \widetilde{\delta_k'},s+s_i',\psi).
      \end{split}
      \]
  By Proposition 7.2 of \cite{Sh2}, each of the factors is non-zero (for $s=0$; recall $s_i'>0$) so for product to be holomorphic, each factor has to be holomorphic. Thus, $\delta_j\nu^{s_i'}\times \delta_k'$ and  $\delta_j\nu^{s_i'}\times \widetilde{\delta_k'}$ has to be irreducible. In terms of L-parameter of $\Theta_{\psi}(\sigma)$ (cf. \cite{Arthur_endoscopic}) we can reformulate  holomorphy as the following condition: if $1\otimes S_a$ appears in the L-parameter of $\Theta_{\psi}(\sigma)$  (and comes from  $\delta_1',\ldots,\delta_l'$ ) for some  even $a,$ then $a\ge 4s_i'.$  

  \smallskip

  \noindent Now we analyze the holomorphicity of $\gamma(\sigma'\times \delta_j,s+s_i',\psi),$ where $\sigma'$ is a  generic discrete series representation.  Either by using \eqref{embedding_discrete} and Shahidi's results on multiplicativity, or, directly with the L-parameter of $\sigma'$ (\cite{Arthur_endoscopic}, cf.~Section A4. of \cite{Atobe_Gan}) we easily  get that the holomorphicity of $\gamma(\sigma'\times \delta_j,s+s_i',\psi)$ forces each $a,$ such that $1\otimes S_a$ belongs to the L-parameter of $\sigma',$ to be greater of equal to $4s_i'.$
\end{proof}

In Theorem 1.3 (v) of \cite{Gan_Savin_Metaplectic_2012},  the authors claim that if $\pi$ is an irreducible, $\psi$--generic and tempered representation of $\widetilde{Sp(2n)},$ then $\Theta_{\psi}(\pi)$ is generic tempered representation of $SO(2n+1).$ The proof of Theorem \ref{generic_reps_metaplectic} enables us to explicitly describe when the lift of an irreducible, $\psi$--generic representation $\pi$ of $\widetilde{Sp(2n)}$ is generic.

\begin{cor}
\label{cor_transfer_gen}
Let $\pi=L(\gamma_{\psi}^{-1}\delta_1\nu^{s_1},\gamma_{\psi}^{-1}\delta_2\nu^{s_2},\ldots,\gamma_{\psi}^{-1}\delta_k\nu^{s_k};\sigma)$ be a $\psi$-generic, irreducible representation of $\widetilde{Sp(2n)}.$ Then, $\Theta_{\psi}(\pi)$ is generic if and only if 
none of the representations $\delta_1\nu^{s_1},\ldots, \delta_k\nu^{s_k}$ is of the form $\delta([\nu^{1/2},\allowbreak\nu^{m+1/2}])$ for some $m\in \N_0.$
\end{cor}
\begin{proof} By Theorem 1.3.~(iii) of \cite{Gan_Savin_Metaplectic_2012}, $\Theta_{\psi}(\pi)=L(\delta_1\nu^{s_1},\delta_2\nu^{s_2},\ldots,\delta_k\nu^{s_k};\Theta_{\psi}(\sigma)).$ Now the result directly follows using claims \eqref{eq:singlestar} and \eqref{eq:doublestar} from the proof of Theorem \ref{generic_reps_metaplectic}, coupled with the factorization of the long intertwining operator and the standard module conjecture for the odd orthogonal groups.
\end{proof}

Theorem \ref{generic_reps_metaplectic} characterizes standard $\psi$-generic representation which possess a $\psi$-generic Langlands quotient. However, we have yet to characterize the reducibility of such a standard representation. We do this in the following theorem.

\begin{thm}
\label{standard_module_reducibility}
Let $\sigma$ be an irreducible tempered $\psi$-generic representation of $\widetilde{Sp(2n,F)}$ such that $l(\sigma) = 0$; let $\phi$ denote its $L$-parameter. We set $a_0= \min\{a: a \text{ is even and }S_a \hookrightarrow \phi\}$. Assume that $s_i \geqslant 1$ is a half-integer (i.e., $2s_i \in \mathbb{Z}$). Then the following holds:
\[
\gamma_\psi^{-1} St_{2s_i}\nu^{s_i} \rtimes \sigma \text{ is }
\begin{cases}
\text{reducible}, &\quad \text{if }4s_i \lneqq a_0\\
\text{irreducible}, &\quad \text{if }4s_i = a_0.
\end{cases}
\]
\end{thm}
\noindent Note that this dictates the reducibility of the standard representation considered in Theorem \ref{generic_reps_metaplectic}. We have omitted the case $s_i =\frac{1}{2}$, as it is treated in Proposition \ref{prop_reducibility_1_2}. We postpone the proof of Theorem \ref{standard_module_reducibility} to Appendix A.

\section{Non-vanishing of theta lifts}
\label{sec:non-vanishing}
We want to find the dimensions of the first occurrences of a $\psi$--generic irreducible representation of $\widetilde{Sp(2n,F)}$ in a pair of orthogonal towers with respect to  $\psi$--theta correspondence. After the determination of the levels of the first occurrence, we want to determine completely how these (small) theta lifts look like, not only for the first occurrence, but also further up in the quadratic towers. 
For simplicity of notation, we first treat the (more complicated) case of the lifts to the pair of orthogonal towers attached to the trivial quadratic character (i.e., in notation of \cite{Atobe_Gan}, $\chi_V$ is trivial).

In addition to the notation introduced at the end of Section \ref{subs:thetaintro}, we use the following to denote the first occurrence indices (cf. \cite[Section 2.10]{Atobe_Gan}).
For an irreducible representation $\pi$ of  $\widetilde{Sp(2n,F)}$ let $m^{down}(\pi)$ denote the dimension at which $\pi$ appears in the theta correspondence with the tower where it first appears (between the two towers in a pair). This is the extension of the definition of $m^{down}(\pi)$ when $\pi$ is tempered, given in \cite{Atobe_Gan}. Analogously we define $m^{up}(\pi).$ Let $l(\pi)=2n+1-m^{down}(\pi);$ again this is the extension of the definition of $l(\pi)$ when $\pi$ is tempered; note that for tempered $\psi$-generic $\pi,$ $l(\pi)\in\{0,2\}$ (cf.~Theorem 4.1.~of \cite{Atobe_Gan}). The conservation conjecture then guarantees that $\pi$ appears at the going up tower for $l=-l(\pi)-2.$
Then we have the following 
\begin{prop} 
\label{l(pi)}
Let $\pi=L(\gamma_{\psi}^{-1}\delta_1\nu^{s_1},\gamma_{\psi}^{-1}\delta_2\nu^{s_2},\ldots,\gamma_{\psi}^{-1}\delta_k\nu^{s_k};\sigma)$ be a $\psi$--generic representation of $\widetilde{Sp(2n,F)}.$ Then we have one of the following:
\begin{enumerate}
\item Assume that none of $\delta_1\nu^{s_1},\ldots,\delta_k\nu^{s_k}$ is of the form $\delta([\nu^{1/2},\nu^{2s_i-1/2}]).$ Then, $l(\pi)=l(\sigma).$
\item Assume that there exists (a unique!, cf.~Theorem  \ref{generic_reps_metaplectic}) $i\in \{1,2,\ldots,k\}$ such that $\delta_i\nu^{s_i}=\delta([\nu^{1/2},\nu^{2s_i-1/2}]).$ Then,
\begin{itemize}
\item If $l(\sigma)=2,$ then $s_i=\frac{1}{2}$ and $l(\pi)=l(\sigma)=2.$
\item If $l(\sigma)=0,$  then $l(\pi)=l(\sigma)=0$ unless $s_i=\frac{1}{2}.$
\end{itemize}
\end{enumerate}
\end{prop}
\begin{proof}
We look at the lifts of $\pi$ on the going-up tower for $\sigma.$ Let $l=-l(\sigma),$ so that $\Theta_l(\sigma)=0.$
Then, by the the consecutive use of Lemma \ref{cor5.3}, we get that
\[\delta_1\nu^{s_1}\times \delta_2\nu^{s_2}\times \cdots\times\delta_k\nu^{s_k}\rtimes\Theta_l(\sigma)\twoheadrightarrow \Theta_l(\pi).\]
Note that we are not in the exceptional case of that Lemma: indeed, since $l\le 0,$ so that $\frac{l-t}{2}<0$ for each positive $t,$ then $\delta_i\nu^{s_i}\ncong St_t^{\frac{l-t}{2}}.$ This means that $\Theta_l(\pi)=0,$ so that, in that tower, $\pi$ appears later (or for the same $l$) than $\sigma.$ But it means that on the other tower it appears earlier than $\sigma,$ i.e., $l(\pi)\ge l(\sigma).$ Note that this holds for every  $\psi$--generic $\pi;$ also, this means that the going-up tower for $\pi$  is the same as for $\sigma.$
Now we argue similarly for the going-down tower; we look at the level $l=l(\sigma) +2.$ We  again get that $\Theta_l(\pi)=0,$ provided we are not in the exceptional case of Lemma \ref{cor5.3}. If we are not in the exceptional case, this gives $l(\pi)\le l(\sigma)$  and $l(\pi)=l(\sigma).$ Now we just want to see what are the exceptional cases. 
If $l(\sigma)=2$, these are: $\nu^{3/2},\;\nu^1St_2\;\nu^{1/2}St_3.$ Note that the first  and the third representation are not of the form $St_{2s_i}\nu^{s_i},$ thus by Theorem \ref{generic_reps_metaplectic},  in the standard representation corresponding to $\pi$ we can interchange $\nu^{3/2}$ by $\nu^{-3/2}$  and $\nu^{1/2}St_{3}$ by $\nu^{-1/2}St_{3}$ so that we avoid exceptional case of Lemma \ref{cor5.3}. On the other hand, in the second case we are in the situation where $s_i=1,$ and $\nu^1St_{2}$ is indeed an exceptional case of Theorem \ref{generic_reps_metaplectic}. But since $l(\sigma)=2,$ an irreducible summand $S_2$ appears in the L-parameter of of $\sigma$ (Theorem 4.1 of \cite{Atobe_Gan}), this contradicts genericity requirement $4s_i\le 2$ appearing in Theorem \ref{generic_reps_metaplectic}, thus this case cannot happen. We conclude that if $l(\sigma)=2,$ then $l(\pi)=2.$ Using Theorem \ref{generic_reps_metaplectic}, we see that the only representation of the form $\delta_i\nu^{s_i}=\delta([\nu^{1/2},\nu^{2s_i-1/2}])$ which can appear among
$\delta_1\nu^{s_1},\ldots,\delta_k\nu^{s_k}$ if $l(\sigma)=2$ (i.e., if $S_2$ appears in the L-parameter of $\sigma$) is $\nu^{1/2}.$

Similarly, the exceptional case if $l(\sigma)=0$ is only  $\delta_i\nu^{s_i}=\nu^{1/2}.$ 
\end{proof}
From Proposition \ref{l(pi)} it obvious that the reducibility of the representation $\gamma_{\psi}^{-1}\nu^{1/2}\rtimes \sigma$ plays an important role in the calculation of the theta lifts of $\pi,$ if $\nu^{1/2}$ appears among  $\delta_1\nu^{s_1},\ldots,\delta_k\nu^{s_k}.$

\begin{prop} 
\label{prop_reducibility_1_2}
Assume that $\sigma$ is a $\psi$--generic, irreducible, tempered representation of $\widetilde{Sp(2n,F)}.$ Then $\gamma_{\psi}^{-1}\nu^{1/2}\rtimes \sigma$ is reducible if and only if $l(\sigma)=0.$
\end{prop}
\begin{proof} We first assume that $l(\sigma)=0.$ Then, the representation $\delta([\nu^{-1/2},\nu^{1/2}])\allowbreak \rtimes \allowbreak \Theta_{0}(\sigma)$ reduces ($\sigma$ and $\Theta_{0}(\sigma)|_{SO(2n+1,F)}$ have the same $L$--parameter). Furthermore, we know that $\gamma_{\psi}^{-1}\delta([\nu^{-1/2},\nu^{1/2}])\rtimes \sigma$ is reducible by Proposition 3.4.~of \cite{Hanzer_R_metaplectic}. This means that the long intertwining operator 
\[A(w_0,s):\gamma_{\psi}^{-1}\delta([\nu^{-1/2},\nu^{1/2}])\nu^s\rtimes \sigma\to \gamma_{\psi}^{-1}\delta([\nu^{-1/2},\nu^{1/2}])\nu^{-s}\rtimes \sigma\]
is holomorphic at $s=0.$ Indeed, if $\sigma$ is square-integrable, this is an easy extension of Harish-Chandra R-groups theory to the metaplectic case. We comment below the case when $\sigma$ is  not square-integrable. Here $w_0$ is the longest Weyl group element of $\widetilde{Sp(2n+4,F)}$ modulo the longest one in the corresponding Levi subgroup. But, $A(w_0,s)$ is the restriction of the intertwining operator $B(w_0,s)$ acting on $\gamma_{\psi}^{-1}(\nu^{1/2}\times\nu^{-1/2})\rtimes \sigma,$ which  is composed of the following intertwining operators
\begin{gather*}
\gamma_{\psi}^{-1}\nu^{1/2+s}\times\gamma_{\psi}^{-1}\nu^{-1/2+s}\rtimes \sigma\xrightarrow{T_1(s)}\\
\gamma_{\psi}^{-1}\nu^{1/2+s}\times\gamma_{\psi}^{-1}\nu^{1/2-s}\rtimes \sigma\xrightarrow{T_2(s)}\\
\gamma_{\psi}^{-1}\nu^{1/2-s}\times\gamma_{\psi}^{-1}\nu^{1/2+s}\rtimes \sigma\xrightarrow{T_3(s)}\\
\gamma_{\psi}^{-1}\nu^{1/2-s}\times\gamma_{\psi}^{-1}\nu^{-1/2-s}\rtimes \sigma.
\end{gather*}
Now we assume that $\gamma_{\psi}^{-1}\nu^{1/2}\rtimes \sigma$ is irreducible. Then, maybe after multiplying with a certain non-negative power of $s,$ $T_1(s),T_2(s),T_3(s)$ are all holomorphic isomorphisms. Thus, there exists $k\ge 0$ such that $\lim_{s\to 0}s^kB(w_0,s)$ is holomorphic isomorphism. Note that $k$ is actually strictly greater than $0$ since $T_2(s)$ certainly has a pole for $s=0.$ This means that $A(w_0,s)$ has a pole for $s=0,$ a contradiction. Thus, $\gamma_{\psi}^{-1}\nu^{1/2}\rtimes \sigma$ is reducible.\\
We now comment on holomorphy of $A(w_0,s)$ when $\sigma$ is non-square integrable. Let
\begin{equation}
\label{embedding_sigma}
\sigma\hookrightarrow \gamma_{\psi}^{-1}\delta_1\times \cdots \times \gamma_{\psi}^{-1}\delta_r\rtimes \sigma_0,
\end{equation}
where $\delta_i,\;i=1,2,\ldots,r$ and $\sigma_0$ are square-integrable. Then, $A(w_0,s)$ is a restriction of an operator which is composed of $GL$--induced operators
\[\gamma_{\psi}^{-1}\delta([\nu^{-1/2},\nu^{1/2}])\nu^s\times \gamma_{\psi}^{-1}\delta_i\to \gamma_{\psi}^{-1}\delta_i\times \gamma_{\psi}^{-1}\delta([\nu^{-1/2},\nu^{1/2}])\nu^s\]
and
\[\gamma_{\psi}^{-1}\delta_i\times \gamma_{\psi}^{-1}\delta([\nu^{-1/2},\nu^{1/2}])\nu^{-s}\to  \gamma_{\psi}^{-1}\delta([\nu^{-1/2},\nu^{1/2}])\nu^{-s}\times \gamma_{\psi}^{-1}\delta_i ,\]
and the operator
\[\gamma_{\psi}^{-1}\delta([\nu^{-1/2},\nu^{1/2}])\nu^{s}\rtimes \sigma_0\to\gamma_{\psi}^{-1}\delta([\nu^{-1/2},\nu^{1/2}])\nu^{-s}\rtimes \sigma_0.\]
The last operator is holomorphic, by the discussion above. Since for all $i=1,2,\ldots,r$ we have $\delta_i\ncong \delta([\nu^{-1/2},\nu^{1/2}]),$ by the Harish-Chandra commuting algebra theorem, all the $GL$--induced operators are holomorphic.

\noindent Now we assume $l(\sigma)=2.$  We need the following Lemma.
\begin{lem}
\label{length_two}
Let  $\sigma$ be a tempered $\psi$-generic representation of $\widetilde{Sp(2n,F)}$ whose L-parameter contains $S_2$ 
with  positive multiplicity. Then
\begin{enumerate}
\item Each (possible) irreducible subquotient of $\gamma_{\psi}^{-1}\nu^{1/2}\rtimes \sigma$ and of  $\nu^{1/2}\rtimes \Theta_{0}(\sigma)$ different from their Langlands quotient is tempered.
\item The representation $\nu^{1/2}\rtimes \Theta_{0}(\sigma)$  is of length two.
\end{enumerate}
\end{lem}
We postpone the proof of this Lemma after the proof of this Proposition.
Now we can finish the proof of Proposition. We recall that we now assume that the multiplicity in which $S_2$ appears in the L-parameter of $\sigma$ is positive (i.e.~$l(\sigma)=2$). Assume that $\gamma_{\psi}^{-1}\nu^{1/2}\rtimes \sigma$ is reducible. Thus, 
there is an irreducible, and according to Lemma  \ref{length_two}, tempered representation $T$ such that $\gamma_{\psi}^{-1}\nu^{-1/2}\rtimes \sigma\twoheadrightarrow T.$ We  prove that  $T$ cannot have a non-zero lift on the split orthogonal tower for $l=0.$ Thus, the lift on the non-split tower for $l=0$ should be non-zero, but, in this case, we also get a contradiction. So, assume firstly that $\Theta_{0}(T)$ is non-zero on the split tower. Note that necessarily $\Theta_{0}(T)$ is irreducible and tempered (\cite{Gan_Savin_Metaplectic_2012}, Theorem 8.1). Then, we either have
\begin{equation}
\label{epi}
\nu^{-1/2}\rtimes \Theta_{0}(\sigma)\twoheadrightarrow \Theta_{0}(T)
\end{equation}
or
\begin{equation}
\label{epi1}
\theta_{-2}(\sigma)=\Theta_{0}(T).
\end{equation}
The second possibility cannot occur, since by Theorem 4.3. of \cite{Atobe_Gan}, $\theta_{-2}(\sigma)$ is not tempered. Assume that the first possibility occurs. Then, by Lemma \ref{length_two}, $\Theta_{0}(T)$ is the generic tempered subquotient of $\nu^{-1/2}\rtimes \Theta_{0}(\sigma).$ But then $\Theta_{0}(\Theta_{0}(T))=T$ is $\psi$--generic (Theorem 9.1.~of \cite{Gan_Savin_Metaplectic_2012}), which is not true, since in $\gamma_{\psi}^{-1}\nu^{1/2}\rtimes \sigma$  its Langlands quotient is $\psi$--generic by Theorem \ref{generic_reps_metaplectic}. So, we conclude that $\Theta_{0}(T)$ has to be non-zero on the non-split tower. Now we again have \eqref{epi} or \eqref{epi1}, but on the non-split tower. But then neither \eqref{epi} nor \eqref{epi1} can hold, since both $\Theta_{0}(\sigma)$ and $\Theta_{-2}(\sigma)$ are zero on the non-split tower---$\sigma$ lifts to the non-split tower first on the level $l=-4$ as a consequence of the conservation relation. 
We conclude that $\gamma_{\psi}^{-1}\nu^{1/2}\rtimes \sigma$ is irreducible and we have finished the proof of this Proposition.
\end{proof}

\begin{proof} (of Lemma \ref{length_two}). For a representation $\tau$ of the metaplectic or odd-orthogonal group (full or special) we denote by $\mu^*(\tau)$ the semisimplification of the sum of Jacquet modules of $\tau$ over all maximal standard parabolic subgroups. There is a formula by Tadić (\cite{Tad_struc}) for classical groups, checked for metaplectic groups in \cite{Hanzer_Muic_metaplectic_Jacquet}, for calculation of $\mu^*$ for parabolically induced representations (e.g., Proposition 4.5 of \cite{Hanzer_Muic_metaplectic_Jacquet}).
Now, let $\pi\in \{\sigma,\Theta_{0}(\sigma)\},$ as we are going to address simultaneously the metaplectic and the orthogonal case. Assume that there exists a non-tempered irreducible subquotient $\pi_1$ of $\nu^{1/2}\rtimes \pi$ (if we are in the metaplectic case, we just add $\gamma_{\psi}^{-1}$ before a $GL$-representation; the same remark applies  for the expansion of $\mu^*(\nu^{1/2}\rtimes \pi)$ below). 
Then, directly form the Langlands classification, we get that there exist $l_1,l_2\in \R$ such that $l_1+l_2+1\in \N$ and $l_1-l_2<0,$ an irreducible cuspidal representation $\rho$ and  an irreducible representation $\pi_2$ of a smaller metaplectic or odd orthogonal group such that \[\pi_1\hookrightarrow \delta([\widetilde{\rho}\nu^{-l_2},\widetilde{\rho}\nu^{l_1}])\rtimes \pi_2.\]
Then 
\[\delta([\widetilde{\rho}\nu^{-l_2},\widetilde{\rho}\nu^{l_1}])\otimes \pi_2\le\mu^*(\pi_1)\le \mu^*(\nu^{1/2}\rtimes \pi).\]
Using the above mentioned formula for $ \mu^*(\nu^{1/2}\rtimes \pi),$ we get that
\[\delta([\widetilde{\rho}\nu^{-l_2},\widetilde{\rho}\nu^{l_1}])\otimes \pi_2\le \sum_{\delta'\otimes \sigma_1}(\nu^{-1/2}\times \delta'\otimes \sigma_1+\delta'\otimes \nu^{1/2}\rtimes \sigma_1 +\nu^{1/2}\times \delta'\otimes \sigma_1),\]
where $\mu^*(\pi)=\sum \delta'\otimes \sigma_1.$ Now we discuss all three summands above.
\begin{itemize}
\item $\delta([\widetilde{\rho}\nu^{-l_2},\widetilde{\rho}\nu^{l_1}])\le \nu^{-1/2}\times \delta',\;\pi_2=\sigma_1.$\\
Looking at the cuspidal support,  which has to be multiplicity free, and the fact that $\delta'$ has to be non-degenerate, we get that $\rho=1_{GL_1},$ and  $\delta'=\delta([\nu^{-l_2},\nu^{-3/2}])\times \delta([\nu^{1/2},\nu^{l_1}]).$ If the first factor is non-trivial, this contradicts the temperedness of $\pi.$  But then $-l_2\ge -\frac{1}{2}.$ Also, if the second factor exists, $l_1\ge \frac{1}{2}.$ This would contradict the requirement $l_1-l_2<0,$ thus, the second factor also does not exist (i.e., $\delta'=1$). Thus, $\delta([\widetilde{\rho}\nu^{-l_2},\widetilde{\rho}\nu^{l_1}])=\nu^{-\frac{1}{2}}$ and $\sigma_1=\pi.$ This gives that $\pi_1$ is actually the Langlands quotient of $\nu^{1/2}\rtimes \pi.$
\item $\delta([\widetilde{\rho}\nu^{-l_2},\widetilde{\rho}\nu^{l_1}])=\delta',\;\pi_2\le \nu^{1/2}\rtimes \sigma_1.$ This cannot happen, since this would contradict the temperedness of $\pi.$
\item $\delta([\widetilde{\rho}\nu^{-l_2},\widetilde{\rho}\nu^{l_1}])\le \nu^{1/2}\times \delta',\;\pi_2=\sigma_1.$ We have $\delta'=\delta([\nu^{-l_2},\nu^{-1/2}])\times \delta([\nu^{3/2},\nu^{l_1}]).$ The existence of the first factor would violate the temperedness of $\pi,$ so $-l_2\ge 1/2.$ The existence of the second factor would mean that $l_1\ge 3/2$ so that $-l_2+l_1\ge 2.$ Thus, the second factor does not exist, $\delta'=1,\;\sigma_1=\pi$ and $\delta([\widetilde{\rho}\nu^{-l_2},\widetilde{\rho}\nu^{l_1}])=\nu^{1/2},$ a contradiction.
\end{itemize}
We have proved the first part of this Lemma. As for the second part, we  prove a slightly more general

\noindent {\textbf{Claim}} Let $\pi_{00}$ be an irreducible, generic and tempered representation of $SO(2n+1,F)$ (or $O(2n+1,F)$). Then, the representation  $ \nu^{1/2}\rtimes \pi_{00}$ is of length two.\\
We now prove the {\textbf{Claim}}. We know that $\nu^{1/2}\rtimes \pi_{00}$ is reducible; namely the irreducible generic subqotient, say $T_1,$ is a tempered subrepresentation; cf.~Propositions 4.5 and 4.8 of \cite{Hanzer_injectivity}. 
We first prove the {\textbf{Claim}} in the situation in which $\pi_{00}$ is square-integrable. Assume firstly that $S_2$  does not appear in the L-parameter of $\pi_{00}.$ Then, we know by \cite{Hanzer_injectivity} that $\nu^{1/2}\rtimes \pi_{00}$ has a square-integrable irreducible generic subrepresentation. By the first part of the proof, all other irreducible subquotients of $\nu^{1/2}\rtimes \pi_{00},$ other than the Langlands quotient, are tempered. More precisely, they are all square-integrable, since they share the common cuspidal support. So, the  maximal proper subspace of $\nu^{1/2}\rtimes \pi_{00}$  is a tempered representation of finite length, with each subquotient square-integrable. Because of the projectivity of discrete series in the category of finite length tempered representations, paired with the fact that  the multiplicity of $\nu^{1/2}\otimes \pi_{00}$  in the appropriate Jacquet module of 
$\nu^{1/2}\rtimes \pi_{00}$ equals one (one can readily confirm this using Tadić's formula, since $S_2$ does not appear in the L-parameter of $\pi_{00}$), we get the claim.  If, on the other hand, $S_2$ appears in the L-parameter of $\pi_{00},$ then there exists an irreducible generic and square-integrable representation $\xi$ such that $\pi_{00}\hookrightarrow\nu^{1/2}\rtimes \xi.$ Let $T$ be any irreducible tempered subquotient of $\nu^{1/2}\rtimes \pi_{00}.$ Then $T\le \nu^{1/2}\times \nu^{-1/2}\rtimes \xi.$ So we either have $T\le \zeta(-1/2,1/2)\rtimes \xi,$ or $T\le \delta([\nu^{-1/2},\nu^{1/2}])\rtimes \xi.$ Both of these induced representations are actually unitarizable, so $T$ is their subrepresentation. Then the first possibility cannot occur because of Casselman temperedness  criterion. So, each tempered subquotient $T$ is embedded in $\delta([\nu^{-1/2},\nu^{1/2}])\rtimes \xi.$ But when we calculate the multiplicity of $\delta([\nu^{-1/2},\nu^{1/2}])\otimes \xi$ in the appropriate Jacquet module of  $\nu^{1/2}\rtimes \pi_{00},$ we see that it is equal to one; thus only one such $T$ can occur. 
	
\noindent Now we turn to the case in which $\pi_{00}$ is tempered (and generic). Let 
\[\pi_{00}\hookrightarrow \pi_1\times \pi_2\times \cdots \times \pi_k\rtimes\xi_0\]
be the embbeding of $\pi_{00}$ in its tempered support. Then, we also have
\[\nu^{1/2}\rtimes \pi_{00}\hookrightarrow \Pi:=\pi_1\times \pi_2\times \cdots \times \pi_k\times \nu^{1/2}\rtimes\xi_0.\]
The representation $\nu^{1/2}\rtimes\xi_0,$ by the previous cases, has the unique irreducible tempered (or square-integrable) generic subquotient (which is a subrepresentation), say $\xi_{00}.$ Similarly as above, we conclude than any tempered subquotient $T$ of $\nu^{1/2}\rtimes \pi_{00}$ is embedded as $T\hookrightarrow \Pi_1=:\pi_1\times \pi_2\times \cdots \times \pi_k  \rtimes \xi_{00}.$  Note that any tempered subquotient of $\Pi$ is actually a subrepresentation of $\Pi_1,$ which has the multiplicity one property.  This means that the maximal proper subrepresentation of $\nu^{1/2}\rtimes \pi_{00}$ (which is generated by all the irreducible tempered subquotients) is actually a subrepresentation of $\Pi_1,$ so it is semi-simple, thus the socle of $\nu^{1/2}\rtimes \pi_{00}.$ If, additionally, $S_2$ does not appear in the L-parameter of $\pi_{00},$ we get that the multiplicity of  $\nu^{1/2}\otimes \pi_{00}$ in the appropriate Jacquet module of  $\nu^{1/2}\rtimes \pi_{00}$ is one, meaning that there is only one irreducible subrepresentation occurring in $\nu^{1/2}\rtimes \pi_{00}.$ We now address the cases in which $\pi_{00}$ is tempered and  $S_2$ does appear in its L-parameter. We argue by induction, separately on the parity of the appearance of $S_2$ in the L-parameter. We first address the even parity.

 \noindent Assume that $S_2$ appears in the L-parameter of $\pi_{00}$ with multiplicity $2h,\;h\ge 1.$ First, assume that $h=1.$ Then, $\pi_{00}\hookrightarrow \delta([\nu^{-1/2},\nu^{1/2}])\rtimes \pi_0,$ where $\pi_0$ is tempered, generic and $S_2$ does not appear in its L-parameter. Similarly as above, we get that each irreducible tempered subquotient $T$ of $\nu^{1/2}\rtimes \pi_{00}$ (which is actually a subrepresentation there, by the previous reasoning) is embedded in $\delta([\nu^{-1/2},\nu^{1/2}])\rtimes \zeta_0,$ where $\zeta_0$ is the unique subquotient of $\nu^{1/2}\rtimes \pi_0$ different from its Langlands quotient. But $\zeta_0$ is generic and has  $S_2$ in its Langlands parameter. This means that  $\delta([\nu^{-1/2},\nu^{1/2}])\rtimes \zeta_0$ is irreducible and generic, so $T$ is the unique irreducible generic subquotient of $\nu^{1/2}\rtimes \pi_0$ and thus appears with the multiplicity one there; we are finished with the basis of the induction. Now the induction step is performed virtually in the same fashion as the basis of the induction procedure, where we now assume that $S_2$ appears in the L-parameter of $\pi_{00}$ with multiplicity $2h+2$ so that it appears for $\zeta_0$ with multiplicity $2h.$ Now we finish the case with odd multiplicities. As a basis, we consider the case where $S_2$ appears in the L-parameter of $\pi_{00}$ with multiplicity one. The proof in this case is similar to the case in which $\pi_{00}$ is square-integrable and $S_2$ appears in the L-parameter of $\pi_{00}$ with multiplicity one, discussed above. The induction step is proved in the same way as in the case of even multiplicities.
\end{proof}
\begin{cor}
\label{cor:stdmod_l=2}Let $\pi=L(\gamma_{\psi}^{-1}\delta_1\nu^{s_1},\gamma_{\psi}^{-1}\delta_2\nu^{s_2},\ldots,\gamma_{\psi}^{-1}\delta_k\nu^{s_k};\sigma)$ be a $\psi$-generic, irreducible representation of $\widetilde{Sp(2n)}.$ If $l(\sigma)=2,$ the standard module conjecture holds for this representation, i.e., $\gamma_{\psi}^{-1}\delta_1\nu^{s_1}\times \gamma_{\psi}^{-1}\delta_2\nu^{s_2}\times \cdots\times \gamma_{\psi}^{-1}\delta_k\nu^{s_k}\rtimes \sigma$ is irreducible.
\end{cor}
\begin{proof}This follows from Theorem \ref{generic_reps_metaplectic}, especially on  requirements on the form of factors $\delta([\nu^{1/2},\nu^{m+1/2}]),$ Proposition \ref{l(pi)} 2) and Proposition \ref{prop_reducibility_1_2}.
\end{proof}

\section{Theta lifts of generic representations}
\label{sec:lifts}
We are now ready to determine the theta lifts of those generic representations which satisfy the standard module conjecture. By Corollary \ref{cor:stdmod_l=2} and Theorem \ref{generic_reps_metaplectic}, these include all the generic representations $\pi=L(\gamma_{\psi}^{-1}\delta_1\nu^{s_1},\gamma_{\psi}^{-1}\delta_2\nu^{s_2},\ldots,\gamma_{\psi}^{-1}\delta_k\nu^{s_k};\sigma)$ for which $l(\sigma)=2$, as well as those for which $l(\sigma)=0$ and no $\delta_i\nu^{s_i}$ is of the form $\delta([\nu^{1/2},\nu^{m+1/2}])$. The following proposition describes the lifts of such representations.
We note that the remaining generic representations, i.e., those with $l(\sigma)=0$ and $\delta_i\nu^{s_i} \cong \delta([\nu^{1/2},\nu^{m+1/2}])$ for some $i$, may or may not be irreducible (see Theorem \ref{standard_module_reducibility}). Their lifts are treated in Section \ref{sec:lifts2}.
\begin{prop}
\label{prop:lifts}
Let $\pi$ be an irreducible generic representation which satisfies the standard module conjecture, i.e.
\[
\pi = \gamma_{\psi}^{-1}\delta_1\nu^{s_1} \times \gamma_{\psi}^{-1}\delta_2\nu^{s_2} \times \dotsm \times \gamma_{\psi}^{-1}\delta_k\nu^{s_k} \rtimes \sigma.
\]
Let $l$ be an even integer such that $\theta_{l}(\pi) \neq 0$. If $l=2$, assume $\delta_i\nu^{s_i} \neq \nu^\frac{1}{2}$, $\forall i$. Then
\[
\delta_r\nu^{s_r} \times \dotsm \times \delta_1\nu^{s_1} \rtimes \theta_{l}(\sigma) \twoheadrightarrow \theta_{l}(\pi).
\]
Furthermore, if $\theta_{l}(\sigma) = L(\delta_k'\nu^{s_k'}, \dotsc \times \delta_1'\nu^{s_1'}; \tau)$, then $\theta_{l}(\pi)$ is uniquely determined by
\[
\theta_{l}(\pi) = L(\delta_1\nu^{s_1}, \dotsc, \delta_k\nu^{s_k}, \delta_1'\nu^{s_1'}, \dotsc, \delta_k'\nu^{s_k'}; \tau).
\]
Note that we use the notation for the Langlands quotient loosely, i.e., we assume that the $\{\delta_1\nu^{s_1},\delta_2\nu^{s_2},\ldots,\delta_{k}\nu^{s_{k}}\}$ are permuted among $\{\delta_1'\nu^{s_1'}, \dotsc, \delta_k'\nu^{s_k'}\}$ so that they appear in descending order.

In the exceptional case when $l=2$ and there is an index $i$ (unique, by Theorem \ref{generic_reps_metaplectic}) such that $\delta_i\nu^{s_i} = \nu^\frac{1}{2}$, we may assume $i=1$. We then have
\[
\theta_{2}(\pi) = L(\delta_2\nu^{s_2}, \dotsc, \delta_k\nu^{s_k}; \theta_0(\sigma)).
\]
\end{prop}
The proof is distributed among the following two sections. We begin by considering the first non-zero lifts.

\subsection{The first lifts of a $\psi$-generic representation which satisfies the standard module conjecture}
\label{subs:std_mod_lifts}
We have two basic cases to consider: $l(\sigma) = 2$ and $l(\sigma) = 0$. We treat these cases separately. Recall that we also assume that none of the $\delta_i\nu^{s_i}$ is of the form $\delta([\nu^{1/2},\nu^{m+1/2}])$, $m\in \Z_{\ge 0}$ when $l(\sigma) = 0$ (we consider these exceptional representations in Section \ref{sec:lifts2}).

\medskip

\noindent\underline{Case 1}: $l(\sigma) = 0$.\\
We first look at the going-down tower. Proposition \ref{l(pi)} shows that $\Theta_{0}(\pi)$ is the first lift in this case. Repeatedly applying Lemma \ref{cor5.3} with $l=0$ like in Proposition \ref{l(pi)}, we get
\[
\delta_1\nu^{s_1}\times \delta_2\nu^{s_2}\times \cdots\times\delta_k\nu^{s_k}\rtimes\Theta_0(\sigma)\twoheadrightarrow \Theta_0(\pi) \twoheadrightarrow \theta_0(\pi).
\]
Since $\Theta_0(\sigma)$ is irreducible and tempered (by Proposition \ref{irreducibility_temp_lift}), this gives us the standard module for $\theta(\pi)$.

In the going-up tower, $\pi$ first appears when $l=-2$ (by the conservation relation). We use the same argument, this time with $l=-2$. We get
\[
\delta_1\nu^{s_1}\times \delta_2\nu^{s_2}\times \cdots\times\delta_k\nu^{s_k}\rtimes\Theta_{-2}(\sigma)\twoheadrightarrow \Theta_{-2}(\pi) \twoheadrightarrow \theta_{-2}(\pi).
\]
Note that $\Theta_{-2}(\sigma)$ is irreducible and tempered. Indeed, this follows easily from Lemma \ref{cor5.3}, using the tempered support of $\sigma$ and the irreducibility of $\Theta_{-2}(\pi_{00})$, where $\pi_{00}$ is the classical part of the tempered support of $\sigma$ (cf.~the remark after Proposition \ref{first_lifts}). Consequently, the left-hand side of the above map is the standard module for $\theta_{-2}(\pi)$.

\medskip

\noindent\underline{Case 2}: $l(\sigma) = 2$.\\
We employ the same strategy in this case. On the going-down tower we have $l(\pi) = 2$, by Proposition \ref{l(pi)}. To avoid technical difficulties, we initially assume that none of the $\delta_i\nu^{s_i}$ equals $\nu^\frac{1}{2}$. This eliminates the only exceptional case of Lemma \ref{cor5.3} for $l=2$, so we can use it to get
\[
\delta_1\nu^{s_1}\times \delta_2\nu^{s_2}\times \cdots\times\delta_k\nu^{s_k}\rtimes\Theta_2(\sigma)\twoheadrightarrow \Theta_2(\pi) \twoheadrightarrow \theta_2(\pi).
\]
In this case we do not know if $\Theta_2(\sigma)$ is irreducible, but we know that all of its subquotients are tempered \cite[Proposition 5.5]{Atobe_Gan}. We need to determine which subquotient $\tau$ of $\Theta_2(\sigma)$ participates in the above map. In fact, we will show that $\tau = \theta_2(\sigma)$.

To do this, we utilize Lemma \ref{cor5.3} in reverse direction to get
\[
\gamma_{\psi}^{-1}\delta_1\nu^{s_1}\times \gamma_{\psi}^{-1}\delta_2\nu^{s_2}\times \cdots\times \gamma_{\psi}^{-1}\delta_k\nu^{s_k}\rtimes \Theta_{-2}(\tau) \twoheadrightarrow \pi.
\]
If $\Theta_{-2}(\tau)$ is the first lift of $\tau$ to the tower of metaplectic groups, then we may again deduce that $\Theta_{-2}(\tau)$ is irreducible and tempered. The uniqueness of the standard module for $\pi$ now implies $\theta_{-2}(\tau) = \sigma$, which in turn gives us $\tau = \theta_2(\sigma)$. If $\Theta_{-2}(\tau)$ is not the first lift, i.e.~if $\Theta_{0}(\tau)$ is non-zero, then $\theta_{-2}(\tau)$ is not tempered, so we need a different argument.

In this case, one easily shows that $\Theta_{0}(\theta_2(\pi))$ is also non-zero. Thus, setting $l=0$ instead of $l=-2$ we get
\[
\gamma_{\psi}^{-1}\delta_1\nu^{s_1}\times \gamma_{\psi}^{-1}\delta_2\nu^{s_2}\times \cdots\times \gamma_{\psi}^{-1}\delta_k\nu^{s_k}\rtimes \Theta_{0}(\tau) \twoheadrightarrow \theta_{0}(\theta_2(\pi)).
\]
Since $\pi = \theta_{-2}(\theta_2(\pi))$ is a subquotient of $\gamma_{\psi}^{-1}\nu^{\frac{1}{2}}\rtimes\theta_{0}(\theta_2(\pi))$, heredity implies that $\theta_{0}(\theta_2(\pi))$ is generic. But then the above map shows that $\Theta_{0}(\tau)$ (which is irreducible and tempered) must also be generic, again by heredity. Using Theorem 1.3 (v) of \cite{Gan_Savin_Metaplectic_2012} we now get that $\tau$ is generic. Similarly, the fact that $\theta_{0}(\sigma)$ is a generic subquotient of $\nu^{-\frac{1}{2}} \rtimes \theta_2(\sigma)$ implies that $\theta_2(\sigma)$ is also generic.

We now know that $\theta_2(\sigma)$ and $\tau$ are both generic subquotients of $\Theta_2(\sigma)$. Since they both belong to the same $L$-packet (see Lemma 6.4 of \cite{Atobe_Gan}), this means that they can differ only by a twist of the determinant character. However, it is easy to show that this is not the case, and that $\theta_2(\sigma)$ and $\tau$ are in fact isomorphic. Indeed, this follows from the fact that both $\Theta_0(\tau)$ and $\Theta_0(\theta_2(\sigma))$ are non-zero.

It remains to comment on the exceptional case when one of the $\delta_i\nu^{s_i}$ equals $\nu^\frac{1}{2}$. Since the standard module of $\pi$ is assumed to be irreducible, $\nu^\frac{1}{2}$ appears at most once. Furthermore, we can rearrange the $\delta_i$'s in any order, so we may assume that $i = 1$. We thus have $\gamma_{\psi}^{-1}\nu^\frac{1}{2}\rtimes\pi' \cong \pi$ with $\pi' \cong \gamma_{\psi}^{-1}\delta_2\nu^{s_2}\times \cdots\times \gamma_{\psi}^{-1}\delta_k\nu^{s_k}\rtimes \sigma$. A simple application of Kudla's filtration (like the one used in Lemma \ref{cor5.3}) combined with the irreducibility of $\nu^\frac{1}{2} \rtimes \pi'$ now shows that we have
\[
\nu^\frac{1}{2} \rtimes \Theta_2(\pi') \twoheadrightarrow \Theta_2(\pi) \twoheadrightarrow \Theta_0(\pi') \twoheadrightarrow 0.
\]
This implies that $\theta_2(\pi)$ equals $\theta_0(\pi')$, which we have already determined in Case 1.

The first lift on the going-up tower appears when $l = -4$. As before, we have
\[
\delta_1\nu^{s_1}\times \delta_2\nu^{s_2}\times \cdots\times\delta_k\nu^{s_k}\rtimes\Theta_{-4}(\sigma)\twoheadrightarrow \Theta_{-4}(\pi) \twoheadrightarrow \theta_{-4}(\pi).
\]
We have two distinct subcases, depending on the multiplicity of $S_2$ in the parameter of $\sigma$. If the multiplicity is odd, then $\Theta_{-4}(\sigma)$ is the first non-zero lift on this tower, and is irreducible and tempered by the same argument already used for $\Theta_{-2}(\sigma)$ in Case 1. If the multiplicity is even, the lift $\theta_{-4}(\sigma)$ is non-tempered so we must proceed in a different manner. We treat this case in the following section.

In all cases of the above discussion, we have had (for a suitable value of $l$)
\[
\delta_1\nu^{s_1}\times \delta_2\nu^{s_2}\times \cdots\times\delta_k\nu^{s_k}\rtimes\theta_{l}(\sigma)\twoheadrightarrow \theta_{l}(\pi).
\]
Thus far, $\theta_l(\sigma)$ has been tempered, which allowed us to deduce that the left-hand side of the above map is the standard module for $\theta_{l}(\pi)$. In the next section, we show that the above map is still valid in all the remaining cases. Furthermore, we demonstrate how the standard module can be determined from this map despite the fact that $\theta_{l}(\sigma)$ is no longer tempered.

\subsection{The higher lifts of a $\psi$-generic representation which satisfies the standard module conjecture}
\label{subs:higher}
The previous section deals with the cases in which $l = 2, 0, -2$ and a part of the $l=-4$ case (recall that $l= 2n+1-m$). In this section we assume that $l \leqslant -4$. Adjusting the notation, we let $l \geqslant 4$ be even; we wish to determine $\pi' = \theta_{-l}(\pi)$. As before, we begin by applying Lemma \ref{cor5.3} to get
\begin{equation}
\label{eq:epi0}
\delta_1\nu^{s_1}\times \delta_2\nu^{s_2}\times \cdots\times\delta_k\nu^{s_k}\rtimes\Theta_{-l}(\sigma)\twoheadrightarrow \Theta_{-l}(\pi) \twoheadrightarrow \pi'.
\end{equation}
The technical difficulty here is that, in contrast with the previous cases, we do not know if $\Theta_{-l}(\sigma)$ is irreducible. In other words, we need to determine the irreducible subquotient of $\Theta_{-l}(\sigma)$ which participates in the above map. To do this, we introduce the tempered support of $\sigma$: we know that $\sigma$ is a subquotient of
\[
\delta_1' \times \dotsm \times \delta_l'\rtimes \pi_{00}
\]
where $\delta_1', \dotsc, \delta_l',\pi_{00}$ are irreducible discrete series representations. In fact, the above representation is semisimple. Thus, setting $\Delta = \delta_1' \times \dotsm \times \delta_l'$ and repeatedly applying Lemma \ref{cor5.3} again, we get
\[
\Delta \rtimes \Theta_{-l}(\pi_{00}) \twoheadrightarrow \Theta_{-l}(\sigma).
\]
This allows us to expand \eqref{eq:epi0} into
\begin{equation}
\label{eq:epi1}
\delta_1\nu^{s_1} \times \dotsm \times \delta_k\nu^{s_k} \times \Delta \rtimes \Theta_{-l}(\pi_{00}) \twoheadrightarrow \pi'.
\end{equation}
We now need an important technical lemma:
\begin{lem}
\label{lem:main}
The subquotient of $\Theta_{-l}(\pi_{00})$ which participates in \eqref{eq:epi1} is $\theta_{-l}(\pi_{00})$. In other words, we have
\begin{equation}
\label{eq:epi2}
\delta_1\nu^{s_1} \times \dotsm \times \delta_k\nu^{s_k} \times \Delta \rtimes \theta_{-l}(\pi_{00}) \twoheadrightarrow \pi'.
\end{equation}
\begin{proof}
This is Proposition 7.1 of \cite{bakic_generic}.
\end{proof}
\end{lem}
The above map \eqref{eq:epi2} is not sufficient to uniquely determine $\pi'$. To do this, we have to find the standard module of $\pi'$. Before we start, let us return for a moment to \eqref{eq:epi0}. Our goal is to show two things (see Proposition \ref{prop:lifts}):
\begin{itemize}
\item the subquotient of $\Theta_{-l}(\sigma)$ which participates in that map is $\theta_{-l}(\sigma)$;
\item the standard module of $\pi'$ is obtained by adding $\delta_1\nu^{s_1}, \dotsc, \delta_k\nu^{s_k}$ to the standard module of $\theta_{-l}(\sigma)$ (and sorting the representations decreasingly with respect to the exponents).
\end{itemize}
The shape of $\theta_{-l}(\sigma)$, which is crucial in the ensuing caluculations, is completely determined by \cite[Theorems 4.3 and 4.5]{Atobe_Gan}. We compile the results of these theorems in the following proposition.
\begin{prop}
\label{prop:lifts_temp_Mp}
Let $\sigma$ be an irreducible tempered representation of the metaplectic group. Let $\phi$ denote its $L$-parameter; assume that $l > 0$ is even.
\begin{enumerate}[(i)]
\item On the going-down tower, $\theta_{0}(\sigma)$ is tempered and we have
	\[
	\nu^\frac{l-1}{2} \times \dotsm \times \nu^\frac{1}{2} \rtimes \theta_{0}(\sigma) \twoheadrightarrow \theta_{-l}(\sigma).
	\]
\item If $l(\sigma)=0$ then $\theta_{-2}(\sigma)$ is the first non-zero lift on the going-up tower; it is tempered. Furthermore, $\theta_{-l}(\sigma)$ is given by
	\[
	\nu^\frac{l-1}{2} \times \dotsm \times \nu^\frac{3}{2} \rtimes \theta_{-2}(\sigma) \twoheadrightarrow \theta_{-l}(\sigma).
	\]
\item Assume that $l(\sigma)=2$ and that $m_\phi(S_2)$ (the multiplicity of $S_2$ in $\phi$) is odd. Then $\theta_{-4}(\sigma)$ is the first lift on the going up tower; it is tempered. We have
		\[
	\nu^\frac{l-1}{2} \times \dotsm \times \nu^\frac{5}{2} \rtimes \theta_{-4}(\sigma') \twoheadrightarrow \theta_{-l}(\sigma).
	\]
\item Assume $l(\sigma)=2$ and that $m_\phi(S_2)=2h > 0$. Then there are no tempered lifts on the going-up tower and  we have
		\[
	\nu^\frac{l-1}{2} \times \dotsm \times \nu^\frac{5}{2}\times\text{St}_3\nu^\frac{1}{2} \times (\text{St}_2,h-1)\rtimes \theta_{-2}(\sigma') \twoheadrightarrow \theta_{-l}(\sigma).
	\]
Here $\sigma'$ denotes the (unique) irreducible tempered representation such that $\sigma \hookrightarrow (\text{St}_2,h) \rtimes \sigma'$ (also $(\text{St}_2,h) = \text{St}_2 \times \dotsm \times \text{St}_2$, $h$ times). Since the parameter of $\theta_{-2}(\sigma')$ contains $S_2$, the representation $(\text{St}_2,h-1)\rtimes \theta_{-2}(\sigma')$ is irreducible and tempered.
\end{enumerate}
\end{prop}
Our proof starts by analyzing the map \eqref{eq:epi2} established by Lemma \ref{lem:main}. We have four cases depending on the shape of $\sigma$, each of them corresponding to one of the cases of the previous Proposition. Before beginning the case-by-case analysis, we state a useful technical lemma:
\begin{lem}
\label{lem:zelevinsky}
\begin{enumerate}[(i)]
\item Let $a \leqslant c \leqslant b < d\in \mathbb{R}$ be congruent mod $\mathbb{Z}$. Then
\[
\zeta(c,d) \times \delta([\nu^a,\nu^b]) \quad \text{and} \quad \delta([\nu^a,\nu^b])\times \zeta(c,d)
\]
are irreducible and isomorphic. (We say that two or more numbers are congruent modulo $\mathbb{Z}$ if their difference is an integer). Note that $a \leqslant c \leqslant b < d$ implies that the segment $[c,d]$ intersects $[a,b]$ from "above".
\item If $[a,b]$ and $[c,d]$ are not linked, then \[
\zeta(c,d) \times \delta([\nu^a,\nu^b]) \quad \text{and} \quad \delta([\nu^a,\nu^b]) \times \zeta(c,d)
\]
are irreducible and isomorphic.
\item When the two segments are disjoint but linked ($c = b+1$), then $\delta([\nu^a,\nu^b])\times \zeta(b+1,d)$ has exactly two irreducible subquotients:
\[
L(\nu^d, \dotsc, \nu^{b+1}, \delta([\nu^a,\nu^b]))\quad \text{and} \quad  L(\nu^d, \dotsc, \nu^{b+2}, \delta([\nu^a,\nu^{b+1}])).
\]
\end{enumerate}
\end{lem}
\begin{proof}
This is Lemma 6.1 and Remarks 6.3 and 6.4 of \cite{bakic_generic}.
\end{proof}

\bigskip

\noindent\underline{{Case 1:}{ the going-down tower}}\\
According to the above Proposition, in this case $\theta_{-l}(\pi_{00})$ is the Langlands quotient of
	\[
	\nu^\frac{l-1}{2} \times \dotsm \times \nu^\frac{1}{2} \rtimes \theta_{0}(\pi_{00}).
	\]
This also implies that $\theta_{-l}(\pi_{00})$ is the unique quotient of
\[
\zeta(\frac{1}{2},\frac{l-1}{2}) \rtimes \theta_{0}(\pi_{00}).
\]
Combining this with the map \eqref{eq:epi2} we get
\[
\delta_1\nu^{s_1} \times \dotsm \times\delta_k\nu^{s_k} \times \Delta \times \zeta(\frac{1}{2},\frac{l-1}{2}) \rtimes \theta_{0}(\pi_{00}) \twoheadrightarrow \pi'.
\]
We now use Lemma \ref{lem:zelevinsky}: $\zeta(\frac{1}{2},\frac{l-1}{2})$ can switch places with all the $\delta_i'$ appearing in $\Delta$. This means that we can write
\begin{equation}
\label{eq:epi3}
\delta_1\nu^{s_1} \times \dotsm \times \delta_k\nu^{s_k} \times\zeta(\frac{1}{2},\frac{l-1}{2}) \times \Delta \rtimes \theta_{0}(\pi_{00}) \twoheadrightarrow \pi'.
\end{equation}
Finally, we observe that there is an irreducible subquotient $\tau$ of $\Delta \rtimes \theta_{0}(\pi_{00})$ such that
\begin{equation}
\label{eq:epi4}
\delta_1\nu^{s_1} \times \dotsm \times \delta_k\nu^{s_k} \times\zeta(\frac{1}{2},\frac{l-1}{2}) \rtimes \tau \twoheadrightarrow \pi'.
\end{equation}
Note that $\tau$ is tempered, because $\theta_{0}(\pi_{00})$ is, too (moreover, in this case, $\theta_{0}(\pi_{00})$ is in discrete series), as are all the irreducible subquotients of $\Delta$. We now claim the following:
\begin{lem}
\label{lem:unique_quotient}
The representation appearing on the left-hand side of \eqref{eq:epi4} has a unique irreducible quotient.
\end{lem}

\begin{proof}
We prove that the representation in question is itself a quotient of a standard module, and the conclusion follows. This standard module is obtained by inserting $\nu^\frac{l-1}{2}, \nu^\frac{l-3}{2}, \dotsc, \nu^\frac{1}{2}$ between $\delta_1\nu^{s_1}, \dotsc, \delta_k\nu^{s_k}$ so that the decreasing order is preserved. To show this, let $[\rho\nu^a,\rho\nu^b]$ be the segment which defines $\delta_1\nu^{s_1}$ (in particular, we have $s_1 = \frac{a+b}{2}$). We assume that $\rho$ is the trivial representation; for $\rho \neq \mathbb{1}$ we only need a simpler version of the following argument. If $s_1 \geqslant \frac{l-1}{2}$ then the representation in question is a quotient of the standard module
\[
\delta_1\nu^{s_1} \times \dotsm \times \delta_k\nu^{s_k} \times \nu^\frac{l-1}{2} \times \dotsm \nu^\frac{1}{2} \rtimes \tau,
\]
and we are done. If $s_1 < \frac{l-1}{2}$ we use the following observation based on Lemma \ref{lem:zelevinsky}.
\begin{lem}
\label{lem:tricky_condition}
Let $\sigma$ be an irreducible representation of $O(V)$ and assume that
\[
A \times \delta([\nu^a,\nu^b]) \times \zeta(c,d) \rtimes \sigma_0 \twoheadrightarrow \sigma
\]
for $\frac{a+b}{2} \leqslant d$ and some representations $A$ and $\sigma_0$ . If
\[
\label{eq:tricky1}\tag{i}
c \neq b+1
\]
then setting $s =\mathrm{min}\{s''\in[c,d]:s''\ge \frac{a+b}{2}\}$ we have
\[
A \times \zeta(s,d) \times \delta([\nu^a,\nu^b]) \times \zeta(c,s-1) \rtimes \sigma_0 \twoheadrightarrow \sigma.
\]
Here, the segment $[c,s-1]$ can be empty. Assume, a fortiori, that
\[
\label{eq:tricky2}\tag{ii}
s-1 \lneqq \frac{a+b}{2}
\]
(notice that this implies \eqref{eq:tricky1}). Then, for any $\delta([\nu^{a'},\nu^{b'}])\times \zeta(s,d)$ with $\frac{a'+b'}{2}\geqslant \frac{a+b}{2}$ the number $s'=\mathrm{min}\{s''\in[s,d]:s''\geqslant \frac{a'+b'}{2}\}$ also satisfies the above condition \eqref{eq:tricky2} with respect to $[a',b']$.
\end{lem}

\begin{proof}
We know that $\zeta(c,d)$ is a quotient of $\zeta(s,d) \times \zeta(c,s-1)$, so we have
\[
A \times \delta([\nu^a,\nu^b]) \times \zeta(s,d) \times \zeta(c,s-1) \rtimes \sigma_0 \twoheadrightarrow \sigma.
\]
If \eqref{eq:tricky1} holds, then Lemma \ref{lem:zelevinsky} shows that $\delta([\nu^a,\nu^b])$ and $\zeta(s,d)$ can switch places. We thus get
\[
A \times \zeta(s,d) \times \delta([\nu^a,\nu^b]) \times \zeta(c,s-1) \rtimes \sigma_0 \twoheadrightarrow \sigma,
\]
as required.

For the second part of the claim, assume that $s$ satisfies condition \eqref{eq:tricky2}. Then $s' - 1 \geqslant \frac{a'+b'}{2}$ would imply
\[
s' -1 \geqslant \frac{a'+b'}{2} \geqslant \frac{a+b}{2},
\]
which (since $s'$ is defined to be minimal) implies $s' = s$. However, this forces $s-1 \geqslant \frac{a+b}{2}$, contradicting \eqref{eq:tricky2}.
\end{proof}

Inductively applying Lemma \ref{lem:tricky_condition}---first with $\delta([\nu^a,\nu^b]) = \delta_1\nu^{s_1}$ and $[c,d] = [\frac{1}{2},\frac{l-1}{2}]$, then $\delta([\nu^a,\nu^b]) = \delta_2\nu^{s_2}$ and $[c,d]
= [s,\frac{l-1}{2}]$, etc.---we show that the representation appearing in  \eqref{eq:epi4} is indeed a quotient of a standard representation. This proves Lemma \ref{lem:unique_quotient}.
\end{proof}

We have now found the standard module of $\pi'$: the representations $\nu^\frac{l-1}{2},\allowbreak \nu^\frac{l-3}{2}, \dotsc, \nu^\frac{1}{2}$ are simply inserted among $\delta_1\nu^{s_1}, \dotsc, \delta_k\nu^{s_k}$ so that the exponents form a decreasing sequence. The only thing that remains to be determined is the tempered part, $\tau$.

We now know that the left-hand side of \eqref{eq:epi4} has a unique irreducible quotient. Therefore, we have
\begin{equation}
\label{eq:epi5}
\delta_1\nu^{s_1} \times \dotsm \times \delta_k\nu^{s_k} \rtimes \tau' \twoheadrightarrow \pi'
\end{equation}
where $\tau'$ is the unique irreducible quotient of $\zeta(\frac{1}{2},\frac{l-1}{2}) \rtimes \tau$ (that is, the Langlands quotient of $\nu^\frac{l-1}{2} \times \dotsm \times \nu^\frac{1}{2} \rtimes \tau$). It is now important to note the following:
\begin{lem}
\label{lemma:really_subq}
The representation $\tau'$ is a subquotient of $\Theta_{-l}(\sigma)$.
\end{lem}

We will use variations of this observation in all the subsequent cases, so we give a detailed explanation here.
\begin{proof}
We revisit the maps we have used so far: \eqref{eq:epi0}, \eqref{eq:epi1}, \eqref{eq:epi2}, \eqref{eq:epi4} and \eqref{eq:epi5}. Let $\Pi$ denote $\delta_1\nu^{s_1} \times \dotsm \times \delta_k\nu^{s_k}$. Starting from
\[
T \colon \Delta \rtimes \Theta_{-l}(\pi_{00}) \twoheadrightarrow \Theta_{-l}(\sigma)
\]
we induce to obtain
\[
\text{Ind}(T)\colon \Pi \times \Delta \rtimes \Theta_{-l}(\pi_{00}) \twoheadrightarrow \Pi \rtimes \Theta_{-l}(\sigma).
\]
Composing this with \eqref{eq:epi0} (which is given by $S \colon \Pi \rtimes \Theta_{-l}(\sigma) \twoheadrightarrow \pi'$) we get \eqref{eq:epi1}:
\[
S \circ \text{Ind}(T): \Pi \times \Delta \rtimes \Theta_{-l}(\pi_{00}) \twoheadrightarrow \pi'.
\]
Lemma \ref{lem:main} shows that no subquotient of $\Theta_{-l}(\pi_{00})$ except $\theta_{-l}(\pi_{00})$ can participate in the above epimorphism; in other words, we have $\Pi \times \Delta \rtimes \Theta^0 \subseteq \ker S \circ \text{Ind}(T)$ where we have used $\Theta^0$ to denote the maximal proper subrepresentation of $\Theta_{-l}(\pi_{00})$.

Taking the quotient of $S \circ \text{Ind}(T)$ by $\Pi \times \Delta \rtimes \Theta^0$ we get a new map, \eqref{eq:epi2}:
\[
\widetilde{S \circ \text{Ind}(T)}: \Pi \times \Delta \rtimes \theta_{-l}(\pi_{00}) \twoheadrightarrow \pi'.
\]
By the construction of this map it is obvious that any subquotient $\tau'$ of $\Delta \rtimes \theta_{-l}(\pi_{00})$ participating in the above epimorphism must be a subquotient of $\Theta_{-l}(\sigma)$, so we get \eqref{eq:epi5}. This subquotient is written as a subquotient of $\zeta(\frac{1}{2},\frac{l-1}{2})  \rtimes \tau$ in \eqref{eq:epi4}, and Lemma \ref{lem:unique_quotient} shows that $\tau'$ is in fact a quotient of $\zeta(\frac{1}{2},\frac{l-1}{2}) \rtimes \tau$.
\end{proof}

Finally, it remains to verify the following claim.
\begin{lem}
\label{lem:really_small_theta}
The only subquotient of $\Theta_{-l}(\sigma)$ with standard module of the form $\nu^\frac{l-1}{2} \times \dotsm \times \nu^\frac{1}{2} \rtimes \tau$ is $\theta_{-l}(\sigma)$.
\end{lem}

\begin{proof}
Let $\tau'$ be a subquotient of $\Theta_{-l}(\sigma)$ such that
\[
\nu^\frac{l-1}{2}\times \dotsm \times \nu^\frac{1}{2} \rtimes \tau \twoheadrightarrow \tau'
\]
for some tempered $\tau$. Denote by $\tau_1$ the Langlands quotient of $\nu^\frac{l-3}{2}\times \dotsm \times \nu^\frac{1}{2} \rtimes \tau$, so that $\nu^\frac{l-1}{2} \rtimes \tau_1 \twoheadrightarrow \tau'$, i.e.~$\tau' \hookrightarrow \nu^\frac{1-l}{2} \rtimes \tau_1$.

We now use Kudla's filtration:
the map we have just obtained shows that $\Hom(\tau', \nu^\frac{1-l}{2} \rtimes \tau_1) \neq 0$. Using Frobenius reciprocity, this means that $\Hom(R_{Q_1}(\tau'), \nu^\frac{1-l}{2} \otimes \tau_1) \neq 0$, where $Q_1$ denotes the appropriate standard maximal parabolic subgroup of $O(V)$. From here, we deduce that $ \nu^\frac{1-l}{2} \otimes \tau_1$ is a quotient of $R_{Q_1}(\tau')_{ \nu^\frac{1-l}{2}}$, which implies that it is also a subquotient of $R_{Q_1}(\Theta_{-l}(\sigma))_{ \nu^\frac{1-l}{2}}$. On the other hand, $\sigma \otimes R_{Q_1}(\Theta_{-l}(\sigma))_{ \nu^\frac{1-l}{2}}$ is obviously a quotient of $R_{Q_1}(\omega_{m_0,2n_0})$---here $2n_0$ is defined by $\sigma \in \text{Irr}(\widetilde{\text{Sp}(W_{2n_0})})$, $m_0 = 2n_0 + 1 + l$, and $\omega_{m_0,2n_0}$ is the corresponding Weil representation. Kudla's filtration of $R_{Q_1}(\omega_{m_0,2n_0})$ is
\begin{align*}
J^0 &=  \nu^\frac{1-l}{2} \otimes \omega_{m_0-2,2n_0}\quad (\text{the quotient})\\
J^1 &= \text{Ind}(\Sigma_1 \otimes \omega_{m_0-2,2n_0-2}) \quad (\text{the subrepresentation}).
\end{align*}
It is now easy to show that $J^1$ cannot participate in the epimorphism $R_{Q_1}(\omega_{m_0,n_0})\twoheadrightarrow \sigma \otimes R_{P'_1}(\Theta_{-l}(\sigma))_{ \nu^\frac{1-l}{2}}$. Otherwise, an application of the second Frobenius reciprocity would show that $R_{\overline{P}_1}(\sigma)$ (where $\overline{P}_1$ denotes the parabolic subgroup opposite to $P_1$) has a quotient of the form $ \nu^\frac{l-1}{2} \otimes \pi_1$. As $\sigma$ is tempered and $\frac{l-1}{2} > 0$, Casselman's criterion shows that this is impossible.

This means that $\sigma \otimes R_{Q_1}(\Theta_{-l}(\sigma))_{ \nu^\frac{1-l}{2}}$ is a quotient of $J^0$, which immediately implies that $\tau_1$ is a subquotient of $\Theta_{2-l}(\sigma)$.

Inductively repeating this argument shows that $\tau$ is a subquotient of $\Theta_{0}(\sigma)$; however, $\Theta_{0}(\sigma)$ is irreducible, so we must have $\tau = \Theta_{0}(\sigma) = \theta_{0}(\sigma)$. This proves that $\tau'$ is the Langlands quotient of
\[
 \nu^\frac{l-1}{2} \times \dotsm \times  \nu^\frac{1}{2}  \rtimes \theta_{0}(\sigma).
\]
By Proposition \ref{prop:lifts_temp_Mp} (i), we conclude that $\tau' = \theta_{-l}(\sigma)$.
\end{proof}
\noindent This proves
$
\theta_{-l}(\pi) = L(\delta_1\nu^{s_1}, \dotsc, \delta_k\nu^{s_k}, \nu^\frac{l-1}{2}, \dotsc,  \nu^\frac{1}{2}; \theta_{0}(\sigma))
$, thereby completing Case 1.

\bigskip 

\noindent\underline{Case 2: going-up tower when $l(\sigma) = 0$}\\
From $l(\sigma) = 0$ we easily get $l(\pi_{00}) = 0$. In other words, $\pi_{00}$ first appears on the going-up tower when $l=-2$. Furthermore, Proposition \ref{prop:lifts_temp_Mp} shows that $\theta_{-2}(\pi_{00})$ is tempered, whereas  for $l > 2$ the representation $\theta_{-l}(\pi_{00})$ is the Langlands quotient of
\[
\nu^\frac{l-1}{2} \times \dotsm \times \nu^\frac{3}{2} \rtimes \theta_{-2}(\pi_{00}),
\]
that is, the unique quotient of
\[
\zeta(\frac{3}{2},\frac{l-1}{2}) \rtimes \theta_{-2}(\pi_{00}).
\]
Using this in \eqref{eq:epi2} we get
\[
\delta_1\nu^{s_1} \times \dotsm \times \delta_k\nu^{s_k} \times \Delta \times \zeta(\frac{3}{2},\frac{l-1}{2}) \rtimes \theta_{-2}(\pi_{00}) \twoheadrightarrow \pi'.
\]
We proceed like in Case 1: $\zeta(\frac{3}{2},\frac{l-1}{2})$ can (using Lemma \ref{lem:zelevinsky}) switch places with all the $\delta_i'$ appearing in $\Delta$. The only possible exception is $\text{St}_2$ (see Lemma \ref{lem:zelevinsky} (iii)). However, $\text{St}_2$ cannot appear as it would imply $l(\sigma) = 2$. We thus have
\[
\delta_1\nu^{s_1} \times \dotsm \times \delta_k\nu^{s_k} \times \zeta(\frac{3}{2},\frac{l-1}{2}) \times \Delta \rtimes \theta_{-2}(\pi_{00}) \twoheadrightarrow \pi',
\]
so there is an irreducible (and tempered) subquotient $\tau$ of $\Delta \rtimes \theta_{-2}(\pi_{00})$ such that
\[
\delta_1\nu^{s_1} \times \dotsm \times \delta_k\nu^{s_k} \times \zeta(\frac{3}{2},\frac{l-1}{2}) \rtimes \tau \twoheadrightarrow \pi'.
\]
Repeating the arguments of Lemmas \ref{lem:unique_quotient}, \ref{lemma:really_subq} and \ref{lem:really_small_theta}, we can now show that
\[
\delta_1\nu^{s_1} \times \dotsm \times \delta_k\nu^{s_k} \rtimes \theta_{-l}(\sigma)\twoheadrightarrow \pi'
\]
and 
\[\theta_{-l}(\pi) = 
L(\delta_1\nu^{s_1}, \dotsc, \delta_k\nu^{s_k}, \nu^\frac{l-1}{2}, \dotsc, \nu^\frac{3}{2}; \theta_{-2}(\sigma)).
\]
Note that condition (i) of Lemma \ref{lem:tricky_condition} is not trivially satisfied like in Case 1. In the first step, it is crucial that $\nu^\frac{1}{2}$ does not appear among the $\delta_i\nu^{s_i}$. This, however, is a part of our assumptions stated at the beginning of section \ref{subs:std_mod_lifts}.
\bigskip

\noindent\underline{Case 3: going-up tower when $l(\sigma) = 2$, $m_\phi(S_2)$ is odd}\\
The fact that $S_2$ is contained in $\phi$ with odd multiplicity implies that it is also contained in the parameter of $\pi_{00}$. We thus have $l(\pi_{00})=2$ so that $\pi_{00}$ first appears on the going-up tower when $l=-4$. The lift $\theta_{-4}(\pi_{00})$ is tempered, and for $l > 4$ we have
\[
\nu^\frac{l-1}{2} \times \dotsm \times \nu^\frac{5}{2} \rtimes
\theta_{-4}(\pi_{00}) \twoheadrightarrow \theta_{-l}(\pi_{00}).
\]
Combining this with \eqref{eq:epi2} we get
\[
\Pi \times\Delta \times \zeta(\frac{5}{2},\frac{l-1}{2}) \rtimes \theta_{-4}(\pi_{00}) \twoheadrightarrow \pi'
\]
where $\Pi = \delta_1\nu^{s_1} \times \dotsm \times \delta_k\nu^{s_k}$.

For simplicity we now assume that $\Delta$ does not contain $\text{St}_4$. When $\text{St}_4$ appears, the proof requires minor technical modifications, which we leave to the reader (see Case 4 for a discussion which handles a similar exception). The benefit of leaving out $\text{St}_4$ is that $\zeta(\frac{5}{2},\frac{l-1}{2})$ can switch places with all the representations appearing in $\Delta$ (see condition (i) of Lemma \ref{lem:tricky_condition}), so we get
\[
\Pi \times \zeta(\frac{5}{2},\frac{l-1}{2}) \times \Delta \rtimes \theta_{-4}(\pi_{00}) \twoheadrightarrow \pi'.
\]
Therefore, there is an irreducible (and tempered) subquotient $\tau$ of $\Delta \rtimes \theta_{-4}(\pi_{00})$ such that
\[
\Pi \times \zeta(\frac{5}{2},\frac{l-1}{2}) \rtimes \tau \twoheadrightarrow \pi'.
\]
Repeating the arguments of Lemmas \ref{lem:unique_quotient} and \ref{lemma:really_subq} we can now show that this implies
\[
\Pi \rtimes \theta_{-l}(\sigma) \twoheadrightarrow \pi'.
\]
Note that a slight modification of these arguments is necessary when a representation of the form $\delta[\nu^a,\nu^\frac{3}{2}]$ appears in $\Pi$ (see Lemma 7.11 and the discussion preceding it in \cite{bakic_generic}).

It remains to determine the standard module of $\pi' = \theta_{-l}(\pi)$. If $\Pi$ does not contain a representation of the form $\delta[\nu^a,\nu^\frac{3}{2}]$, this is straightforward, using the arguments of Lemma \ref{lem:unique_quotient}. However, if $\Pi$ contains any of these representations, then we cannot use Lemma \ref{lem:tricky_condition} because condition (i) is not satisfied. We briefly sketch the steps used to bypass this problem. 

The only exceptions of the form $\delta[\nu^a,\nu^\frac{3}{2}]$ are $\delta([\nu^{-\frac{1}{2}},\nu^\frac{3}{2}])$ and $\nu^\frac{3}{2}$; note that $\delta([\nu^{\frac{1}{2}},\nu^\frac{3}{2}])$ cannot appear by Proposition \ref{l(pi)}. We may group all the occurrences of the exceptional representations from $\Pi$ into a representation which we denote $S$, so that $\Pi = \Pi' \times S$ (recall that $\delta([\nu^{-\frac{1}{2}},\nu^\frac{3}{2}])$ can appear at most once if the standard module of $\pi$ is irreducible). The fact that the standard module of $\pi$ is irreducible allows us to write $\gamma_\psi^{-1}\Pi' \times \gamma_\psi^{-1}S^\vee \twoheadrightarrow \pi$ instead of $\Pi' \times \gamma_\psi^{-1}S \twoheadrightarrow \pi$. Repeating the steps from the beginning of this section, we now arrive at
\[
\Pi' \times S^\vee \times \Delta \rtimes \Theta_{-l}(\pi_{00}) \twoheadrightarrow \pi'.
\]
We now modify Lemma \ref{lem:main} (see Lemma 7.12 in \cite{bakic_generic}) to show that we may switch from $\Theta_{-l}(\pi_{00})$ to $\theta_{-l}(\pi_{00})$:
\[
\Pi' \times S^\vee \times \Delta \rtimes \theta_{-l}(\pi_{00})  \twoheadrightarrow \sigma.
\]
Recall that $\zeta(\frac{5}{2},\frac{l-1}{2}) \rtimes \theta_{-4}(\pi_{00}) \twoheadrightarrow \theta_{-l}(\pi_{00})$. Furthermore, $\zeta(\frac{5}{2},\frac{l-1}{2})$ can now switch places with all the representations from $S^\vee$ and $\Delta$. Note that there are no problems in $\Delta$ since we have assumed that $\text{St}_4$ does not appear; likewise, $S^\vee$ has no problematic segments, as opposed to $S$. We may thus write
\[
\Pi' \times \zeta(\frac{5}{2},\frac{l-1}{2}) \times S^\vee \times \Delta \rtimes \theta_{-4}(\pi_{00}) \twoheadrightarrow \pi'.
\]
Finally, this shows that there is an irreducible subquotient $B$ of $S^\vee \times \Delta \rtimes \theta_{-4}(\pi_{00})$ such that
\[
\Pi' \times \zeta(\frac{5}{2},\frac{l-1}{2}) \rtimes B \twoheadrightarrow \pi'.
\]
Comparing this map with $\Pi \rtimes \theta_{-l}(\sigma) \twoheadrightarrow \pi'$ (the existence of which we have already proven) we can prove that the standard module of $B$ is given by $S \rtimes \theta_{-4}(\sigma)$. This yields
\[
\Pi' \times \zeta(\frac{5}{2},\frac{l-1}{2}) \times S \times \theta_{-4}(\sigma).
\]
Starting from this map, we may apply the process described in Lemma \ref{lem:unique_quotient} (in particular, there are no more exceptions to Lemma \ref{lem:tricky_condition}), and in this way complete the proof.

A minor detail requires additional attention: if $\Pi'$ contains $\nu^\frac{1}{2}$, then the procedure of Lemma \ref{lem:unique_quotient} will not result in a decreasing order. Namely, $\nu^\frac{1}{2}$ will still need to switch places with the representations of $S$. Note that $\delta([\nu^{-\frac{1}{2}},\nu^\frac{3}{2}])$ does not cause a problem (see Lemma \ref{lem:zelevinsky} (ii)). It is only $\nu^\frac{3}{2}$ that we cannot trivially move, since $\nu^\frac{1}{2} \times \nu^\frac{3}{2}$ reduces. However, this reducibility shows that $\nu^\frac{1}{2}$ and $\nu^\frac{3}{2}$ cannot both appear in the standard module if it is irreducible.

\bigskip

\noindent\underline{Case 4: going-up tower when $l(\sigma) = 2$, $m_\phi(S_2) = 2h > 0$}\\
Since the multiplicity of $S_2$ in $\phi$ is even, we see that $S_2$ does not appear in the parameter of $\pi_{00}$. Thus $l(\pi_{00})=0$ and $\pi_{00}$ first appears on the going-up tower when $l=-2$. In this case, $\theta_{-2}(\pi_{00})$ is tempered (in fact, it is a discrete series representation) and we have
\[
\nu^\frac{l-1}{2} \times \dotsm \times \nu^\frac{3}{2} \rtimes \theta_{-2}(\pi_{00}) \twoheadrightarrow \theta_{-l}(\pi_{00}).
\]
Taking this into account in \eqref{eq:epi2}, we get
\[
\Pi \times \Delta \times \zeta(\frac{3}{2},\frac{l-1}{2}) \rtimes \theta_{-2}(\pi_{00}) \twoheadrightarrow \pi',
\]
where $\Pi = \delta_1\nu^{s_1} \times \dotsm \times \delta_k\nu^{s_k}$. We have already seen that $\zeta(\frac{3}{2},\frac{l-1}{2})$ can switch places with all the representations in $\Delta$, except $\text{St}_2$. Therefore, we write $\Delta = (\text{St}_2,h) \times \Delta'$ where $\Delta'$ is the product of all the $\delta_i'$ different from $\text{St}_2$. Lemma \ref{lem:zelevinsky} (iii) now shows that we have two possibilities:
\begin{enumerate}[(i)]
\item $\Pi \times \zeta(\frac{3}{2},\frac{l-1}{2}) \times \Delta \rtimes \theta_{-2}(\pi_{00}) \twoheadrightarrow \pi'$;

\item $\Pi \times L \times (\text{St}_2,h-1) \times \Delta' \rtimes \theta_{-2}(\pi_{00}) \twoheadrightarrow \pi'$, where $L = L(\nu^\frac{l-1}{2}, \dotsc, \nu^\frac{5}{2}, \text{St}_3\nu^\frac{1}{2})$.
\end{enumerate}

Let us show that (i) is not possible. If (i) were true, there would be an irreducible (and tempered) subquotient $\tau_1$ of $\Delta \rtimes \theta_{-2}(\pi_{00})$ such that
\[
\Pi \times \zeta(\frac{3}{2},\frac{l-1}{2}) \rtimes \tau_1 \twoheadrightarrow \pi'.
\]
We can now argue as in Case 1, in particular, like in Lemmas \ref{lem:unique_quotient} and \ref{lemma:really_subq}: the above map implies that $\Theta_{-l}(\sigma)$ has a subquotient whose standard module is of the form
\[
\nu^\frac{l-1}{2} \times \dotsm \times \nu^\frac{3}{2} \rtimes \tau_1.
\]
In analogy with Case 3, these arguments need to be slightly modified if $\nu^\frac{1}{2}$ appears in $\Pi$; we omit the details here.

Finally, we use the arguments of Lemma \ref{lem:really_small_theta} to show that a subquotient of the above form cannot appear: an inductive procedure shows that it would imply $\theta_{-2}(\sigma) \neq 0$, but this contradicts the assumptions of Case 4. We have thus shown that (i) is not possible.

Similar reasoning now leads to the desired conclusion if (ii) is true. Assuming (ii), we first note that there is an irreducible (and tempered) subquotient $\tau_2$ of $(\text{St}_2,h-1) \times \Delta' \rtimes \theta_{-2}(\pi_{00})$ such that
\[
\label{eq:adhoc}\tag{\textasteriskcentered}
\Pi \times L \rtimes \tau_2 \twoheadrightarrow \pi'.
\]
Let us show that the left-hand side of the above map has a unique irreducible quotient. This is done like in Lemma \ref{lem:unique_quotient}. The only exceptions which require special treatment are those in which some of the representations $\delta_1\nu^{s_1}, \dotsc, \delta_k\nu^{s_k}$ are defined by segments ending in $\nu^\frac{3}{2}$ (thus contradictiong condition (i) of Lemma \ref{lem:tricky_condition}). 

On the other hand, since the tempered support of $\sigma$ contains $\text{St}_2 = \delta[\nu^{-\frac{1}{2}},\nu^{\frac{1}{2}}]$, we see that $\delta_1\nu^{s_1}, \dotsc, \delta_k\nu^{s_k}$ cannot be equal to $\nu^\frac{3}{2}$. The segment $[\nu^\frac{3}{2}]$ is linked to $[\nu^{-\frac{1}{2}}, \nu^\frac{1}{2}]$, and would thus cause the standard module of $\pi$ to reduce.

The only remaining exception which could prevent us from using Lemma \ref{lem:tricky_condition} is $\text{St}_3\nu^\frac{1}{2}$. However, it is easy to show that $L$ and $\text{St}_3\nu^\frac{1}{2}$ can in fact switch places. This means that the procedure of Lemma \ref{lem:unique_quotient} can be applied even in this case. Thus, we may deduce that the left-hand side of (ii) has a unique irreducible quotient.

In particular, this shows that the subquotient of $L \rtimes \tau_2$ (call it $\tau_3$) which participates in \eqref{eq:adhoc} is equal to its unique irreducible quotient. Arguing as in Lemma \ref{lemma:really_subq} we can also show that $\tau_3$ is at the same time a subquotient of $\Theta_{-l}(\sigma)$. We are thus looking for a subquotient of $\Theta_{-l}(\sigma)$ whose standard module is of the form
\[
\nu^\frac{l-1}{2} \times \dotsm \times \nu^\frac{5}{2} \times \text{St}_3\nu^\frac{1}{2} \rtimes \tau_2.
\]
The arguments of Lemma \ref{lem:really_small_theta} now show that $L(\text{St}_3\nu^\frac{1}{2} \rtimes \tau_2)$ should be a subquotient of $\Theta_{-4}(\sigma)$. It is easy to show that the only such subquotient of $\Theta_{-4}(\sigma)$ is in fact $\theta_{-4}(\sigma)$.

We have thus shown the following:
\[
\nu^\frac{l-1}{2} \times \dotsm \times \nu^\frac{5}{2} \rtimes \theta_{-4}(\sigma) \twoheadrightarrow \tau_3.
\]
By Proposition \ref{prop:lifts_temp_Mp}, this means that $\tau_3$ is equal to $\theta_{-l}(\sigma)$. This completes the fourth and final case of our proof.

\section{Theta lifts of generic representations II}
\label{sec:lifts2}

We now turn to the rest of generic representations: we examine the lifts of a $\psi$-generic $\pi=L(\gamma_{\psi}^{-1}\delta_1\nu^{s_1}, \gamma_{\psi}^{-1}\delta_2\nu^{s_2},\ldots , \gamma_{\psi}^{-1}\delta_k\nu^{s_k};\sigma)$ such that $l(\sigma)=0$ and such that there exists, among $\delta_1\nu^{s_1},\ldots,\delta_k\nu^{s_k},$  a representation of the form $\delta([\nu^{1/2},\nu^{m+1/2}])$ for some $m\in \Z_{\ge 0}$. In accordance with Proposition \ref{l(pi)}, the case in which $\nu^{1/2}$ appears among $\delta_1\nu^{s_1},\ldots,\delta_k\nu^{s_k}$ is treated separately.

\begin{rem} When $l(\sigma)=0$ and $\nu^{1/2}$ appears among $\delta_1\nu^{s_1},\ldots,\delta_k\nu^{s_k}$ we can indeed have $l(\sigma)=0$ and $l(\pi)>0:$ for example, $l(\mu_0)=0$ and $l(\omega_{\psi}^+)=L(\gamma_{\psi}^{-1}\nu^{1/2};\mu_0)=2.$ Here $\mu_0$ denotes the unique non-trivial representation of $\widetilde{Sp(0,F)}=\{1,-1\}$ and $\omega_{\psi}^+$ denotes the even Weil representation of $\widetilde{SL(2,F)}$ attached to a character $\psi.$ Actually, we always have this kind of situation.
\end{rem} 

\begin{prop}
\label{prop:l(pi),l(sigma)=0}
  Let $\pi=L(\gamma_{\psi}^{-1}\delta_1\nu^{s_1},\dotsc,\gamma_{\psi}^{-1}\delta_{k-1}\nu^{s_{k-1}},\gamma_{\psi}^{-1}\nu^{1/2};\sigma)$ be a $\psi$--generic representation $\pi$  of $\widetilde{Sp(2n,F)}$ with $l(\sigma)=0.$ Then, $l(\pi)=2.$
\end{prop}

\begin{proof}
From the proof of the Proposition \ref{l(pi)} we know that $l(\pi)\ge l(\sigma).$ We now prove that $l(\pi)\le 2.$ Indeed, assume that $l(\pi)\ge 4.$ Then, on the split-tower (which is, we remind, the going-down tower for $\sigma$), by a consecutive use of Lemma \ref{cor5.3}, we have
\[\delta_1\nu^{s_1}\times \delta_2\nu^{s_2}\times \cdots\times \delta_{k-1}\nu^{s_{k-1}}\times \nu^{1/2}\rtimes \Theta_{4}(\sigma)\twoheadrightarrow \Theta_{4}(\pi).\]
Indeed, there are no exceptions---the possible exceptions would be $\nu^{3/2},\;\nu^1St_{2}$ and $\nu^{1/2}St_{3},$ but similarly as in the proof of Proposition \ref{l(pi)}, we can avoid $\nu^{3/2}$ and $\nu^{1/2}St_{3}$. On the other hand, $\nu^1St_{2}$ cannot appear among $\delta_1\nu^{s_1},\dotsc,\allowbreak\delta_{k-1}\nu^{s_{k-1}}$ since we have $\delta_k\nu^{s_k}=\nu^{1/2}$ (cf.~Theorem \ref{generic_reps_metaplectic}). Now $\Theta_{4}(\sigma)=0,$ so necessarily $\Theta_{4}(\pi)=0.$ This proves $l(\pi)\le 2.$  We now prove that $l(\pi)\neq 0.$
Assume the contrary, i.e., $l(\pi)=0.$ Then, the first  non-zero lift on the non-split tower for $\pi$ is for $l=-2.$  Again, by a consecutive use of Lemma \ref{cor5.3}, on the non-split tower, we have
\begin{equation}
\label{l=-2}
\delta_1\nu^{s_1}\times \delta_2\nu^{s_2}\times \cdots\times \delta_{k-1}\nu^{s_{k-1}}\times \nu^{1/2}\rtimes \Theta_{-2}(\sigma)\twoheadrightarrow \Theta_{-2}(\pi).
\end{equation}
We need the following
\begin{lem}
\label{irreducibility_1_2}
 We keep the assumptions of this Proposition. Then, the representation $\nu^{1/2}\rtimes \Theta_{-2}(\sigma)$ is irreducible. Here $\Theta_{-2}(\sigma)$ is the full lift on the non-split tower.
\end{lem}
\begin{proof}(of Lemma) Note that $\Theta_{-2}(\sigma)$ is irreducible and tempered (see the remark following Proposition \ref{first_lifts}). By the description of the L-parameter $\phi_{\Theta_{-2}(\sigma)}$ of $\Theta_{-2}(\sigma)=\theta_{-2}(\sigma)$ (given in Theorem 4.5 (1) of \cite{Atobe_Gan}; note that $\chi_{V}=1$), we have that $\phi_{\Theta_{-2}(\sigma)}=\phi_{\sigma}\oplus S_2.$ Also, by the same Theorem and Desideratum 3.1 (6)  of \cite{Atobe_Gan}, we have that the character of the component group of the centralizer of  $\phi_{\Theta_{-2}(\sigma)},$ which, by Langlands correspondence, corresponds to $\Theta_{-2}(\sigma),$ attains value $-1$ on $S_2.$ Now assume that $\nu^{1/2}\rtimes \Theta_{-2}(\sigma)$ reduces. By the same arguments as in Lemma \ref{length_two}, there exists an irreducible tempered representation $T$ such that $T\hookrightarrow \nu^{1/2}\rtimes \Theta_{-2}(\sigma).$ Since $S_2$ appears with multiplicity two in the L-parameter of $T,$ there exists an irreducible tempered representation $\sigma_3$ such that $T\hookrightarrow \delta([\nu^{-1/2},\nu^{1/2})]\rtimes \sigma_3.$ This means that  $\delta([\nu^{-1/2},\nu^{1/2})]\otimes \sigma_3$ appears as a subquotient of the appropriate Jacquet module of $T,$ and, accordingly, in the Jacquet module of $\nu^{1/2}\rtimes \Theta_{-2}(\sigma).$ Using Tadi\'c's formula for Jacquet modules of the induced representations, we get that $\nu^{1/2}\otimes \sigma_4$ (for some irreducible $\sigma_4$) appears in the Jacquet module of $\Theta_{-2}(\sigma),$ so that $\Theta_{-2}(\sigma) \hookrightarrow \nu^{1/2}\rtimes \sigma_5,$ for some irreducible representation $\sigma_5.$ We easily get that $\sigma_5$ is necessarily tempered.  Let $\sigma_{00}$ and $\sigma_0$ be the discrete series representations (of the appropriate full odd orthogonal groups) which are the classical parts of the tempered support of $\Theta_{-2}(\sigma)$ and $\sigma_5,$ respectively. We easily get that $\sigma_{00} \hookrightarrow \nu^{1/2}\rtimes \sigma_0.$ This means that the character (of the appropriate component group attached to the L-parameter of $\sigma_{00}$) corresponding to  $\sigma_{00}$
attains value 1 on $S_2$ (cf.~\cite{Moe}, section 5.1., equation (2)). Since this character of  $\sigma_{00}$ is just the restriction of the character corresponding to $\Theta_{-2}(\sigma),$ we get a contradiction.
\end{proof}
Note that, since  $\Theta_{-2}(\sigma)$ is irreducible and tempered, $\theta_{-2}(\pi)$ is precisely the Langlands quotient of the left-hand side of \eqref{l=-2}. The irreducibility of $\nu^{1/2}\rtimes \Theta_{-2}(\sigma)$ now gives that we also have the following epimorphism
\[\delta_1\nu^{s_1}\times \delta_2\nu^{s_2}\times \cdots\times \delta_{k-1}\nu^{s_{k-1}}\times \nu^{-1/2}\rtimes \Theta_{-2}(\sigma)\twoheadrightarrow \theta_{-2}(\pi).\]
Now we go back to the metaplectic tower for $l=2$,  and again, by applying Lemma \ref{cor5.3} without exceptional cases, together with Proposition 5.4. of \cite{Atobe_Gan},  we get
\begin{equation}
\label{eq_2}
\gamma_{\psi}^{-1}\delta_1\nu^{s_1}\times \gamma_{\psi}^{-1}\delta_2\nu^{s_2}\times \cdots\times \gamma_{\psi}^{-1}\delta_{k-1}\nu^{s_{k-1}}\times \gamma_{\psi}^{-1}\nu^{-1/2}\rtimes \sigma\twoheadrightarrow \pi.
\end{equation}
By Theorem \ref{generic_reps_metaplectic}, the representation on the left-hand side of \eqref{eq_2} is isomorphic to
\[\gamma_{\psi}^{-1}\widetilde{\delta_1}\nu^{-s_1}\times \gamma_{\psi}^{-1}\widetilde{\delta_2}\nu^{-s_2}\times \cdots\times \gamma_{\psi}^{-1}\widetilde{\delta_{k-1}}\nu^{-s_{k-1}}\times \gamma_{\psi}^{-1}\nu^{-1/2}\rtimes \sigma.\]
This means that the standard representation for $\pi$ is irreducible.
This is not true, since by Proposition \ref{prop_reducibility_1_2}, $\gamma_{\psi}^{-1}\nu^{-1/2}\rtimes \sigma$ reduces.   Thus, $\theta_{-2}(\pi)=0,$ so $\pi$ appears at the earliest for $l=-4$ on the non-split tower, i.e., $l(\pi)\ge 2.$
\end{proof}

\subsection{The first lifts of a $\psi$--generic representation\\ $L(\gamma_{\psi}^{-1}\delta_1\nu^{s_1},\gamma_{\psi}^{-1}\delta_2\nu^{s_2},\ldots,\gamma_{\psi}^{-1}\delta_{k-1}\nu^{s_{k-1}},\gamma_{\psi}^{-1}St_{2s_k}\nu^{s_{k}};\sigma)$  with $l(\sigma)=0$}
\label{sec:exceptions1} 

\begin{prop} 
\label{first_lifts}
Let $\pi=L(\gamma_{\psi}^{-1}\delta_1\nu^{s_1},\dotsc,\gamma_{\psi}^{-1}\delta_{k-1}\nu^{s_{k-1}},\gamma_{\psi}^{-1}\nu^{1/2};\sigma)$ be a $\psi$--generic representation with $l(\sigma)=0.$ Then, on the split tower we have 
\[\theta_{2}(\pi)=L(\delta_1\nu^{s_1},\delta_2\nu^{s_2},\ldots,\delta_{k-1}\nu^{s_{k-1}};\Theta_{0}(\sigma))\]
and
\[\theta_{0}(\pi)=L(\delta_1\nu^{s_1},\delta_2\nu^{s_2},\ldots,\delta_{k-1}\nu^{s_{k-1}},\nu^{1/2};\Theta_{0}(\sigma)).\]
On the non-split tower we get
\[\theta_{-4}(\pi)=L(\delta_1\nu^{s_1},\delta_2\nu^{s_2},\ldots,\delta_{k-1}\nu^{s_{k-1}},\delta([\nu^{1/2},\nu^{3/2}]);\Theta_{-2}(\sigma)).\]
Again, we assume that we have permuted $\{\delta_1\nu^{s_1}, \delta_2\nu^{s_2}, \dotsc, \delta_{k-1}\nu^{s_{k-1}}\}$ among $\delta([\nu^{1/2},\nu^{3/2}])$ if necessary, so that they appear in the descending order.
\end{prop}
\begin{rem}
\begin{enumerate}[(i)]
\item We know that $\Theta_{-2}(\sigma)$ (the first lift on the non-split tower) is irreducible and tempered. This follows easily from Proposition 5.5 (2) of \cite{Atobe_Gan} (cf.~the comments at the end of Case 1, Section \ref{subs:std_mod_lifts}).
\item Any irreducible subquotient of $\Theta_{-4}(\sigma)$ is either tempered, or equal to $\theta_{-4}(\sigma)=L(\nu^{3/2},\Theta_{-2}(\sigma))$. This is Theorem 6.1 (iii) of \cite{Muic_theta_discrete_Israel}. The arguments are easily adjusted to the setting of $\widetilde{Sp(2n,F)}$.
\end{enumerate}
\end{rem}

\begin{proof} (of Proposition) The description of the first ($l=2$ and $l=0$) lifts on the going-down (which is, in this case, the split tower) follows directly from the discussion in the proof of Proposition \ref{prop:l(pi),l(sigma)=0}. As before, we get that on the non-split tower we have the following epimorphism (no exceptional situations):
\begin{equation}
\label{eq_first_occ}
\delta_1\nu^{s_1}\times \delta_2\nu^{s_2}\times \cdots\times \delta_{k-1}\nu^{s_{k-1}}\times \nu^{1/2}\rtimes \Theta_{-4}(\sigma)\twoheadrightarrow \theta_{-4}(\pi).
\end{equation}
We want to see which irreducible subquotient of $\Theta_{-4}(\sigma)$ participates in this epimorphism. By (ii) of the above Remark, we know that it is either some irreducible tempered representation, or it is $\theta_{-4}(\sigma)=L(\nu^{3/2},\Theta_{-2}(\sigma)).$ We now prove that the latter indeed occurs. Assume the opposite, i.e., that a certain irreducible tempered representation $\tau\le \Theta_{-4}(\sigma)$ participates. Then, we lift back to the metaplectic tower to get
\[\gamma_{\psi}^{-1}\delta_1\nu^{s_1}\times \gamma_{\psi}^{-1}\delta_2\nu^{s_2}\times \cdots\times \gamma_{\psi}^{-1}\delta_{k-1}\nu^{s_{k-1}}\times \gamma_{\psi}^{-1}\nu^{1/2}\rtimes \Theta_{4}(\tau)\twoheadrightarrow \pi,\]
if, among $\delta_i$'s there are no following representations: $\nu^{3/2}, St_{2}\nu^1,St_{3}\nu^{1/2}.$ This cannot happen, similarly to the proof of Proposition \ref{prop:l(pi),l(sigma)=0}. The uniqueness of Langlands classification now forces $\sigma\le \Theta_{4}(\tau).$ Specifically, $\Theta_{4}(\tau)\neq 0,$ which means, by Theorem 4.1. of \cite{Atobe_Gan}, that the L-parameter of $\tau$ contains $S_2\oplus S_4.$ Since all the subquotients of $\Theta_{4}(\tau)$ are tempered, by Lemma 6.4. of \cite{Atobe_Gan}, we get that $\theta_4(\tau)$ and $\sigma$ have the same $L$--parameter. But, since  L-parameter of $\theta_4(\tau)$ is $\phi_{\tau}-S_4,$ (cf.\cite{Atobe_Gan}, Theorem 4.3) we get that the $L$-parameter of $\sigma$ contains $S_2,$ a contradiction.

We conclude that $L(\nu^{3/2},\Theta_{-2}(\sigma))$ appears in \eqref{eq_first_occ} as a subquotient of $\Theta_{-4}(\sigma).$ Thus we either have
\begin{equation}
\label{pravi}
\delta_1\nu^{s_1}\times \delta_2\nu^{s_2}\times \cdots\times \delta_{k-1}\nu^{s_{k-1}}\times \delta([\nu^{1/2},\nu^{3/2}])\rtimes \Theta_{-2}(\sigma)\twoheadrightarrow \theta_{-4}(\pi)
\end{equation}
or
\[\delta_1\nu^{s_1}\times \delta_2\nu^{s_2}\times \cdots\times \delta_{k-1}\nu^{s_{k-1}}\times \nu^{3/2}\times \nu^{1/2}\rtimes \Theta_{-2}(\sigma)\twoheadrightarrow \theta_{-4}(\pi).\]
We now prove that the latter cannot happen.  Note that $\nu^{1/2}\rtimes \Theta_{-2}(\sigma)$  is irreducible by  Lemma \ref{irreducibility_1_2}, so we can exchange the exponent $\frac{1}{2}$ to $-\frac{1}{2}.$ Repeatedly using Lemma \ref{cor5.3} for $l=4$ we lift back the latter expression to the metaplectic tower (repeated application of \ref{cor5.3} is justified by the fact that the above representation has a unique irreducible quotient) we get, by Corollary 3.7 of \cite{bakic_generic}, either
\begin{equation}
\label{pos1}
\gamma_{\psi}^{-1}\delta_1\nu^{s_1}\times \gamma_{\psi}^{-1}\delta_2\nu^{s_2}\times \cdots\times \gamma_{\psi}^{-1}\delta_{k-1}\nu^{s_{k-1}}\times \gamma_{\psi}^{-1}\nu^{-\frac{1}{2}}\times  \gamma_{\psi}^{-1}\nu^{\frac{3}{2}}\rtimes \Theta_{4}(\Theta_{-2}(\sigma))\twoheadrightarrow \pi,
\end{equation}
or
\begin{equation}
\label{pos2}
\gamma_{\psi}^{-1}\delta_1\nu^{s_1}\times \gamma_{\psi}^{-1}\delta_2\nu^{s_2}\times \cdots\times \gamma_{\psi}^{-1}\delta_{k-1}\nu^{s_{k-1}}\times \gamma_{\psi}^{-1}\nu^{-1/2}\rtimes \Theta_{2}(\Theta_{-2}(\sigma))\twoheadrightarrow \pi.
\end{equation}
Note that each irreducible subquotient of $\Theta_{4}(\Theta_{-2}(\sigma))$ is tempered (see the above Remark and \cite[Proposition 5.5]{Atobe_Gan}),
so that \eqref{pos1} gives
\begin{equation}
\label{pos11}
\gamma_{\psi}^{-1}\delta_1\nu^{s_1}\times \gamma_{\psi}^{-1}\delta_2\nu^{s_2}\times \cdots\times \gamma_{\psi}^{-1}\delta_{k-1}\nu^{s_{k-1}}\times \gamma_{\psi}^{-1}\nu^{\frac{3}{2}}\times  \gamma_{\psi}^{-1}\nu^{-\frac{1}{2}}\rtimes \tau''\twoheadrightarrow \pi,
\end{equation}
for some irreducible tempered representation $\tau''\le \Theta_{4}(\Theta_{-2}(\sigma)).$ But, besides the Langlands subrepresentation, according to Lemma \ref{length_two} 1, any other subquotient of  $\gamma_{\psi}^{-1}\nu^{-\frac{1}{2}}\rtimes \tau''$  is tempered (note that the proof of that part of Lemma \ref{length_two} 1) does not depend on the properties of tempered $\sigma$).  In both of these cases we get a contradiction with the Langlands parametrization of $\pi.$
So, \eqref{pos2} has to occur. Note that $\Theta_{2}(\Theta_{-2}(\sigma))=\sigma$ by \cite[Proposition 5.4]{Atobe_Gan}). 
 But then \eqref{pos2} becomes exactly \eqref{eq_2}, which we saw in the proof of Proposition \ref{prop:l(pi),l(sigma)=0} could not happen. Thus, we have \eqref{pravi}. Note  that  the left-hand side there is isomorphic to a standard representation. Indeed, it may happen that for some $s_i,\;i=1,2,\ldots,k-1$ we have $s_i<1.$ But it means that either $\delta_i\nu^{s_i}$ and $ \delta([\nu^{1/2},\nu^{3/2}])$ are not linked, so we can change their position, or (the only possibility of linking with $s_1>0$)  $\delta_i\nu^{s_i}=\nu^{1/2}.$ This cannot happen, because this would mean that $\nu^{1/2}$ appears with multiplicity greater than one in the Langlands parameter of $\pi.$
\end{proof}
The case in which $\pi=L(\gamma_{\psi}^{-1}\delta_1\nu^{s_1},\dotsc,\gamma_{\psi}^{-1}\delta_{k-1}\nu^{s_{k-1}},\gamma_{\psi}^{-1}St_{2s_k}\nu^{s_{k}};\sigma)$ with  $s_k$ half integer greater or equal to $1$  is much easier because of Proposition $\ref{l(pi)}$ which guarantees that in that case $l(\pi)=l(\sigma)=0.$ Then, the split tower is  the going-down tower, the non-split tower is  the going up tower.
The description of the first lift on  the going-down tower can be read off  from the results of \cite{Gan_Savin_Metaplectic_2012}, and on the non-split tower we just follow the discussion in Proposition \ref{prop:l(pi),l(sigma)=0} when we know that $\Theta_{-2}(\pi)$ is non-zero on the non-split tower.
\begin{cor}
\label{s_k_1}
Let $\pi=L(\gamma_{\psi}^{-1}\delta_1\nu^{s_1},\dotsc,\gamma_{\psi}^{-1}\delta_{k-1}\nu^{s_{k-1}},\gamma_{\psi}^{-1}St_{2s_k}\nu^{s_{k}};\sigma)$ be a $\psi$--generic representation with $s_k\ge 1$ and $l(\sigma)=0.$ Then, on the split tower the first non-zero lift is given by
\[\theta_{0}(\pi)=L(\delta_1\nu^{s_1},\delta_2\nu^{s_2},\ldots,\delta_{k-1}\nu^{s_{k-1}},St_{2s_k}\nu^{s_{k}};\Theta_{0}(\sigma)).\]
On the non-split tower we get
\[\theta_{-2}(\pi)=L(\delta_1\nu^{s_1},\delta_2\nu^{s_2},\ldots,\delta_{k-1}\nu^{s_{k-1}},St_{2s_k}\nu^{s_{k}};\Theta_{-2}(\sigma)).\]
Here $\Theta_{-2}(\sigma)$ is irreducible and tempered first lift on the non-split tower.
\end{cor}
\begin{rem} Note that, for  $\psi$--generic 
\[\pi=L(\gamma_{\psi}^{-1}\delta_1\nu^{s_1},\gamma_{\psi}^{-1}\delta_2\nu^{s_2},\ldots,\gamma_{\psi}^{-1}\delta_{k-1}\nu^{s_{k-1}},\gamma_{\psi}^{-1}St_{2s_k}\nu^{s_{k}};\sigma),\]
 when both $\theta_l(\sigma)$ and $\theta_l(\pi)$ are defined  in Proposition \ref{first_lifts} and Corollary \ref{s_k_1}, we have 
\[\delta_1\nu^{s_1}\times \delta_2\nu^{s_2}\times \cdots\times \delta_{k-1}\nu^{s_{k-1}}\times St_{2s_k}\nu^{s_{k}}\rtimes \theta_l(\sigma)\twoheadrightarrow \theta_{l}(\pi).\]
\end{rem}
\subsection{The higher lifts  of a $\psi$--generic representation\\ $L(\gamma_{\psi}^{-1}\delta_1\nu^{s_1},\allowbreak \gamma_{\psi}^{-1}\delta_2\nu^{s_2},\ldots,\gamma_{\psi}^{-1}\delta_{k-1}\nu^{s_{k-1}},\gamma_{\psi}^{-1}St_{2s_k}\nu^{s_{k}};\sigma)$ with $l(\sigma)=0$}
\label{sec:exceptions2} 
 We can  directly check that Lemma \ref{lem:main} also works in this setting. Thus, let $\delta_1',\ldots,\delta_l'$ and $\pi_{00}$ be the tempered support of $\sigma,$ so that we have
 \[\gamma_{\psi}^{-1}\delta_1'\times \gamma_{\psi}^{-1}\delta_2'\times \cdots \gamma_{\psi}^{-1}\delta_l'\rtimes \pi_{00}\twoheadrightarrow \sigma;\]
 we denote $\Delta=\delta_1'\times \delta_2'\times \cdots \times \delta_l'.$
Lemma \ref{lem:main} then gives, for $l\ge 0$ such that $\theta_{-l}(\pi)\neq 0,$ for both towers,
 \begin{equation}
 \label{prop_4_7}
 \delta_1\nu^{s_1}\times \delta_2\nu^{s_2}\times \cdots\times \delta_{k-1}\nu^{s_{k-1}}\times St_{2s_k}\nu^{s_{k}}\times \Delta\rtimes \theta_{-l}(\pi_{00})\twoheadrightarrow \theta_{-l}(\pi).
 \end{equation}
We will also use (i) and (iii) of Proposition \ref{prop:lifts_temp_Mp}
 

We start with the description of the higher lifts on the going-down tower.
 \begin{prop}
 \label{higher_lifts_l(sigma)=0_down}
 Let $\sigma$ be a $\psi$--generic tempered representation with $l(\sigma)=0.$ Then, for $\pi=L(\gamma_{\psi}^{-1}\delta_1\nu^{s_1},\gamma_{\psi}^{-1}\delta_2\nu^{s_2},\ldots,\gamma_{\psi}^{-1}\delta_{k-1}\nu^{s_{k-1}},\gamma_{\psi}^{-1}St_{2s_k}\nu^{s_{k}};\sigma)$ with $s_k\ge \frac{1}{2}$ for the lifts $l>0$ on the going-down tower we have
 \[\theta_{-l}(\pi)=L(\delta_1\nu^{s_1},\delta_2\nu^{s_2},\ldots,\delta_{k-1}\nu^{s_{k-1}},St_{2s_k}\nu^{s_{k}},\nu^{\frac{l-1}{2}}, \nu^{\frac{l-3}{2}}, \ldots, \nu^{\frac{1}{2}}; \theta_0(\sigma)).\]
 Here again we assume that $\nu^{\frac{l-1}{2}}, \nu^{\frac{l-3}{2}}, \ldots, \nu^{\frac{1}{2}}$ are distributed in descending order among
 $\delta_1\nu^{s_1},\delta_2\nu^{s_2},\ldots,\delta_{k-1}\nu^{s_{k-1}},St_{2s_k}\nu^{s_{k}}.$ 
 \end{prop} 
 \begin{proof}
We can apply Lemmas \ref{lem:unique_quotient}, \ref{lemma:really_subq} and \ref{lem:really_small_theta}. Note that we can avoid using the irreducibility of the standard module for $\pi$ (which can be reducible in our case) in these lemmas.
 \end{proof}
We now address the lifts on the going-up tower. From equation \eqref{prop_4_7} and Proposition \ref{prop:lifts_temp_Mp}, we get that there is an epimorphism
 \[\delta_1\nu^{s_1}\times \delta_2\nu^{s_2}\times \cdots\times \delta_{k-1}\nu^{s_{k-1}}\times St_{2s_k}\nu^{s_{k}}\times \Delta\times \zeta(\frac{3}{2}, \frac{l-1}{2})\rtimes \theta_{-2}(\pi_{00})\twoheadrightarrow \theta_{-l}(\pi).\]
 Since there is no summand of the form $S_2$ in the L-parameter of $\sigma,$ and consequently, no $\delta_i'$ is equal to $\delta([\nu^{-1/2},\nu^{1/2}]),$ we can jump by isomorphisms with $\zeta(\frac{3}{2}, \frac{l-1}{2})$ over whole $\Delta$ (cf.~Lemma \ref{lem:zelevinsky}). We conclude that there exists an irreducible tempered subquotient $\tau$ of the representation $\Delta \rtimes \theta_{-2}(\pi_{00})$
 such that
 \begin{equation}
 \label{zadnji_slucaj}
 \delta_1\nu^{s_1}\times \delta_2\nu^{s_2}\times \cdots\times \delta_{k-1}\nu^{s_{k-1}}\times St_{2s_k}\nu^{s_{k}}\times \zeta(\frac{3}{2}, \frac{l-1}{2})\rtimes\tau\twoheadrightarrow \theta_{-l}(\pi).
 \end{equation}
 As before, using Lemma \ref{lemma:really_subq} and the arguments of Section \ref{subs:higher}, Case 3 (including Lemmas 7.11 and 7.12 of \cite{bakic_generic}), we get that necessarily $\tau=\theta_{-2}(\sigma).$
We now want to explicitly describe the standard module of $\theta_{-l}(\pi),$ i.e., how $\nu^{3/2},\ldots,\nu^{\frac{l-1}{2}}$ distribute among $\delta_1\nu^{s_1},\delta_2\nu^{s_2},\ldots,\delta_{k-1}\nu^{s_{k-1}}, St_{2s_k}\nu^{s_{k}}.$

To this end, we invoke Lemma \ref{lem:tricky_condition} again.
In this case we have $\zeta(c,d) = \zeta(\frac{3}{2}, \frac{l-1}{2})$ so the only problematic situation arises if $\delta([\nu^a,\nu^b]) = \delta([\nu^{a},\nu^{1/2}])$,
 i.e., if $St_{2s_k}\nu^{s_{k}}$ appears with $s_k=\frac{1}{2}.$ So we now concentrate on that case (recall that $\nu^{1/2}$ can only appear once in the Langlands parameter of $\pi$).

We deduce that we either have $\delta([\nu^{1/2},\nu^{3/2}])$ appearing in the Langlands parameter of $\theta_{-l}(\pi),$ as in the Langlands parameter of $\theta_{-4}(\pi)$ (cf.~Proposition \ref{first_lifts}) and we rearrange $\nu^{5/2},\ldots,\nu^{\frac{l-1}{2}}$
 among $\delta_1\nu^{s_1}\times \delta_2\nu^{s_2}\times \cdots\times \delta_{k-1}\nu^{s_{k-1}}$ since there are no more problematic situations, as we explained by Lemma  \ref{lem:tricky_condition}, or we have $\nu^{1/2}$ appearing in the L-parameter and and we rearange $\nu^{3/2},\ldots,\nu^{\frac{l-1}{2}}$
 among $\delta_1\nu^{s_1}\times \delta_2\nu^{s_2}\times \cdots\times \delta_{k-1}\nu^{s_{k-1}}.$ We now show that the second possibility cannot occur. 

 \noindent We mimic the proof of Proposition \ref{first_lifts}, so we need to get back to the metaplectic tower. To be able to do that and not to land in the exceptional situation of Lemma \ref{cor5.3}, for each $\delta_i\nu^{s_i},\;i=1,2,\ldots,k-1$ of the form $St_{t}\nu^{\frac{l-t}{2}},$ we interchange it with $\delta_i\nu^{-s_i}$ since $\delta_i\nu^{s_i}\rtimes \sigma$ is irreducible.
 We write the product of all such $\delta_i \nu^{s_i}$'s by $T_1,$ and product of the rest of $\delta_i\nu^{s_i},\;i=1,2,\ldots,k-1$ outside of $T_1$ by $T_2.$
 Then, from \eqref{zadnji_slucaj}, we get
 \[\widetilde{T_1}\times T_2\times \nu^{1/2}\times \zeta(\frac{3}{2}, \frac{l-1}{2})\rtimes\theta_{-2}(\sigma)\twoheadrightarrow \theta_{-l}(\pi).\]
 Indeed, note that $St_{t}\nu^{-\frac{l-t}{2}}\neq St_{t}\nu^{\frac{-l-t}{2}},$ so this goes through without exceptions.
 If we assume that the second possibility from the above occurs, we would have, analogously as in the proof of Proposition \ref{first_lifts},
 \[
 \label{eq:T-ovi}\tag{\textasteriskcentered}
 \widetilde{T_1}\times T_2\times \zeta(\frac{3}{2}, \frac{l-1}{2})\times \nu^{-1/2}\rtimes\theta_{-2}(\sigma)\twoheadrightarrow \theta_{-l}(\pi).\]
 Now we go back to the metaplectic tower by calculating $\Theta_{l}$ and we get either
\[\gamma_{\psi}^{-1}\widetilde{T_1}\times \gamma_{\psi}^{-1}T_2\times \gamma_{\psi}^{-1}\nu^{-1/2}\times \gamma_{\psi}^{-1}\zeta(\frac{3}{2}, \frac{l-1}{2})\rtimes \tau\twoheadrightarrow \pi,\]
or 
\[\gamma_{\psi}^{-1}\widetilde{T_1}\times \gamma_{\psi}^{-1}T_2\times \gamma_{\psi}^{-1}\nu^{-1/2}\rtimes \sigma\twoheadrightarrow \pi,\]
for some tempered representation $\tau.$ Here $\tau\le \Theta_{l}(\theta_{-2}(\sigma)).$ This follows easily from the calculations of the isotypic components of the twists of the trivial representation in the Kudla's filtration, similarly to the calculations on the isotypic components of the essentially square-integrable representations. Note that the left-hand side of \eqref{eq:T-ovi} is easily shown to have a unique irreducible quotient, which justifies the repeated application of Lemma \ref{cor5.3}. Since $\pi$ is $\psi$-generic, from the multiplicity one for the generic representations, it follows that the irreducible $\psi$--generic subquotient of 
$\gamma_{\psi}^{-1}\zeta(\frac{3}{2}, \frac{l-1}{2})\rtimes \tau$ which participates in the epimorphism onto $\pi$ in the former case is necessarily just $\sigma.$ Now, in both cases we get
\[\gamma_{\psi}^{-1}\widetilde{T_1}\times \gamma_{\psi}^{-1}T_2\times \gamma_{\psi}^{-1}\nu^{-1/2}\rtimes \sigma\twoheadrightarrow \pi.\]
Now, from Theorem \ref{generic_reps_metaplectic}, using intertwining operators in a usual way, we get that we can transform the left-hand side above in
\[\gamma_{\psi}^{-1}T_1\times \gamma_{\psi}^{-1}T_2\times \gamma_{\psi}^{-1}\nu^{-1/2}\rtimes \sigma\twoheadrightarrow \pi.\]
But this leads to a contradiction in the same way as in the proof of Proposition \ref{first_lifts}.
We have proved the following 

\begin{prop}
 \label{higher_lifts_l(sigma)=0_up}
 Let $\sigma$ be a $\psi$--generic tempered representation with $l(\sigma)=0.$ Then, for  a $\psi$-generic representation $\pi=L(\gamma_{\psi}^{-1}\delta_1\nu^{s_1},\dotsc,\gamma_{\psi}^{-1}\delta_{k-1}\nu^{s_{k-1}},\allowbreak \gamma_{\psi}^{-1}St_{2s_k}\nu^{s_{k}};\sigma)$ with $s_k\ge \frac{1}{2}$ for the lifts  on the going-up tower we have
 \begin{itemize}
 \item[(i)] If $s_k\ge 1$ for $l\ge 4$ then
 \[
 \theta_{-l}(\pi)=L(\delta_1\nu^{s_1}\dotsc,\delta_{k-1}\nu^{s_{k-1}},St_{2s_k}\nu^{s_{k}},\nu^{\frac{l-1}{2}}, \nu^{\frac{l-3}{2}}, \ldots, \nu^{\frac{3}{2}}; \Theta_{-2}(\sigma)).\]
 Here again we assume that $\nu^{\frac{l-1}{2}}, \nu^{\frac{l-3}{2}}, \ldots, \nu^{\frac{3}{2}}$ are distributed in the descending order among
 $\delta_1\nu^{s_1},\delta_2\nu^{s_2},\ldots,\delta_{k-1}\nu^{s_{k-1}},St_{2s_k}\nu^{s_{k}}.$ 
 \item [(ii)] If $s_k=1/2$ for $l\ge 6$ then
 \[\theta_{-l}(\pi)=L(\delta_1\nu^{s_1},\dotsc,\delta_{k-1}\nu^{s_{k-1}},\nu^{\frac{l-1}{2}}, \nu^{\frac{l-3}{2}}, \ldots, \nu^{\frac{5}{2}}, \delta([\nu^{1/2},\nu^{3/2}]);\Theta_{-2}(\sigma)).\]
 \end{itemize}
 \end{prop} 

 Now we just comment the case of the lifts on the pair of the orthogonal towers with $\chi_V\neq 1.$ Then, to be in accordance with the notation in \cite{Atobe_Gan}, especially suited for the application of Corollary 5.3. there, we write  $\pi=L(\gamma_{\psi}^{-1}\delta_1\nu^{s_1},\gamma_{\psi}^{-1}\delta_2\nu^{s_2},\ldots,\gamma_{\psi}^{-1}\delta_{k-1}\nu^{s_{k-1}},\gamma_{\psi}^{-1}St_{2s_k}\nu^{s_{k}};\sigma)$  above as  $\pi=L(\gamma_{\psi}^{-1}\chi_V\delta_1'\nu^{s_1},\gamma_{\psi}^{-1}\chi_V\delta_2'\nu^{s_2},\ldots,\gamma_{\psi}^{-1}\chi_V\delta_{k-1}'\nu^{s_{k-1}},\gamma_{\psi}^{-1}(\chi_V)^2St_{2s_k}\nu^{s_{k}};\sigma)$ with $s_k\ge \frac{1}{2}.$ Note that in this situation we always have  that $l(\sigma),$ which we now  denote by $l_{\chi_V}(\sigma)$ to emphasize the towers where we lift, is equal to zero (Theorem 4.1. of \cite{Atobe_Gan}). Now, we can repeat the arguments of Proposition 3.1 verbatim. Note that the exception in the case $l_{\chi_V}(\sigma)$ occurs if $\chi_V\nu^{1/2}$ appears in the L-parameter of $\pi,$ (i.e. some $\delta_i'\nu^{s_i}=\nu^{1/2}$). But, since $\chi_V\nu^{1/2}\rtimes \sigma$ is irreducible by Theorem \ref{generic_reps_metaplectic}, we can exchange $\chi_V\nu^{1/2}\rtimes \sigma$ by $\chi_V\nu^{-1/2}\rtimes \sigma$ and avoid that exception. Thus, $l_{\chi_V}(\pi)=0.$ We know that the first  full-lifts of $\sigma$ on both $\chi_V$--towers are irreducible and tempered, so we have the analog of Corollary \ref{s_k_1}: 
\begin{prop}
\label{non_trivial_charcter_first}
Let $\chi_V$ be a non-trivial quadratic character of $F^*.$ Let 
\[\pi=L(\gamma_{\psi}^{-1}\chi_V\delta_1'\nu^{s_1},\gamma_{\psi}^{-1}\chi_V\delta_2'\nu^{s_2},\ldots,\gamma_{\psi}^{-1}\chi_V\delta_{k-1}'\nu^{s_{k-1}},\gamma_{\psi}^{-1}St_{2s_k}\nu^{s_{k}};\sigma)\]
 be a $\psi$--generic representation. Then, on the split $\chi_V$--tower the first non-zero lift is given by
\[\theta_{0}(\pi)=L(\delta_1\nu^{s_1},\delta_2\nu^{s_2},\ldots,\delta_{k-1}\nu^{s_{k-1}},\chi_VSt_{2s_k}\nu^{s_{k}};\Theta_{0}(\sigma)).\]
On the non-split tower we get
\[\theta_{-2}(\pi)=L(\delta_1\nu^{s_1},\delta_2\nu^{s_2},\ldots,\delta_{k-1}\nu^{s_{k-1}},\chi_VSt_{2s_k}\nu^{s_{k}};\Theta_{-2}(\sigma)).\]
Here $\Theta_{-2}(\sigma)$ is irreducible and tempered first lift on the non-split $\chi_V$--tower.
\end{prop}
As for the higher lifts, we can repeat the argument in the case when $\pi$ is irreducible standard representation almost verbatim. Indeed, for the  higher lifts on the level $-l$ one is interested how $\zeta(\frac{1}{2},\frac{l-1}{2})$ on the split tower and $\zeta(\frac{3}{2},\frac{l-1}{2})$ on the non--split tower are arranged among $\delta_1'\nu^{s_1},\dotsc,\delta_{k-1}'\nu^{s_{k-1}},\allowbreak \chi_VSt_{2s_k}\nu^{s_{k}}.$ Here the possible obstacles are, as we saw above, if  some of $\delta_i'\nu^{s_i}$ are of the certain form and necessarily coming from the trivial character of $F^*$ in the cuspidal support. This means that  $\chi_VSt_{2s_k}\nu^{s_{k}}$ cannot complicate things more; moreover, when observing interactions of $\zeta(\frac{1}{2},\frac{l-1}{2})$ with $\delta_i'\nu^{s_i}$ which can cause complications, we can use the same trick used in the  irreducible standard representation case. Namely, in that case $\gamma_{\psi}^{-1}\chi_V\delta_i'\nu^{s_i}\rtimes \sigma$ is irreducible, so we can change the exponent into $-s_i$ if it suits us. To conclude, we have
\begin{prop}
\label{non_trivial_charcter_higher}
Let $\chi_V$ be a non-trivial quadratic character of $F^*.$ Let 
\[\pi=L(\gamma_{\psi}^{-1}\chi_V\delta_1'\nu^{s_1},\gamma_{\psi}^{-1}\chi_V\delta_2'\nu^{s_2},\ldots,\gamma_{\psi}^{-1}\chi_V\delta_{k-1}'\nu^{s_{k-1}},\gamma_{\psi}^{-1}St_{2s_k}\nu^{s_{k}};\sigma)\]
 be a $\psi$--generic representation. 
\begin{enumerate}
\item On the split $\chi_V$--tower and $l>0$ the lifts are given by
 \[\theta_{-l}(\pi)=L(\delta_1'\nu^{s_1},\dotsc,\delta_{k-1}'\nu^{s_{k-1}},\chi_VSt_{2s_k}\nu^{s_{k}},\nu^{\frac{l-1}{2}}, \nu^{\frac{l-3}{2}}, \ldots, \nu^{\frac{1}{2}}; \theta_0(\sigma)).\]
 \item On the non-split $\chi_V$--tower and $l\ge 4$ the lifts are given by
 \[\theta_{-l}(\pi)=L(\delta_1'\nu^{s_1},\dotsc,\delta_{k-1}'\nu^{s_{k-1}},\chi_VSt_{2s_k}\nu^{s_{k}},\nu^{\frac{l-1}{2}}, \nu^{\frac{l-3}{2}}, \ldots, \nu^{\frac{3}{2}}; \theta_{-2}(\sigma)).\]
 \end{enumerate}
\end{prop}

\section*{Appendix A: Proof of Theorem \ref{standard_module_reducibility}.}
Before we begin the proof we remind the reader of our notation: if $\pi$ is an irreducible representation, we let $\phi_\pi$ denote its $L$-parameter.
\begin{proof}
\noindent We first prove reducibility. Recall that $a_0= \min\{a: a \text{ is even and }S_a \hookrightarrow \phi_\sigma\}$. If $4s_i < a_0$, then $\phi_{\Theta_0(\sigma)}$ (which is the same as $\phi_\sigma$) does not contain $S_{4s_i}$. Therefore, $\delta([\nu^{-(2s_i-\frac{1}{2})},\nu^{2s_i-\frac{1}{2}}]) \rtimes \Theta_0(\sigma)$ reduces. We now imitate the proof of Proposition \ref{prop_reducibility_1_2}.

We thus assume that $\sigma$ is square-integrable, so that the long intertwining operator
\[
\begin{split}
A(w_0,s): \gamma_{\psi}^{-1}\delta([\nu^{-(2s_i-\frac{1}{2})},&\nu^{2s_i-\frac{1}{2}}])\nu^s \rtimes \Theta_0(\sigma) \\ &\to \gamma_{\psi}^{-1}\delta([\nu^{-(2s_i-\frac{1}{2})},\nu^{2s_i-\frac{1}{2}}])\nu^{-s} \rtimes \Theta_0(\sigma)
\end{split}
\]
is holomorphic at $s=0$. As before, $A(w_0,s)$ is the restriction of an intertwining operator $B(w_0,s)$ which is composed of the following intertwining operators:
\[
\begin{split}
    \gamma_{\psi}^{-1}\delta([\nu^{\frac{1}{2}},&\nu^{(2s_i-\frac{1}{2})}])\nu^s \times   \gamma_{\psi}^{-1}\delta([\nu^{-(2s_i-\frac{1}{2})},\nu^{-\frac{1}{2}}])\nu^s \rtimes \Theta_0(\sigma)\\ &\xrightarrow{T_1(s)}
       \gamma_{\psi}^{-1}\delta([\nu^{\frac{1}{2}},\nu^{2s_i-\frac{1}{2}}])\nu^s \times   \gamma_{\psi}^{-1}\delta([\nu^{\frac{1}{2}},\nu^{2s_i-\frac{1}{2}}])\nu^{-s} \rtimes \Theta_0(\sigma)\\
       &\xrightarrow{T_2(s)} \gamma_{\psi}^{-1}\delta([\nu^{\frac{1}{2}},\nu^{2s_i-\frac{1}{2}}])\nu^{-s} \times \gamma_{\psi}^{-1}\delta([\nu^{\frac{1}{2}},\nu^{2s_i-\frac{1}{2}}])\nu^s \rtimes \Theta_0(\sigma) \\
       &\xrightarrow{T_3(s)}
       \gamma_{\psi}^{-1}\delta([\nu^{\frac{1}{2}},\nu^{2s_i-\frac{1}{2}}])\nu^{-s} \times   \gamma_{\psi}^{-1}\delta([\nu^{-(2s_i-\frac{1}{2})},\nu^{-\frac{1}{2}}])\nu^{-s} \rtimes \Theta_0(\sigma).
\end{split}
\]
Now assume $\gamma_{\psi}^{-1}\delta([\nu^{\frac{1}{2}},\nu^{2s_i-\frac{1}{2}}])\rtimes \Theta_0(\sigma)$ is irreducible. This implies that (after multiplying by a certain positive power of $s$, if necessary) the operators $T_1(s)$, $T_2(s)$ and $T_3(s)$ are all holomorphic isomorphisms. We thus have a $k \geqslant 0$ such that $\lim_{s \to 0} s^k B(w_0,s)$ is a holomorphic isomorphism. Furthermore, we know that $k$ is strictly positive since $T_2(s)$ has a pole at $s=0$. This implies that $A(w_0,s)$ has a pole for $s=0$, which contradicts the fact that $A(w_0,s)$ is holomorphic at $s=0$. Therefore, $\gamma_{\psi}^{-1}\delta([\nu^{\frac{1}{2}},\nu^{2s_i-\frac{1}{2}}])\rtimes \Theta_0(\sigma)$ must be reducible.

If $\sigma$ is tempered, but not in discrete series, we complete the proof just like in Prop.~\ref{prop_reducibility_1_2}.

\bigskip

\noindent We now turn to the proof of irreducibility when $4s_i=a_0$. First we need the following


\begin{claim}
Apart from the Langlands quotient, all the irreducible subquotients of $\gamma_\psi^{-1} St_{2s_i}\nu^{s_i} \rtimes \sigma$ are tempered.
\end{claim}
\begin{proof}
If $\sigma$ is in discrete series, this is Corollary 2.1 of \cite{M1}, which transfers easily to the metaplectic case. When $\sigma$ is tempered, we can embed it in its tempered support:
\[
 \sigma \hookrightarrow \gamma_\psi^{-1}\delta_1' \times \dotsm \times \gamma_\psi^{-1}\delta_k' \rtimes \sigma_0.
 \]
We then have
\begin{align*}
\gamma_\psi^{-1}\delta \times  \gamma_\psi^{-1}(\delta_1 \times \dotsm\times \delta_k) \rtimes \sigma_0 &\xrightarrow{T_1} \gamma_\psi^{-1}(\delta_1 \times \dotsm\times \delta_k) \times  \gamma_\psi^{-1}\delta\rtimes \sigma_0\\
&\xrightarrow{T_2} \gamma_\psi^{-1}(\delta_1 \times \dotsm\times \delta_k) \times  \gamma_\psi^{-1}\delta^\vee \rtimes \sigma_0 \\
&\xrightarrow{T_3} \gamma_\psi^{-1}\delta^\vee \times  \gamma_\psi^{-1}(\delta_1 \times \dotsm\times \delta_k) \rtimes \sigma_0.
\end{align*}
Restricting $T_3 \circ T_2 \circ T_1$ to the subrepresentation $\gamma_\psi^{-1}\delta \rtimes \sigma$ we get
\[
T:  \gamma_\psi^{-1}\delta \rtimes \sigma \to \gamma_\psi^{-1}\delta^\vee \rtimes \sigma.
\]
The image of $T$ is exactly the Langlands quotient of $\gamma_\psi^{-1}\delta \rtimes \sigma$. We want to show that the kernel consists only of tempered subquotients. We first note that $T_1$ and $T_3$ are isomorphisms: since $4s_i$ is minimal among $a$ such that $S_a \hookrightarrow \phi$, none of the segments which define $\delta_1, \dotsc, \delta_k$ are linked to the segment which defines $\delta$ (or $\delta^\vee$). On the other hand, $T_2$ need not be an isomorphism, but it is induced from $T_2': \gamma_\psi^{-1}\delta \rtimes \sigma_0 \to \gamma_\psi^{-1}\delta^\vee \rtimes \sigma_0$. By the part of this claim which concerns discrete series representations, $\ker T_2'$ only has tempered subquotients. From here we easily deduce the same for $T_2$. Thus, any subquotient of $\ker T$ is in fact contained in $\ker T_2$ and is therefore tempered. This proves the claim.
\end{proof}

\bigskip

\noindent We are now ready to prove irreducibility.
The $l=0$ lift of $\sigma$ to the split tower is $\Theta_0(\sigma)$. This is a generic representation of the orthogonal group, so we can use the results of \cite{Hanzer_injectivity}. By Propositions 4.7 and 5.1 of \cite{Hanzer_injectivity}, we know that there is a irreducible tempered generic representation $\sigma_1$ such that $\Theta_0(\sigma) \hookrightarrow \delta([\nu^{\frac{1}{2}},\nu^{2s_i-\frac{1}{2}}]) \rtimes \sigma_1$. Therefore, we have $\delta([\nu^{-(2s_i-\frac{1}{2})},\nu^{-\frac{1}{2}}]) \rtimes \sigma_1 \twoheadrightarrow \Theta_0(\sigma)$. We may now use a slight generalization of Lemma \ref{cor5.3} (cf.~Corollary 3.7 of \cite{bakic_generic}) to obtain either 
\[
\gamma_\psi^{-1}\delta([\nu^{-(2s_i-\frac{1}{2})},\nu^{-\frac{3}{2}}]) \rtimes \Theta_{-2}(\sigma_1) \twoheadrightarrow \sigma
\]
or
\[
\gamma_\psi^{-1}\delta([\nu^{-(2s_i-\frac{1}{2})},\nu^{-\frac{1}{2}}]) \rtimes \Theta_0(\sigma_1) \twoheadrightarrow \sigma.
\]
We show that the latter cannot be true. To this end, we use Theorem 4.1 (2) of \cite{Atobe_Gan}. It shows that we have $m^{\text{down}}(\sigma_{1{|SO}}) = m^\alpha(\sigma_{1{|SO}})$, with $\alpha = \eta(z_{\phi_{\sigma_1}})\cdot \epsilon({\phi_{\sigma_1}}) = \epsilon({\phi_{\sigma_1}})$ (recall that $\eta$ is trivial since $\sigma_1$ is generic). Similarly, $m^{\text{down}}(\Theta_0(\sigma)_{{|SO}}) = m^{\alpha'}(\Theta_0(\sigma)_{{|SO}})$, with $\alpha' = \epsilon({\phi_{\Theta_0(\sigma)_{{|SO}}}})$. Now it is easy to see that $\sigma_1$ and $\Theta_0(\sigma)$ have the same central character. On the other hand, one shows that $\epsilon({\phi_{\sigma_1}}) = -\epsilon({\phi_{\Theta_0(\sigma)_{{|SO}}}})$. Since $\Theta_0(\sigma)$ obviously has a non-zero lift on level $l=0$ (it is equal to $\sigma$), this implies that $\Theta_0(\sigma_1) = 0$ (and $\Theta_0(\sigma_1 \otimes \det) \neq 0$).

\noindent We have thus shown that
\[
\gamma_\psi^{-1}\delta([\nu^{-(2s_i-\frac{1}{2})},\nu^{-\frac{3}{2}}]) \rtimes \Theta_{-2}(\sigma_1) \twoheadrightarrow \sigma.
\]
Moreover, we now know that $\Theta_{-2}(\sigma_1)$ is the first non-zero lift of $\sigma_1$, so it is tempered by Theorem 4.5 (1) of \cite{Atobe_Gan}. It is also generic, as shown by the above map (and the hereditary property).

Now assume $\gamma_\psi^{-1}\delta([\nu^{\frac{1}{2}},\nu^{2s_i-\frac{1}{2}}]) \rtimes \sigma$ reduces. We have shown that any irreducible subquotient of $\gamma_\psi^{-1}\delta([\nu^{\frac{1}{2}},\nu^{2s_i-\frac{1}{2}}]) \rtimes \sigma$ which is not the Langlands quotient is necessarily tempered. Thus let $T$ be an irreducible tempered representation such that
\[
T \hookrightarrow \gamma_\psi^{-1}\delta([\nu^{\frac{1}{2}},\nu^{2s_i-\frac{1}{2}}]) \rtimes \sigma,
\]
By the above discussion, this implies
\begin{align*}
&T \hookrightarrow \gamma_\psi^{-1}\delta([\nu^{\frac{1}{2}},\nu^{2s_i-\frac{1}{2}}]) \times \gamma_\psi^{-1}\delta([\nu^{\frac{3}{2}},\nu^{2s_i-\frac{1}{2}}]) \rtimes \Theta_{-2}(\sigma_1)\\
\Rightarrow \quad  &T \hookrightarrow \gamma_\psi^{-1}\delta([\nu^{\frac{3}{2}},\nu^{2s_i-\frac{1}{2}}]) \times \gamma_\psi^{-1}\delta([\nu^{\frac{1}{2}},\nu^{2s_i-\frac{1}{2}}]) \rtimes \Theta_{-2}(\sigma_1)\\
\Rightarrow \quad &T \hookrightarrow \gamma_\psi^{-1}\delta([\nu^{\frac{3}{2}},\nu^{2s_i-\frac{1}{2}}]) \times \gamma_\psi^{-1}\delta([\nu^{\frac{3}{2}},\nu^{2s_i-\frac{1}{2}}]) \times  \gamma_\psi^{-1}\nu^\frac{1}{2}\rtimes \Theta_{-2}(\sigma_1).
\end{align*}
Now $\gamma_\psi^{-1}\nu^\frac{1}{2}\rtimes \Theta_{-2}(\sigma_1) = \gamma_\psi^{-1}\nu^{-\frac{1}{2}}\rtimes \Theta_{-2}(\sigma_1)$ by Proposition \ref{prop_reducibility_1_2}, so that
\begin{align*}
&T \hookrightarrow \gamma_\psi^{-1}\delta([\nu^{\frac{3}{2}},\nu^{2s_i-\frac{1}{2}}]) \times \gamma_\psi^{-1}\delta([\nu^{\frac{3}{2}},\nu^{2s_i-\frac{1}{2}}]) \times  \gamma_\psi^{-1}\nu^{-\frac{1}{2}}\rtimes \Theta_{-2}(\sigma_1)\\
\text{i.e.} \quad &T \hookrightarrow \gamma_\psi^{-1}\nu^{-\frac{1}{2}} \times \gamma_\psi^{-1}\delta([\nu^{\frac{3}{2}},\nu^{2s_i-\frac{1}{2}}]) \times \gamma_\psi^{-1}\delta([\nu^{\frac{3}{2}},\nu^{2s_i-\frac{1}{2}}]) \rtimes \Theta_{-2}(\sigma_1).
\end{align*}
However, this contradicts Casselman's criterion. We have thus shown that it is impossible for $\gamma_\psi^{-1}\delta([\nu^{\frac{1}{2}},\nu^{2s_i-\frac{1}{2}}]) \rtimes \sigma$ to reduce. This completes our proof.
\end{proof}

 \bibliographystyle{siam}
\bibliography{deg_eisen_Sp4_27_07}
{}
\bigskip
\end{document}